\documentclass{amsart}

\usepackage{comment}
\usepackage[noadjust]{cite}
\usepackage{hyperref}
\usepackage[capitalise,noabbrev]{cleveref}
\usepackage{amsmath}
\usepackage{enumitem}
\usepackage{lmodern}

\usepackage{todonotes}
\usepackage[margin=1.115in]{geometry}

\usepackage[labelformat=simple]{subcaption}
\usepackage{tikz}
\usepackage{tkz-graph}
\tikzstyle{VertexStyle} = [shape = circle, draw, fill=white,very thick]
\tikzset{pre/.style={-}}    
\tikzstyle{every node}=[circle, inner sep=0pt, minimum width=4pt,very thick]
\makeatletter
\newcommand{\myitem}[1]{%
\item[#1]\protected@edef\@currentlabel{#1}%
}
\makeatother

\newenvironment{reptheorem}[1]
  {\renewcommand\thetheorem{\ref*{#1}}\theorem}
  {\endtheorem}

\newcommand{\ba}{\backslash}
\DeclareMathOperator{\Ev}{Ev}
\newcommand{\symdiff}{\mathrel{\triangle}}
\newcommand{\hajos}{\mathbin{\triangle}}

\newtheorem{theorem}{Theorem}[section]
\newtheorem{lemma}[theorem]{Lemma}
\newtheorem{proposition}[theorem]{Proposition}

\newtheorem{corollary}[theorem]{Corollary}
\newtheorem{claim}{Claim}[theorem]
\theoremstyle{definition}
\newtheorem{definition}[theorem]{Definition}
\newenvironment{subproof}{\begin{proof}[Subproof.]}{\end{proof}}

\setenumerate{label=\rm(\roman*),midpenalty=2}

\crefrangeformat{claim}{Claims~#3#1#4--#5#2#6}
\crefmultiformat{claim}{Claims~#2#1#3}{ and~#2#1#3}{, #2#1#3}{ and~#2#1#3}

\begin{document}

\title[$k$-choosability of graphs with maximum local edge-connectivity $k$]{A Brooks-type theorem for the $k$-choosability of graphs with maximum local edge-connectivity $k$}
\author{Sam Bastida} 
\author{Nick Brettell} 
\address{School of Mathematics and Statistics\\
  Victoria University of Wellington\\
  New Zealand}
\email{samuel.a.bastida@gmail.com}
\email{nick.brettell@vuw.ac.nz}

\maketitle

\begin{abstract}
    For a graph~$G$ with at least two vertices, the \emph{maximum local edge-connectivity} of $G$ is the maximum number of edge-disjoint $(u,v)$-paths over all distinct pairs of vertices $(u,v)$ in $G$.
    Stiebitz and Toft (2018) proved a Brooks-type theorem for graphs with maximum local edge-connectivity $k$, showing that a graph with maximum local edge-connectivity~$k$ is not $k$-colourable if and only if it has a block in $\mathcal{H}_k$, which is the class of graphs that can be obtained by taking Haj\'os joins of copies of $K_{k+1}$ and, when $k=3$, odd wheels.
    We prove that a $2$-connected graph with maximum local edge-connectivity~$k$ is $k$-choosable if and only if it is not in $\mathcal{H}_k$.
    On the other hand, deciding $k$-choosability when restricted to graphs with maximum local edge-connectivity $k$ (that might not be $2$-connected) is $\Pi_2$-complete.
    To prove the former result, we first prove several generalisations of a well-known characterisation of degree-choosability; these may be of independent interest.
\end{abstract}

\section{Introduction}

Brooks' Theorem states that a connected graph with maximum degree $k$ is $k$-colourable\footnote{Here, by ``$k$-colourable'', we mean that the graph has a proper vertex $k$-colouring.  For any other undefined terminology, see \cref{prelims}.} unless it is an odd cycle or a complete graph, in which case $k+1$ colours are necessary~\cite{Brooks1941}. 
There is an extension of Brooks' Theorem for list colouring, due to Vizing~\cite{Vizing1976} and, independently, Erd\H os, Rubin, and Taylor~\cite{ERT1979}. 
Let $G$ be a graph.
A \emph{list assignment} of $G$ is a function $L$ from $V(G)$ to subsets of the positive integers.
For a positive integer~$k$, a \emph{$k$-list assignment} of $G$ is a list assignment of $G$ such that $|L(v)| = k$ for all $v \in V(G)$.
For a list assignment $L$ of $G$, we say that $G$ is \emph{$L$-colourable} if there is a proper $k$-colouring~$\phi$ of $G$ such that $\phi(v) \in L(v)$ for each $v \in V(G)$.
A graph~$G$ is \emph{$k$-choosable} if it is $L$-colourable for every $k$-list assignment~$L$ of $G$.
Using a greedy colouring strategy, it is easily seen that a graph with maximum degree $k$ is $(k+1)$-choosable.
However, as with Brooks' theorem, we can say more when the graph is not an odd cycle or a complete graph.

\begin{theorem}[Vizing~\cite{Vizing1976}; Erd\H os, Rubin, and Taylor~\cite{ERT1979}]
    \label{ertthm}
    Let $G$ be a connected graph with maximum degree $k$.
The graph~$G$ is $k$-choosable if and only if it is not an odd cycle or a complete graph.
\end{theorem}

The \emph{local edge-connectivity} between distinct vertices $u$ and $v$ in a graph is the maximum number of edge-disjoint paths from $u$ to $v$.
For a graph~$G$ with at least two vertices, the \emph{maximum local edge-connectivity} of $G$ is the maximum local edge-connectivity taken over all distinct pairs of vertices of $G$; for a graph with at most one vertex, we define the maximum local edge-connectivity to be zero.
It follows from a result of Mader~\cite{Mader1973} that a graph with maximum local edge-connectivity~$k$ has a vertex of degree~$k$.
Moreover, the maximum local edge-connectivity of a graph~$G$ can only decrease in a subgraph of $G$.
So, for a graph~$G$ with maximum local edge-connectivity $k$, every subgraph of $G$ has a vertex of degree at most $k$; that is, $G$ is $k$-degenerate.
By choosing an appropriate vertex ordering and then using a greedy colouring strategy, it is easily seen that such a graph is $(k+1)$-colourable and, in fact, $(k+1)$-choosable.

Stiebitz and Toft proved a Brooks-type theorem for graphs with maximum local edge-connectivity~$k$~\cite{ST2018}.
To state this result, we require an operation known as a Haj\'os join.
Let $G_1$ and $G_2$ be two graphs on disjoint edge sets, with $V(G_1) \cap V(G_2) = \{v\}$, where $uv$ is an edge of $G_1$, and $vw$ is an edge of $G_2$.
We obtain a \emph{Haj\'os join of $G_1$ and $G_2$ on $(v,u)$ and $(v,w)$}, denoted $(G_1,v,u) \hajos (G_2,v,w)$, by deleting the edge $uv$ from $G_1$, deleting the edge $vw$ from $G_2$, and then obtaining our new graph by identifying $v$ in $G_1$ and $G_2$, and adding an edge $uw$ (see \cref{fighajos})%
\footnote{We allow Haj\'os joins on graphs with parallel edges.  We use the notation $(G_1,v,u) \hajos (G_2,v,w)$ even when there are parallel edges between $u$ and $v$ and/or $v$ and $w$; as we are only interested in graphs up to isomorphism, we avoid any potential ambiguity.}.
Now, for $k \ge 4$, let $\mathcal{H}_k$ be the class of graphs obtained by closing the class of graphs isomorphic to $K_{k+1}$ under Haj\'os joins.  We also let $\mathcal{H}_3$ be the class of graphs obtained by closing the class of odd wheels under Haj\'os joins.

\begin{figure}[t]
    \includegraphics[width=11cm]{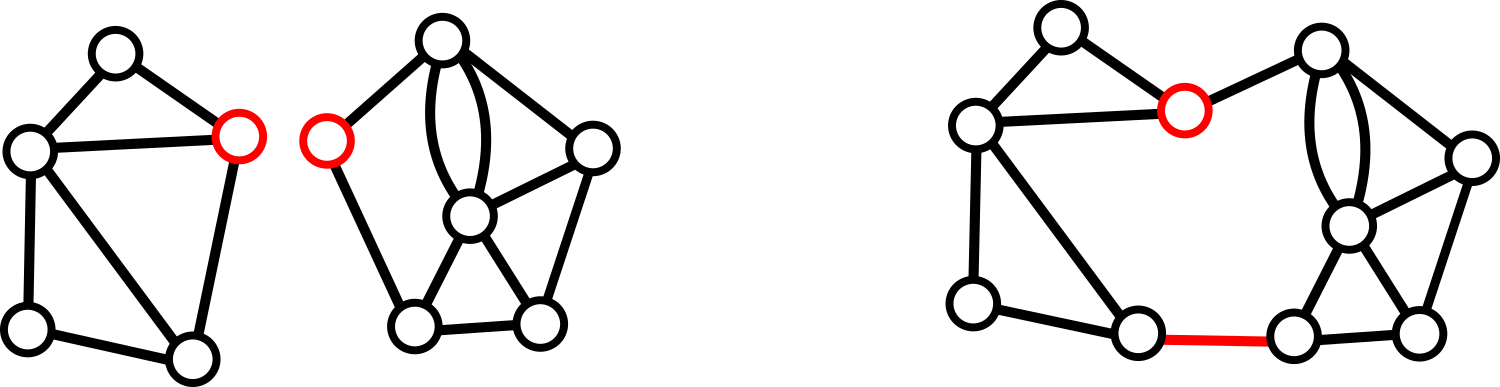}
    \caption{A Haj\'os join}
    \label{fighajos}
\end{figure}

\begin{theorem}[Stiebitz and Toft~\cite{ST2018}]
    \label{stthm}
    Let $G$ be a graph with maximum local edge-connectivity~$k$, for $k \ge 3$.
    The graph~$G$ is $k$-colourable if and only if no block of $G$ is isomorphic to a member of $\mathcal{H}_k$.
\end{theorem}

As a graph with maximum degree $k$ has maximum local edge-connectivity at most~$k$, this result generalises Brooks' theorem.
We note that Aboulker et al.~\cite{ABHMT2017} first proved the case where $k=3$, as well as the special case where the graphs are also $k$-connected for $k \ge 4$. 
Furthermore, we note that it is easy to extend \cref{stthm} to $k \le 2$ (as also observed in \cite{ST2018}).
This is because when $G$ is a connected graph, it is not difficult to see that $G$ has maximum local edge-connectivity at most~$2$ if and only if it is a cactus graph (where any two distinct cycles have no edges in common), and has local edge-connectivity at most~$1$ if and only if it is a tree.

\begin{figure}[b]
  \includegraphics[width=9.5cm]{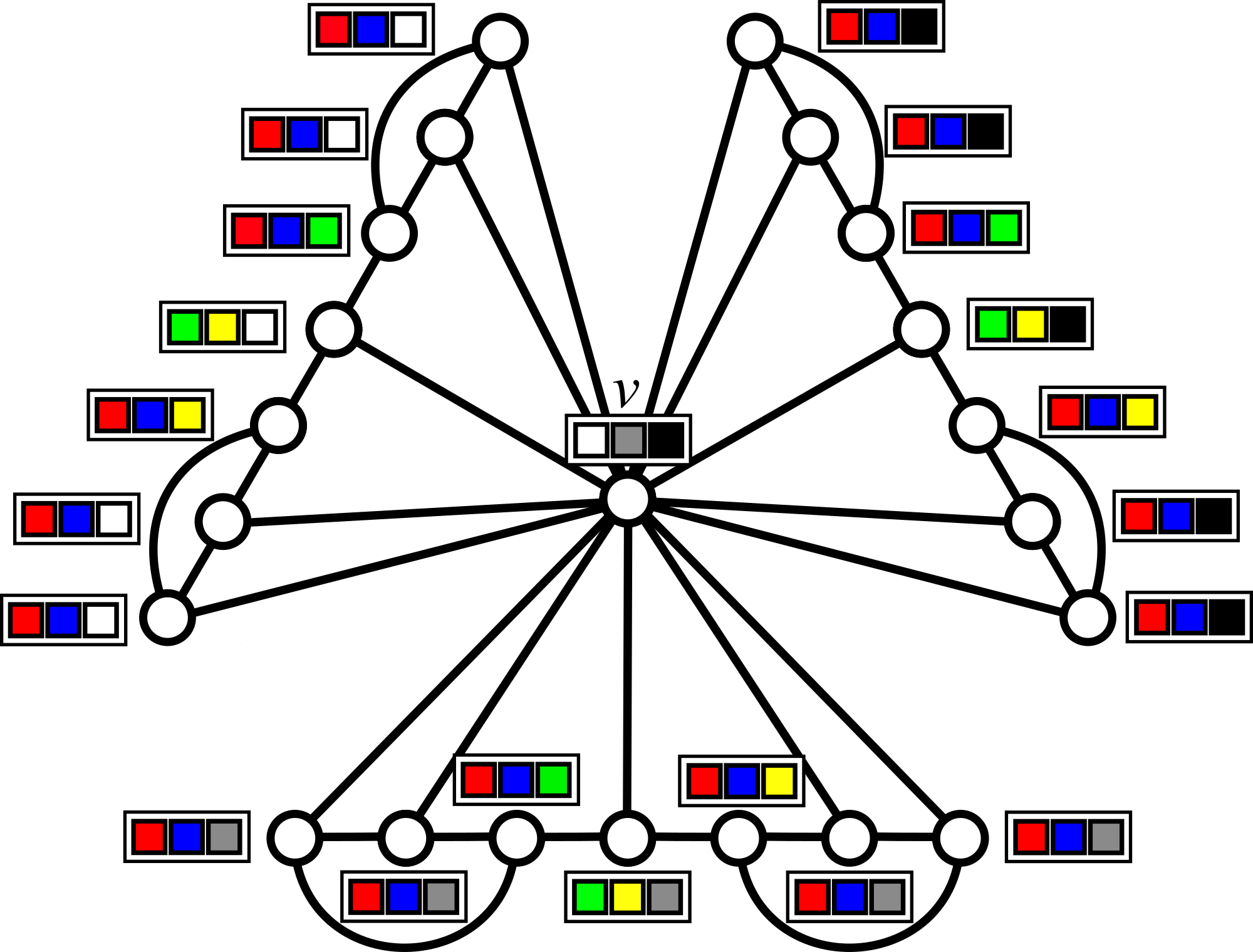}
  \caption{A graph $H$ with maximum local edge-connectivity $3$ that is not $3$-choosable (illustrated by the particular $3$-list assignment), even though $H$ is $3$-colourable, and each block of $H$ is $3$-choosable.}
  \label{figblockcounter}
\end{figure}

In this paper, we consider a Brooks-type theorem for the $k$-choosability of graphs with maximum local edge-connectivity~$k$, in an attempt to generalise both \cref{ertthm} and \cref{stthm}.
Optimistically, one might hope, analogous to \cref{ertthm}, that the only graphs that are not $k$-choosable are those that are not $k$-colourable.
It turns out that this is not the case: in \cref{figblockcounter}, we illustrate a graph with maximum local edge-connectivity $3$ that is $3$-colourable but not $3$-choosable.
Note, however, that this graph is not $2$-connected.
We prove the following:

\begin{theorem}
    \label{mainthm}
    Let $G$ be a $2$-connected graph with maximum local edge-connectivity~$k$, for $k \ge 3$.
    The graph~$G$ is $k$-choosable if and only if $G \notin \mathcal{H}_k$.
\end{theorem}

Moreover, we prove that such a graph in $\mathcal{H}_k$ is not $L$-colourable for some $k$-list assignment $L$ only if $L$ maps every vertex to the same set of $k$ colours (see \cref{mainthmstrong}).

As there is a polynomial-time algorithm for deciding if a graph with maximum local edge-connectivity at most $k$ is a member of $\mathcal{H}_k$ (see \cite{ST2018}), we obtain as a corollary that \textsc{$k$-choosability}, the problem of deciding if a graph is $k$-choosable for some fixed $k$, is polynomial-time solvable when restricted to $2$-connected graphs with maximum local edge-connectivity~$k$.

It is worth commenting further on the different connectivity requirements in \cref{stthm,mainthm}.
Recall that a \emph{block} of a graph is a maximal connected subgraph $H$ such that $|V(H)| \ge 2$ and $H$ has no cut vertices.
A graph~$G$ is $k$-colourable if and only if each block of $G$ is $k$-colourable, since the colours given to a block can be permuted, if necessary, in order to obtain a colouring of $G$.
\Cref{stthm} can be obtained by combining this fact with a characterisation of the $2$-connected graphs with maximum local connectivity $k$ that are not $k$-colourable.
However, it is not always the case that a graph is $k$-choosable when each of its blocks is $k$-choosable; the graph in \cref{figblockcounter} demonstrates this.

Nonetheless, one might still hope for a version of \cref{mainthm} where we do not require that $G$ is $2$-connected, but instead have a larger class of exceptional graphs, including the graph in \cref{figblockcounter} (for example).
Unfortunately, it appears unlikely that such a characterisation, in any reasonable form, is possible, due to the following theorem.  The proof of this result is a straightforward extension of a hardness proof of Erd\H os, Rubin, and Taylor~\cite{ERT1979}.

\begin{theorem}
    \label{hardnessthm}
    For each integer $k \ge 3$, the \textsc{$k$-choosability} problem, when restricted to graphs with maximum local edge-connectivity~$k$, is $\Pi_2$-complete.
\end{theorem}

Despite this, we can say more about the $2$-connected graphs with maximum local edge-connectivity~$k$ that can be problematic for $k$-choosability when appearing as a block in a larger connected graph.
Consider the graph in \cref{figblockcounter} again.  It consists of three blocks that are pairwise isomorphic.  Call one such block $B$.
We say that a vertex~$u$ in a graph~$G$ is \emph{$k$-restricted} if there exists a $k$-list assignment~$L$ of $G$ with $c \in L(u)$ such that every $L$-colouring~$\phi$ of $G$ has $\phi(u) \neq c$.
In the example, although $B$ is $3$-choosable, the vertex labelled $v$ is a $3$-restricted vertex of $B$.
Given $k$ graphs, where each is $k$-choosable and has a $k$-restricted vertex, one can identify the $k$-restricted vertex in each in order to obtain a new graph that is not $k$-choosable.

We prove a necessary condition for a $2$-connected graph with maximum local edge-connectivity~$k$ to have a $k$-restricted vertex, as follows.
To describe this, we first require some definitions.
Note that we allow graphs to have parallel edges in what follows.
For a graph $G$ with maximum degree at most $k$ and minimum degree strictly less than $k$, let $G^+_k$ be the graph obtained from $G$ by adding a vertex $h$ and, for each vertex $v \in V(G)$, adding $k-d_G(v)$ edge(s) in parallel between $h$ and $v$.
Note that except for possibly $h$, each vertex of $G^+_k$ has degree $k$.
We call $G^+_k$ the \emph{almost $k$-regular extension} of $G$.
A graph~$G$ is a \emph{Gallai tree} if it is non-empty, connected, and each block of $G$ is isomorphic to a complete graph or an odd cycle.
For each $k \ge 3$, let $\mathcal{T}_k$ be the class of graphs containing the almost $k$-regular extensions of any Gallai tree with maximum degree at most $k$ and minimum degree strictly less than $k$, closed under Haj\'os joins (and isomorphism).
Observe that $\mathcal{H}_k \subseteq \mathcal{T}_k$ for any $k \ge 3$.

\begin{theorem}
    \label{restrictedthm}
    Let $G$ be a $2$-connected graph with maximum local edge-connectivity~$k$, for $k \ge 3$.
    If $G$ has a $k$-restricted vertex, then $G \in \mathcal{T}_k$.
\end{theorem}


Consider again the problem of determining if a graph $G$ with maximum local connectivity $k$ is $k$-choosable, when $G$ is not $2$-connected.  If no block of $G$ is in $\mathcal{T}_k$, then no block of $G$ has a $k$-restricted vertex, and it follows that $G$ is $k$-choosable.  However, if $G$ does have such a block, it may or may not be $k$-choosable, and determining this seems hopeless, due to \cref{hardnessthm}.

\subsection{Our strategy}

In an attempt to prove \cref{mainthm}, one might hope to generalise the strategy used by Stiebitz and Toft \cite{ST2018} to prove \cref{stthm}.
Their approach relies heavily on the theory of \emph{critical} graphs; that is, graphs for which the chromatic number decreases in any proper subgraph.
Although list-critical graphs~\cite{STV2009} and choosable-critical graphs~\cite{Voigt1998} have been studied, they appear to lack many of the useful properties of critical graphs used in \cite{ST2018}.

A different approach was taken by Aboulker et al.~\cite{ABHMT2017}.
They first considered the special case where the graphs are also $k$-connected, proving a weaker version of \cref{stthm} for this class.
They use a structural property of these graphs (a $k$-connected graph with maximum local edge-connectivity~$k$ can be decomposed into parts, each containing at most one vertex of degree more than $k$, where the parts are separated by disjoint $k$-edge cuts \cite[Lemma~3.2]{ABHMT2017}) together with a version of Brooks' Theorem where a single high degree vertex is allowed~\cite[Lemma~3.1]{ABHMT2017}.
Then, in the case that $k=3$, the more general result can be obtained by a case analysis after decomposing along $2$-separations.

It is not too difficult to obtain a weaker version of \cref{mainthm} for graphs that are also $k$-connected, using a similar approach together with a known characterisation of degree-choosability (see \cref{dcthm} in \cref{prelims}).  However, dealing with $2$-separations is vastly more complicated for $k$-choosability.
To give some intuition for the difficulty, consider a graph~$G$ with maximum local connectivity $3$, where $G$ is the Haj\'os join $(G_0,v,u) \hajos (G_1,v,w)$ and $G_0$ and $G_1$ are $3$-connected graphs.
If $G_0$ and $G_1$ are both $3$-colourable, then, by permuting colours, it is easy to obtain a $3$-colouring of $G$.
Suppose $G_0$ and $G_1$ are both $3$-choosable.
Let $L$ be a $3$-list assignment for $G$ and let $L_0 = L|_{V(G_0)}$ and $L_1 = L|_{V(G_1)}$.
It is possible that, for any $L_1$-colouring~$\phi_1$ of $G_1$, if $\phi_1(v)=1$ then $\phi_1(w) = 2$.
In this case, to say that $\phi_1$ extends to an $L$-colouring of $G$, it is not enough to know that $G_0$ is $3$-choosable: we need an $L_0$-colouring~$\phi_0$ of $G_0$ where if $\phi_0(v)=1$, then $\phi_0(u) \neq 2$.

To address this, we introduce the notion of polar list colourings (see \cref{secpolar}).
We consider graphs with a subset of edges that are ``polarised'', where for a polarised edge, say for example $e=uv$, rather than the usual constraint that $u$ and $v$ are coloured different colours, we instead have a constraint that we cannot have both $u$ coloured some colour $c_u$ and $v$ coloured some colour $c_v$ (that is, if $u$ is coloured $c_u$, then $v$ is not coloured $c_v$, and vice versa).
We then prove some generalisations of a well-known characterisation of degree-choosability (see \cref{dcthm}); these results may be of independent interest.
Firstly, we prove a result characterising when graphs with polarised edges are degree-choosable (see \cref{dcpolarthm} in \cref{dcpolarsec}).
We then prove a further generalisation where a single higher-degree vertex is allowed (see \cref{dcpolarhdthm} in \cref{dcpolarhdsec}), and give a necessary property for graphs with a single higher-degree vertex to have a $k$-restricted vertex (see \cref{k restricted theorem} in \cref{k restricted sec}).
Using these results, our strategy to prove \cref{mainthm} is to first prove \cref{restrictedthm}, then consider the graphs in $\mathcal{T}_k$ that are not $k$-choosable.

We note that polar list colourings are related to DP-colourings (also known as correspondence colourings).  These were introduced by Dvorak and Postle \cite {DP2018} in order to prove that planar graphs with no cycles of length 4 to 8 are $3$-choosable, for similar reasons to those described in the previous two paragraphs.  A polar list colouring can be viewed as a DP-colouring, but DP-colourings are more general.  Bernshteyn et al.~\cite{BKP2017} proved a characterisation of degree-choosability for DP-colouring graphs (allowing for parallel edges), and the exceptional cases were described fully by Kim and Ozeki~\cite{KO2019}; these results imply our \cref{dcpolarthm}, though some analysis is required, so we provide an independent proof for completeness.
Nonetheless, we believe polar graphs are a useful abstraction, and may be of independent interest for settings, like here, where the full generality of DP-colouring is not required.

\subsection{Related classes, and previous work}

The class of graphs with maximum local edge-connectivity~$k$, for some positive integer~$k$, have been studied previously, along with several related classes.  We now describe these classes, and briefly summarise previous work relating to them.

There is an analogue of maximum local edge-connectivity for vertex connectivity.
Let $G$ be a graph, and let $u$ and $v$ be distinct vertices in $G$.
The \emph{local connectivity} between $u$ and $v$, denoted $\kappa(u,v)$, is the maximum number of internally disjoint paths from $u$ to $v$.
For a graph~$G$ with at least two vertices, the \emph{maximum local connectivity} of $G$ is the maximum of $\kappa(u,v)$ taken over all distinct vertices $u,v$ of $G$; for a graph with at most one vertex, we define the maximum local connectivity to be zero.
We let $\lambda(u,v)$ denote the local edge-connectivity between $u$ and $v$.
Clearly $\kappa(u,v) \le \lambda(u,v) \le \min\{d(u),d(v)\}$, so a graph with maximum degree at most $k$ has maximum local edge-connectivity at most $k$, and, in turn, such a graph has maximum local connectivity at most $k$.
As the class of cubic graphs has unbounded treewidth, it follows that the class of graphs with maximum local (edge-)\allowbreak connectivity $k$, for $k \ge 3$, also has unbounded treewidth.
The problem of determining the maximum number of edges in a graph with maximum local (edge-)\allowbreak connectivity $k$, and the extremal graphs, dates back to Bollob\'as and Erd\H os in the 1960s, and is well studied \cite{Bollobas1978,Leonard1973,Sorensen1974,Mader1973}.

For any integer $k \ge 2$, a graph~$G$ is $k$-connected if it has at least $k+1$ vertices and $\kappa(u,v) \ge k$ for all distinct vertices $u,v \in V(G)$ (this definition is equivalent to a more standard definition of $k$-connectedness, by Menger's theorem).
A graph~$G$ is \emph{minimally $k$-connected} if it is $k$-connected and the deletion of any edge results in a graph that is not $k$-connected.
A $k$-connected graph is minimally $k$-connected if and only if $\kappa(u,v) = k$ for every pair of adjacent vertices $u$ and $v$ \cite[Lemma~4.2]{Bollobas1978}.
So the class of minimally $k$-connected graphs is a superclass of $k$-connected graphs with maximum local connectivity $k$.

A graph $G$ is \emph{uniformly $k$-connected} if $\kappa(u,v)=k$ for every pair of distinct vertices $u$ and $v$ in $G$, whereas a graph is \emph{uniformly $k$-edge-connected} if $\lambda(u,v)=k$ for every pair of distinct vertices $u$ and $v$ in $G$.
Observe that a graph with maximum local (edge-)connectivity $k$ is uniformly $k$-connected (or uniformly $k$-edge-connected) if and only if it is $k$-connected (or $k$-edge-connected, respectively).
Moreover, a $k$-connected graph with maximum local edge-connectivity~$k$ is uniformly $k$-edge-connected.
Uniformly $k$-connected graphs were studied in \cite{BOP2002,GHS2022,GH2024,Xu2024}.
It was shown by Aboulker et al.~\cite{ABHMT2017} and, independently, G\"oring et al.~\cite{GHS2022}, that
a uniformly $k$-connected graph is uniformly $k$-edge-connected when $k=3$, but this does not hold for $k\ge 4$.
Thus, using \cref{stthm}, a polynomial-time algorithm for deciding $3$-colourability (and finding a $3$-colouring) when restricted to uniformly $3$-connected graphs follows easily (as first observed in \cite{ABHMT2017}).
However, deciding if a graph with maximum local connectivity~$k$ is $k$-colourable is NP-complete when $k = 3$ \cite[Proposition 4.2]{ABHMT2017}, and, using a variant of this approach, one can show it is also NP-complete when $k\ge 4$ (see \cref{appendix}).
%
This suggests there is no prospect of extending Brooks' Theorem to graphs with maximum local connectivity $k$, 
for $k \ge 3$.
To the best of our knowledge, the complexity of deciding if a uniformly $k$-connected graph is $k$-colourable for $k \ge 4$ remains open (see \cite[Question~1.7]{ABHMT2017}).


Finally, we note that \cref{stthm} has also been generalised to hypergraphs~\cite{SST2022}, and digraphs~\cite{AAC2023}.

\section{Preliminaries}
\label{prelims}

Our terminology follows \cite{BM2008} unless otherwise specified.
We view a graph as a pair $(V,E)$, together with an implicit incidence function that describes, for each $e \in E$, a distinct pair of vertices in $V$, which we also call the \emph{ends} of $e$.
Throughout, graphs are loopless, but may have parallel edges\footnote{Although it is atypical to allow parallel edges when considering graph colouring problems, it is convenient for our notion of polar list colourings, and for defining certain graph classes that are closed under Haj\'os join.}; as we are usually interested in graphs up to isomorphism, we still refer to an edge $e$ with ends $u$ and $v$ as ``$uv$'', without ambiguity.
A graph is \emph{null} if it has no edges, and \emph{empty} if it has no vertices.
A graph is \emph{simple} if it has no parallel edges.
A \emph{parallel class} is the set of all edges between a pair of adjacent vertices.

Let $G$ be a graph.
We write $V(G)$ and $E(G)$ to denote the vertex set and edge set of $G$, respectively.
For a graph $H$, we write $G \cong H$ to denote that $G$ is isomorphic to $H$.
For $V' \subseteq V(G)$, we let $G[V']$ denote the subgraph of $G$ on vertex set $V'$ and edge set consisting of each edge in $E(G)$ having both ends in $V'$.
We say $G'$ is an \emph{induced subgraph} of $G$ if $G' \cong G[V']$ for some $V' \subseteq V(G)$.
For $v \in V(G)$, we write $G-v$ to denote $G[V(G) \setminus \{v\}]$. 
Likewise, for $X \subseteq V(G)$, we write $G-X$ to denote $G[V(G) \setminus X]$.
For $e \in E(G)$ (or $X \subseteq E(G)$), we write $G \ba e$ (or $G \ba X$ respectively) to denote the graph with vertex set $V(G)$ and edge set $E(G) \setminus \{e\}$ (or $E(G) \setminus X$, respectively).
The \emph{degree} of a vertex $v \in V(G)$ is the number of edges incident with $v$.
For a vertex~$v \in V(G)$, we write $d_G(v)$ to denote the degree of $v$, or just $d(v)$ when $G$ is obvious from context.
The \emph{maximum degree} (or \emph{minimum degree}) of $G$ is the maximum degree (or minimum degree, respectively) of a vertex, taken over all vertices of $G$.
A graph is \emph{$k$-regular} if every vertex has degree $k$.
The \emph{neighbourhood} of a vertex $v$ in $G$, denoted $N_G(v)$ or just $N(v)$, is the set of vertices that are adjacent to $v$.
A vertex $v$ in $G$ is \emph{dominating} if $N(v) \cup \{v\} = V(G)$.
For $X,Y \subseteq V(G)$, an \emph{$(X,Y)$-path} is a path from some $u\in X$ to some $v \in Y$.
We refer to an $(x,y)$-, $(X,y)$-, and \emph{$(x,Y)$-path} rather than an $(\{x\},\{y\})$-, $(X,\{y\})$-, and $(\{x\},Y)$-path, respectively.

When convenient, we treat a path or cycle as a subgraph.
In particular, for a path $P$ we write $V(P)$ and $E(P)$ to refer to the vertices and edges of the path respectively.
The \emph{length} of a path $P$ is $|E(P)|$.
Two $(u,v)$-paths $P$ and $P'$ are \emph{internally disjoint} if $E(P) \cap E(P') = \emptyset$ and $V(P) \cap V(P') = \{u,v\}$.
The \emph{length} of a cycle or path is the number of edges in the cycle or path, and a cycle (or path) is \emph{odd} if it has odd length.
An \emph{odd wheel} is a simple graph that can be obtained from an odd cycle (of length at least three) by adding a vertex, called the \emph{hub}, that is adjacent to every vertex of the cycle.
A \emph{complete graph}, denoted $K_n$,
is a simple graph on $n$ vertices where every vertex has degree $n-1$, for some $n \ge 1$.
A \emph{clique} in a graph $G$ is a set $X \subseteq V(G)$ such that $G[X]$ is isomorphic to a complete graph.
We let $I_k$ denote the graph consisting of two vertices with $k$ edges between them.

For $v \in V(G)$, we say $v$ is a \emph{cut vertex} if $G-v$ has more components than $G$; otherwise, $v$ is an \emph{internal} vertex.
For $X \subseteq V(G)$, we say $X$ is a \emph{vertex cut} if $G-X$ has more components than $G$.  For disjoint sets $A,X,B \subseteq V(G)$, we say that $X$ separates $A$ from $B$ if every $(A,B)$-path in $G$ contains a vertex in $X$.
For $e \in E(G)$, we say $e$ is a \emph{bridge} if $G \ba e$ has more components than $G$.

For graphs $G$ and $G'$ with disjoint edge sets, we define the {\em union} of $G$ and $G'$, denoted $G \cup G'$, to be the graph with vertex set $V(G) \cup V(G')$ and edge set $E(G) \cup E(G')$.
When $G$ and $G'$ are graphs with disjoint edge sets but $|V(G) \cap V(G')| = 1$, then we say $G \cup G'$ is the $1$-join of $G$ and $G'$.

A \emph{(vertex) colouring} of a graph $G$ is a function $\phi$ from $V(G)$ to a set of colours~$Z$.
We say $\phi$ is a \emph{$k$-colouring} when $|Z|=k$.  Usually it is convenient to assume $Z = \{1,2,\dotsc,k\}$.  
A colouring of $G$ is \emph{proper} if, for each edge $uv$ of $G$, we have $\phi(u) \neq \phi(v)$.
A graph is \emph{$k$-colourable} if it has a proper $k$-colouring.
The \emph{chromatic number} of a graph $G$ is the smallest positive integer $\chi$ such that $G$ is $\chi$-colourable.

We write $[t]$ to denote the set $\{1,2,\dotsc,t\}$.  For sets $X$ and $Y$, we write $X \symdiff Y$ to denote set difference, that is, $X \symdiff Y = (X \setminus Y) \cup (Y \setminus X)$.
When $\phi$ is a function with domain $X$, and $X' \subseteq X$, we write $\phi|_{X'}$ to denote the restriction of $\phi$ to $X'$.
When $\phi$ is a function with codomain $\mathbb{N}$, we say that it is \emph{bounded by $k$}, for a positive integer~$k$, if each image is less than or equal to $k$.
We write $\emptyset$ to denote both the empty set, and a function with empty domain.


\subsection{Connectivity and blocks}

For any integer $k \ge 2$, a graph~$G$ is \emph{$k$-connected} if it has at least $k+1$ vertices and, for all distinct vertices $u,v \in V(G)$, there are at least $k$ internally disjoint paths from $u$ to $v$.
Note that, by Menger's theorem, a connected graph with at least $k+1$ vertices is $k$-connected if and only if it has no vertex cuts of size less than $k$.
A graph~$G$ is \emph{$k$-edge-connected} if it has at least two vertices and, for all distinct vertices $u,v \in V(G)$, there are at least $k$ edge-disjoint paths from $u$ to $v$.
Recall also that a \emph{block} of a graph is a maximal connected subgraph $H$ such that $|V(H)| \ge 2$ and $H$ has no cut vertices.

Let $G$ be a graph, let $\mathcal{B}_G$ denote the set of blocks of $G$, and let $X$ be the set of cut vertices of $G$.
We let $B(G)$ denote the \emph{block-cut graph} of $G$; that is, $B(G)$ is the bipartite graph on vertex set $X \cup \mathcal{B}_G$ where there is an edge between $x \in X$ and $B \in \mathcal{B}_G$ if and only if $x \in V(B)$.
The following facts are well known (see \cite{BM2008} for example).

\begin{theorem}
\label{block cut graph theorem}
    Let $G$ be a graph. The following hold:
    \begin{enumerate}
        \item $B(G)$ is a forest, and the components of $B(G)$ are the block-cut graphs of the components of $G$.\label{block cut graph is forest}
        \item For distinct blocks $B, B'$ in $G$, we have $|V(B) \cap V(B')| \le 1$ and $E(B) \cap E(B') = \emptyset$.\label{no sharing}
        \item There are distinct blocks $B$ and $B'$ of $G$ that both contain a vertex $v$ if and only if $v$ is a cut vertex of $G$.\label{only share cut}
    \end{enumerate}
\end{theorem}

We say that a block $B$ of $G$ is a \emph{$K_2$-block} if $B \cong K_2$.
If $B$ and $B'$ are distinct blocks of $G$ that share a vertex, we say they are \emph{adjacent}.
If $B$ is a block of $G$ such that $B$ is a leaf of $B(G)$, then $B$ is a {\em leaf block}; otherwise, we say it is a \emph{non-leaf block}.
A leaf block contains precisely one cut vertex of $G$.
It is easily seen that if a connected graph $G$ has at least two blocks, then it has at least two leaf blocks.
We require the following elementary result.

\begin{lemma}
\label{leaf block lemma}
    Let $G$ be a connected graph with at least two blocks. Then $G$ contains a leaf block that is adjacent to at most one non-leaf block.
\end{lemma}

We also use the following well-known lemma, which follows from Menger's Theorem.
\begin{lemma}
    \label{2fanlemma}
    Let $G$ be a $2$-connected graph with distinct vertices $u$, $v$, and $v'$.
    Then there exists a $(u,v)$-path $P$ and a $(u,v')$-path $P'$ such that $P$ and $P'$ are internally disjoint.
\end{lemma}

\subsection{Degree-choosability}
\label{dcsec}

A list assignment $L$ of $G$ is a \emph{degree-list assignment} if $|L(v)| \ge d(v)$ for all $v \in V(G)$.
A graph~$G$ is \emph{degree-choosable} if it is $L$-colourable for every degree-list assignment $L$.
Recall that a graph~$G$ is a \emph{Gallai tree} if it is connected and each block of $G$ is isomorphic to a complete graph or an odd cycle.

\begin{theorem}[Vizing~\cite{Vizing1976}; Erd\H os, Rubin, and Taylor~\cite{ERT1979}]
    \label{dcthm}
    Let $G$ be a connected graph.  Then $G$ is degree-choosable if and only if $G$ is not a Gallai tree.
\end{theorem}

Let $G_1$ and $G_2$ be graphs with disjoint edge sets, having list assignments $L_1$ and $L_2$ respectively.
We obtain a list assignment $L$ for $G_1 \cup G_2$, which we call the \emph{list assignment union} and denote $L_1 \cup L_2$, by defining
$$L(u) = \begin{cases}
    L_1(u) & \mbox{if $u \in V(G_1) \setminus V(G_2)$}\\
    L_2(u) & \mbox{if $u \in V(G_2) \setminus V(G_1)$}\\
    L_1(u) \cup L_2(u) & \mbox{if $u \in V(G_1) \cap V(G_2)$.}
\end{cases}$$
It is easily seen that the list assignment union is commutative and associative, so we can consider the union of a set of list assignments without ambiguity.

For graphs $G_1$ and $G_2$ with list assignments $L_1$ and $L_2$ respectively, we say that $L_1$ and $L_2$ are \emph{non-conflicting} if, for each $v \in V(G_1) \cap V(G_2)$, we have $L_1(v) \cap L_2(v) = \emptyset$; otherwise, they are \emph{conflicting}.
A $k$-list assignment $L$ for a graph $G$ is \emph{uniform} if $L(u) = L(v)$ for all $u,v \in V(G)$.

Let $G$ be a Gallai tree.  Note that each vertex in a block $B$ of $G$ has the same degree in $B$, which we denote $d(B)$.
A degree-list assignment $L$ of a Gallai tree~$G$ is \emph{bad} if $L$ can be obtained as the union, taken over each block $B$ of $G$, of uniform $d(B)$-list assignments that are pairwise non-conflicting.
Note, in particular, that for a bad degree-list assignment~$L$ of $G$, we have that $|L(v)|= d_G(v)$ for all $v \in V(G)$.

The following is well known (see \cite[Problem~9.12(b)]{Lovasz1979}, for example).

\begin{proposition}
    \label{baddegreeassign}
    Let $G$ be a Gallai tree with a degree-list assignment $L$.  Then $G$ is $L$-colourable if and only if $L$ is not a bad degree-list assignment.
\end{proposition}

\section{Polar list colourings}
\label{secpolar}

In this section, we introduce polar list colourings. 

A {\em polar graph} is a pair $(G,F)$ where $G$ is a graph and $F \subseteq E(G)$.
Let $Q=(G,F)$ be a polar graph.
We call $G$ the {\em underlying graph} and $F$ the {\em polarised edges of $Q$}, and denote these $G(Q)$ and $F(Q)$ respectively.
We often refer to $Q$, rather than $G(Q)$, when there is no ambiguity; in particular, we write $V(Q)$ and $E(Q)$, rather than $V(G(Q))$ and $E(G(Q))$, respectively.
We say two polar graphs $Q_0$ and $Q_1$ are {\em isomorphic}, and write $Q_0 \cong Q_1$, if there is a pair of bijections $f_V:V(Q_0) \rightarrow V(Q_1)$ and $f_E:E(Q_0) \rightarrow E(Q_1)$ that preserve incidence and $f_E(e)$ is polarised if and only if $e$ is polarised.

A {\em polarisation} of $Q$ is a function $R$ on the polarised edges of 
$Q$ that assigns an (ordered) pair of colours to the vertices incident with the edge.
When $e=uv$ is a polarised edge, we write $R(e,u,v)=(c,d)$ to denote that $R$ assigns $u$ the colour~$c$, and $v$ the colour~$d$ (where, for a colouring to respect this polarisation, we will require that either the colour for $u$ is not $c$, or the colour for $v$ is not $d$).
Note that this assignment is associated with the edge $e$, and other polarised edges could assign other colours to $u$ and $v$. 
To simplify notation, we require that the outputs of $R$ are independent of the ordering of the ends of an edge given as input. %
Formally, for an edge $e=uv$, let $Z(e) = \{(e,u,v),(e,v,u)\}$, and for $E' \subseteq E(Q)$, let $Z(E') = \bigcup_{e \in E'}Z(e)$. %
Then $R : Z(F) \rightarrow \mathbb{N} \times \mathbb{N}$ with the property that if $R(e,u,v) = (c_u,c_v)$, then $R(e,v,u) = (c_v,c_u)$.

A {\em polar assignment} $\Gamma$ of 
$Q$ is a pair $(L,R)$, where $L$ is a list assignment of $G(Q)$ and $R$ is a polarisation of $Q$.
A \emph{$\Gamma$-colouring} $\phi$ of 
$Q$ is a colouring $\phi$ of $G(Q)$ such that
\begin{enumerate}[label=\textbf{(\textup{P\arabic*})}]
        \item $\phi(v) \in L(v)$ for all $v \in V(Q)$, and\label{Gamma 1}
        \item for each $e \in E(Q)$, if $e=uv$ is not polarised then $\phi(u) \neq \phi(v)$, otherwise, when $e=uv$ is polarised, we have $R(e,u,v) \neq (\phi(u),\phi(v))$.\label{Gamma 2}
\end{enumerate}
Note that $\phi$ need not be a proper colouring.
Throughout, we illustrate list assignments using colours (instead of integers), and illustrate a polarisation $R(e,u,v) = (c_u,c_v)$ by colouring the half of the polarised edge nearest to $u$ with the colour $c_u$, and the half nearest to $v$ with the colour $c_v$; see \cref{ncp1fig} for example.

Let $\Gamma=(L,R)$ be a polar assignment.
We say that $Q$ is {\em $\Gamma$-colourable} if there exists a $\Gamma$-colouring of $Q$.
If $L$ is a $k$-list assignment for some integer $k$, then we say that $\Gamma$ is a \emph{$k$-polar assignment}.
If, for some integer $k$, the polar graph $Q$ is $\Gamma$-colourable for each $k$-polar assignment~$\Gamma$, then we say that $Q$ is {\em $k$-choosable}.
If $L$ is uniform, then we say that $\Gamma$ is {\em uniform}.
Let $f:V(Q)\rightarrow \mathbb{N}$ be a function.
If $|L(v)|=f(v)$ for all $v \in V(Q)$, then we say that $\Gamma$ is an \emph{$f$-polar assignment}.
If $Q$ is $\Gamma$-colourable for each $f$-polar assignment~$\Gamma$, then we say that $Q$ is {\em $f$-choosable}. 

Occasionally, it is useful to consider a polar graph $Q$ together with a polar assignment $\Gamma$ for $Q$, as a pair $(Q,\Gamma)$.
In this case, we say that ``$(Q,\Gamma)$ is colourable'' as a shorthand for ``$Q$ is $\Gamma$-colourable''.

Notice that if $F(Q)=\emptyset$, then $Q$ is $\Gamma$-colourable if and only if $G(Q)$ is $L$-colourable.
So polar list colouring is a generalisation of list colouring.
It is natural to ask if there exist any polar graphs that are not $k$-choosable, but whose underlying graphs are $k$-choosable.
As it turns out, such graphs do exist.
\Cref{ncp1fig} illustrates a polar graph that is not $f$-choosable, but the underlying graph is $f$-choosable, when $f$ maps $v_3$ and $v_4$ to $3$, and maps every other vertex to $2$.
\Cref{ncp2fig} illustrates a polar graph that is not $3$-choosable, but its underlying graph is $3$-choosable.
The fact that the latter polar graph is not $3$-choosable will be evident by the end of this section (see \cref{topuplemma}), when we will describe the construction of a polar graph with a $k$-polar assignment that simulates a polar graph with an $f$-polar assignment.


\begin{figure}[hb]
    \centering
    \begin{subfigure}{0.45\textwidth}
        \centering
        \includegraphics[scale=0.9]{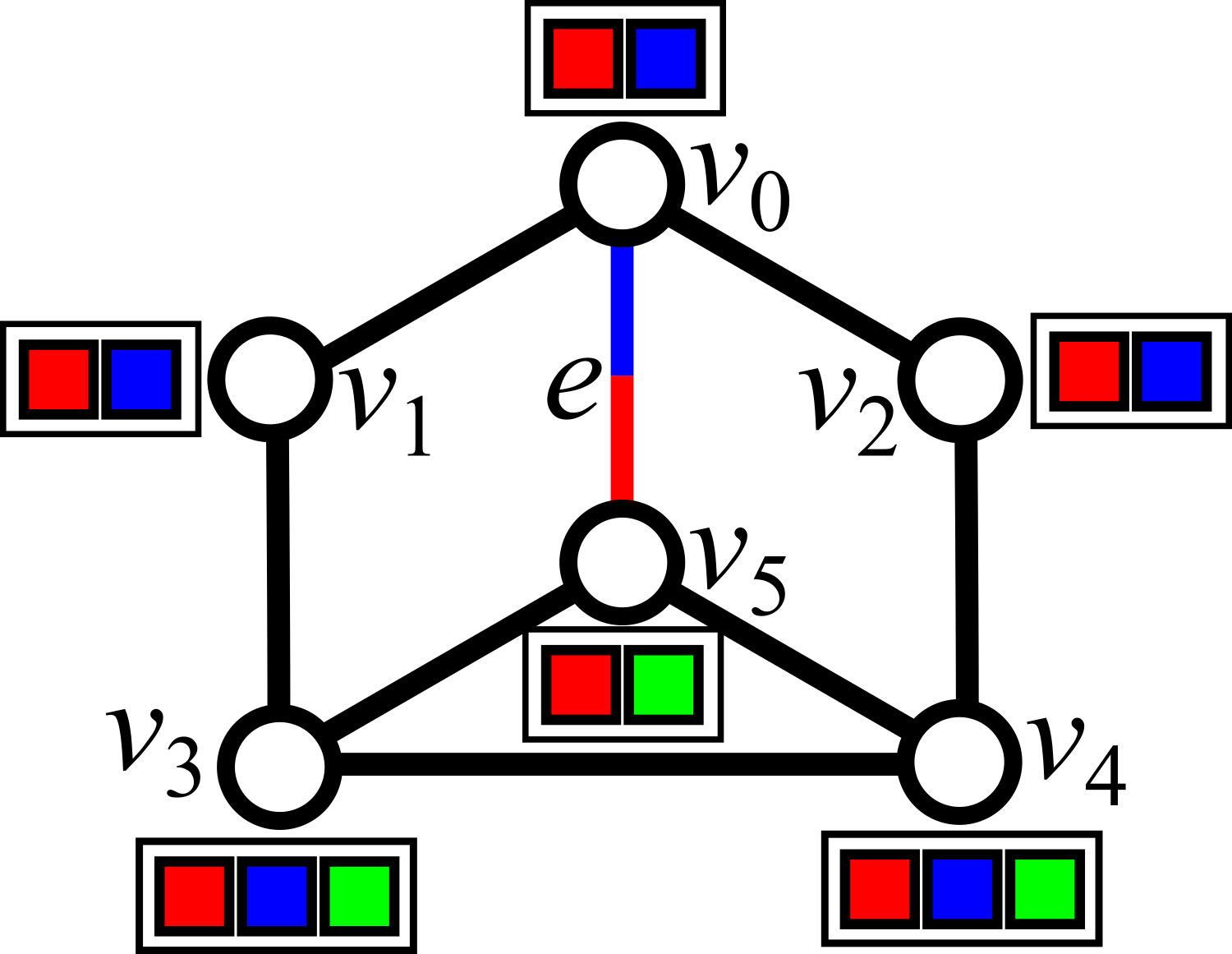}
        \subcaption{A polar graph $Q_6$, where $e$ is the only polarised edge. %
            This polar graph is not $f$-choosable, but the underlying graph is $f$-choosable (where $f(v_3)=f(v_4)=3$ and $f(v_0)=f(v_1)=f(v_2)=f(v_5)=2$). %
            An $f$-polar assignment that certifies this is illustrated.}
        \label{ncp1fig}
    \end{subfigure}\quad
    \begin{subfigure}{0.45\textwidth}
        \centering
        \includegraphics[scale=0.45]{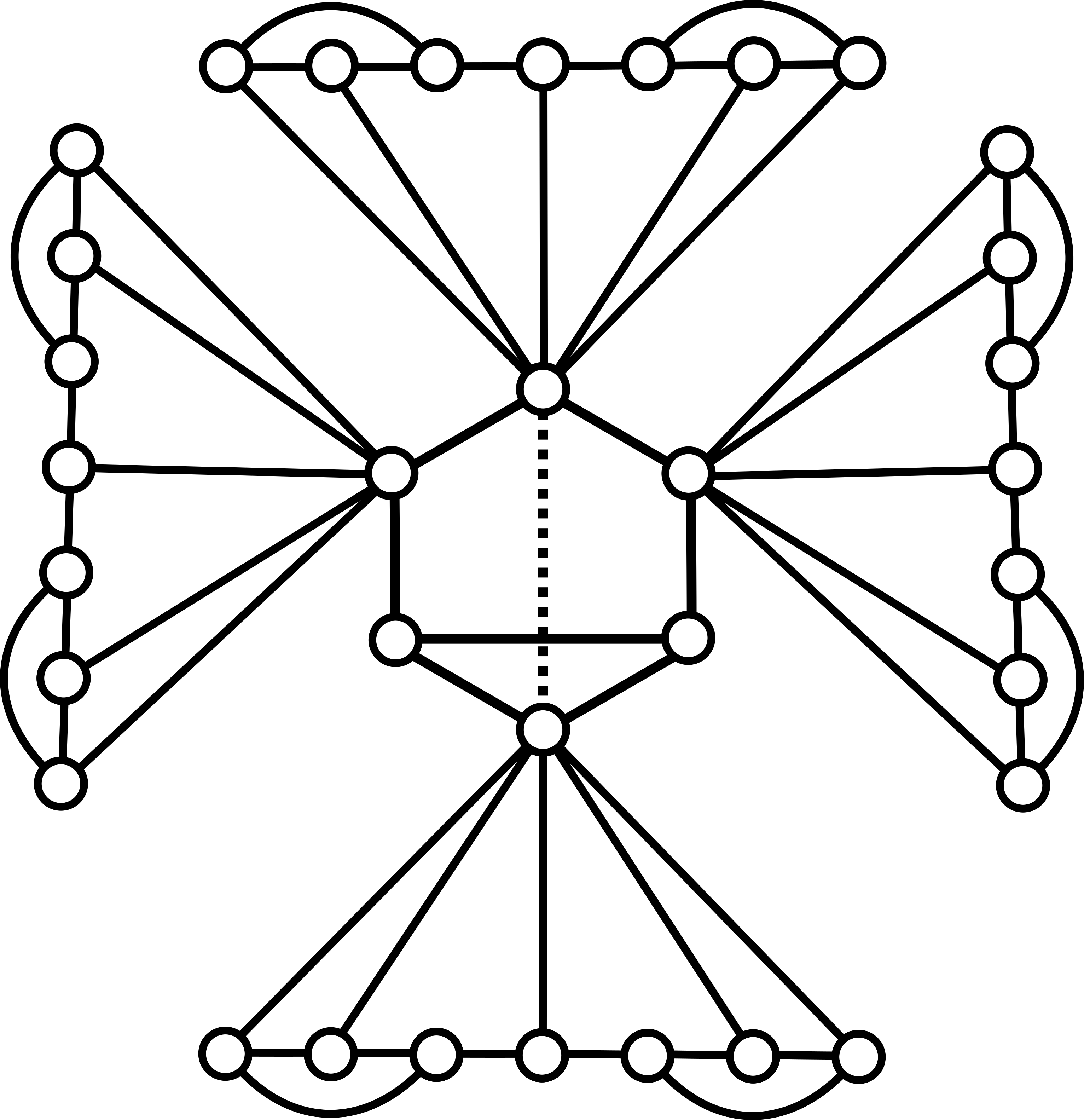}
        \subcaption{A polar graph that is not $3$-choosable, but the underlying graph is $3$-choosable.  The single polarised edge is illustrated with a dotted line.}
        \label{ncp2fig}
    \end{subfigure}
    \caption{Polar graphs that are not choosable, but the underlying graph is choosable.}
\end{figure}

\subsection{Substructures}

Let $Q=(G,F)$ be a polar graph.
We say that $Q'=(G',F')$ is a \emph{polar subgraph} of $Q$ whenever $G'$ is a subgraph of $G$ and 
$F' = F \cap E(G')$.
We say a polar subgraph $Q'$ of $Q$ is \emph{proper} if $Q' \neq Q$.
For a subgraph $G'$ of $G$, we say that $(G', F \cap E(G'))$ is the \emph{polar subgraph induced by $G'$}.
In particular, for $e \in E(Q)$, we write $Q \ba e$ to denote the polar subgraph induced by $G \ba e$.
The \emph{connected components} of $Q$ are the polar subgraphs of $Q$ whose underlying graphs are the connected components of $G$,
and the \emph{blocks} of $Q$ are the polar subgraphs of $Q$ whose underlying graphs are the blocks of $G$.
We write $\mathcal{B}_Q$ to denote the set of all blocks of $Q$.

Let $R$ be a polarisation of $Q$, and let $Q'$ be a polar subgraph of $Q$.
The \emph{restriction} of $R$ to $Q'$, denoted $R|Q'$, is the restriction of $R$ to $Z(F(Q'))$, that is, triples whose first component is in $F(Q')$.
Now let $\Gamma=(L,R)$ be a polar assignment for $Q$.
The \emph{restriction} of $\Gamma$ to $Q'$, denoted $\Gamma|Q'$, is the polar assignment $(L|_{V(Q')}, R|Q')$.
Note that if $Q$ is $\Gamma$-colourable, then $Q'$ is $\Gamma|Q'$-colourable.
For simplicity, we refer to a $\Gamma|Q'$-colouring of $Q'$ as a $\Gamma$-colouring of $Q'$, and say $Q'$ is $\Gamma$-colourable, rather than $\Gamma|Q'$-colourable.

We say that $Q'$ is an \emph{induced polar subgraph} of $Q$ if there exists some $V' \subseteq V(Q)$ such that $Q'=(G[V'], F \cap E(G[V']))$, in which case we denote $Q'$ by $Q[V']$.
For $V'' \subseteq V(Q)$, we also write $Q-V''$ to denote the induced subgraph $Q[V(Q) \setminus V'']$.

\subsection{Colour deletion and restricted vertices}

Sometimes we would like to colour a small set of vertices and consider if this partial colouring can be extended to the rest of the graph.
We now formalise these notions in the context of polar list colourings.


Let $Q'$ and $Q$ be polar graphs with $V(Q') \subseteq V(Q)$, and let $\Gamma'$ and $\Gamma$ be polar assignments of $Q'$ and $Q$ respectively.
Let $\phi'$ be a $\Gamma'$-colouring of $Q'$.
For a $\Gamma$-colouring $\phi$ of $Q$, we say that \emph{$\phi'$ extends to $\phi$} if $\phi|_{V[Q']}=\phi'$.
We say that \emph{$\phi'$ extends to a $\Gamma$-colouring of $Q$} if there exists a $\Gamma$-colouring $\phi$ of $Q$ such that $\phi'$ extends to $\phi$.
When $\phi'$ is a $\Gamma$-colouring of $Q'$ (that is, $\Gamma'=\Gamma|Q'$) and $\phi'$ extends to a $\Gamma$-colouring of $Q$, we also say simply that \emph{$\phi'$ extends to $Q$}.


\begin{figure}[b]
    \centering
    \includegraphics[scale=.7]{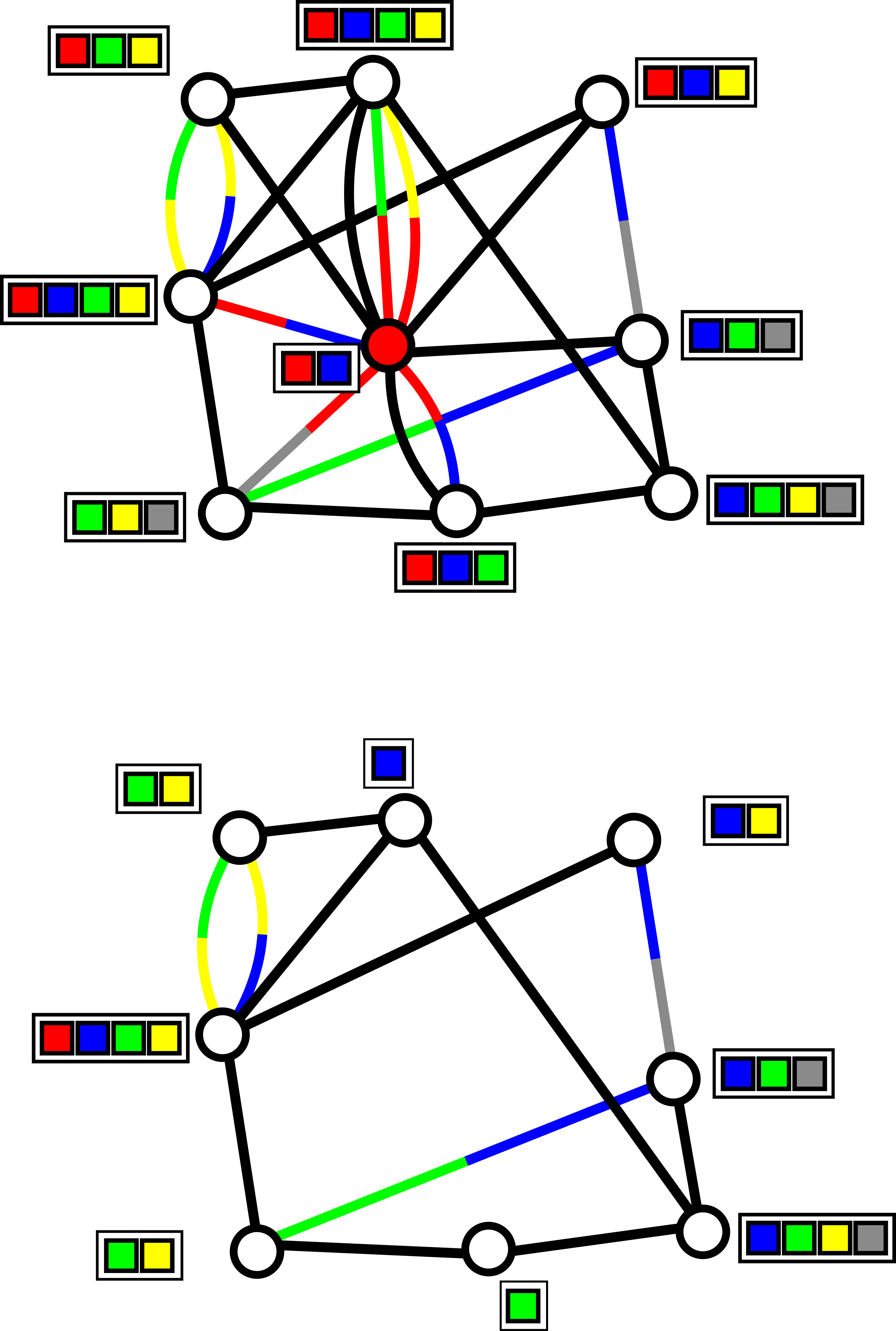}
    \caption{An example of a colour deletion. At the top is a polar graph with a polar assignment, and we consider a colouring $\phi$ where the central vertex $v$ is coloured red. At the bottom is the resulting $(\{v\},\phi)$-colour deletion.}
    \label{cdeg}
\end{figure}

\begin{definition}[colour deletion]
\label{full colour deletion def}
Let $Q$ be a polar graph with polar assignment~$\Gamma$, and let $\phi$ be a function such that $\phi|_{V'}$ is a $\Gamma$-colouring of $Q[V']$ for some $V' \subseteq V(Q)$. The {\em $(V',\phi)$-colour deletion} from $(Q,\Gamma)$ is the pair $(Q',\Gamma')$ such that
    $Q'=Q-V'$ and $\Gamma' = (L',R|Q')$ where $L'$ is the list assignment on $V(Q')$ such that, for each $u \in V(Q')$, we have $L'(u) = L(u) \setminus C_u$ with
    \begin{multline*}
        C_u = \{c \in L(u) : \textrm{there exists } e=uv  \in E(Q) \textrm{ with } v \in V' \textrm{ such that either } \\
            (e \notin F(Q) \textrm{ and } \phi(v)=c) \textrm { or } (e \in F(Q) \textrm{ and }R(e,u,v)=(c,\phi(v)))\}).
    \end{multline*}
\end{definition}

An example of a colour deletion is given in \cref{cdeg}.

If $e=uv \in E(Q)$ for some $u \in V(Q')$ and $v \in V'$ such that either $(e \notin F(Q)$ and $\phi(v)=c)$ or $(e \in F(Q)$ and $R(e,u,v)=(c,\phi(v)))$,
then we say that $e$ {\em removes} $c$ from $L(u)$. %
Note that each edge $uv$ with $u \in V(Q')$ and $v \in V'$ can remove at most one colour from $L(u)$. %
Therefore, after a colour deletion, the number of colours that are removed from the list of a vertex $u \in V(Q')$ is at most the number of edges between $u$ and vertices in $V'$. %

The next lemma demonstrates why colour deletions are useful;
it is used freely without reference.  

\begin{lemma}
    Let $Q$ be a polar graph with polar assignment $\Gamma$.
    Let $(V_1,V_2)$ be a partition of $V(Q)$, and let $\phi : V(Q) \rightarrow \mathbb{N}$ be a function such that $\phi|_{V_1}$ is a $\Gamma$-colouring of $Q[V_1]$.
    Let $(Q_2,\Gamma_2)$ be the $(V_1,\phi|_{V_1})$-colour deletion from $(Q,\Gamma)$.
    Then $\phi$ is a $\Gamma$-colouring of $Q$ if and only if $\phi|_{V_2}$ is a $\Gamma_2$-colouring of $Q_2$.
\end{lemma}
\begin{proof}
    Let $\Gamma = (L,R)$ and $\Gamma_2 = (L_2,R_2)$.
    Assume first that $\phi|_{V_2}$ is a $\Gamma_2$-colouring of $Q_2$.
    For any $v \in V_1$, we have that $\phi(v) \in L(v)$, and for any $v \in V_2$ we have $\phi(v) \in L_2(v) \subseteq L(v)$.
    Therefore $\phi$ satisfies \ref{Gamma 1}.
    If $e$ is an edge with both ends in $V_1$ or both ends in $V_2$, then since $\phi|_{V_1}$ is a $\Gamma$-colouring of $Q[V_1]$, and $\phi|_{V_2}$ is a $\Gamma_2$-colouring, \ref{Gamma 2} will be satisfied for the edge $e$.
    Suppose $e=uv$ where $v \in V_1$ and $u \in V_2$.
    First assume that $e$ is not polarised.
    In this case $\phi(v) \notin L_2(u)$ by \cref{full colour deletion def}, but $\phi(u) \in L_2(u)$, so $\phi(u) \neq \phi(v)$.
    So \ref{Gamma 2} is satisfied for such a non-polarised edge.
    Now assume that $e$ is polarised and let $R(e,u,v)=(c,d)$.
    Clearly \ref{Gamma 2} is satisfied for $e$ if $\phi(v) \neq d$.
    On the other hand, if $\phi(v)=d$, then $c \notin L_2(u)$ by \cref{full colour deletion def}, so $\phi(u) \neq c$, and \ref{Gamma 2} is satisfied for $e$.
    Therefore $\phi$ is a $\Gamma$-colouring of $Q$.

    Next assume that $\phi$ is a $\Gamma$-colouring of $Q$.
    Clearly $\phi|_{V_2}$ will satisfy \ref{Gamma 2} on every edge of $Q_2$.
    Hence we need only check that $\phi|_{V_2}$ satisfies \ref{Gamma 1}, that is $\phi(u) \in L_2(u)$ for all $u \in V(Q_2)$.
    Suppose that $\phi(u) \notin L_2(u)$ for some $u \in V(Q_2)$.
    Then, by \cref{full colour deletion def}, $\phi(u) \in C_u$, that is, there exists $e=uv \in E(Q)$ with $v \in V_1$ such that either $e \notin F(Q)$ and $\phi(v)=\phi(u)$, or $e\in F(Q)$ and $R(e,u,v)=(\phi(u),\phi(v))$. %
    In either case, this violates \ref{Gamma 2}, contradicting that $\phi$ is a $\Gamma$-colouring of $Q$.
    We deduce that $\phi(u) \in L_2(u)$ for all $u \in V(Q_2)$, and it follows that $\phi|_{V_2}$ is a $\Gamma_2$-colouring of $Q_2$, as required.
\end{proof}

When performing a colour deletion for a single vertex, we write ``the $(v,c)$-colour deletion from $(Q,\Gamma)$'' to mean ``the $(\{v\},\phi)$-colour deletion from $(Q,\Gamma)$ where $\phi(v)=c$''.
It is easy to see that a $(V',\phi)$-colour deletion is equivalent to a sequence of $(v,\phi(v))$-colour deletions for each $v\in V'$.

Recall, from the introduction, the notion of a $k$-restricted vertex in a graph.  We now define $k$-restricted vertices in polar graphs.

\begin{definition}[restricted vertex]
\label{restricted vertex def}
Let $Q$ be a polar graph with polar assignment~$\Gamma=(L,R)$. Let $v \in V(Q)$ and let $c \in L(v)$ such that the $(v,c)$-colour deletion from $(Q,\Gamma)$ is not colourable. Then we say that $v$ is {\em $\Gamma$-restricted on $c$ in $Q$}.
We also say that $v$ is a \emph{$\Gamma$-restricted vertex} of $Q$.
\end{definition}

Let $Q$ be a polar graph.
If there exists a $k$-polar (or $f$-polar) assignment~$\Gamma$ such that $Q$ has a $\Gamma$-restricted vertex $v$, then we say that $Q$ is \emph{$k$-restricted} (or \emph{$f$-restricted}, respectively), and that $v$ is a \emph{$k$-restricted vertex} in $Q$ (or an \emph{$f$-restricted vertex} in $Q$, respectively).
For a polar assignment~$\Gamma$, we say that $Q$ is {\em $\Gamma$-restricted} if $Q$ has a $\Gamma$-restricted vertex. %
We say that $Q$ is {\em $k$-unrestricted} (respectively, {\em $f$-unrestricted}) if, for every $k$-polar assignment $\Gamma$ (respectively, $f$-polar assignment $\Gamma$), we have that $Q$ is not $\Gamma$-restricted. %
If $G$ is a graph such that every polar graph $Q$ with $G(Q)=G$ is $k$-unrestricted (respectively, $f$-unrestricted), then we say that $G$ is {\em $k$-polar-unrestricted} (respectively {\em $f$-polar-unrestricted}).

Let $Q$ be a polar graph with polar assignment $\Gamma=(L,R)$, and let $v$ be a vertex in $V(Q)$. We define
$$M_{Q,\Gamma}(v) = \{c \in L(v) : v \textrm{ is $\Gamma$-restricted on $c$ in $Q$}\}.$$
We define the {\em $k$-restriction index} of a vertex $v$, denoted $m_{Q,k}(v)$, to be the maximum $|M_{Q,\Gamma}(v)|$ attains over all $k$-polar assignments $\Gamma$ for $Q$. We define the {\em $f$-restriction index} of a vertex $v$, denoted $m_{Q,f}(v)$, to be the maximum $|M_{Q,f}(v)|$ attains over all $f$-polar assignments $\Gamma$ for $Q$.

It will be useful to manipulate the colours that certain vertices are restricted on.  This is possible with the aid of the next \lcnamecref{switching}.
We first introduce some terminology.
Let $Q$ be a polar graph, let $f : V(Q) \rightarrow \mathbb{N}$, and let $\Gamma = (L,R)$ be an $f$-polar assignment for $Q$.
We say that $c \in \mathbb{N}$ is \emph{new for $\Gamma$} if $c \notin L(v)$ for all $v \in V(Q)$, and $c \notin R(e,u,v)$ for all $e=uv \in F(Q)$.
Suppose $c,c' \in \mathbb{N}$ such that $c'$ is new for $\Gamma$, but $c$ is not new for $\Gamma$.
We say that the $f$-polar assignment $\Gamma'=(L',R')$ is obtained from the \emph{colour swap of $c$ and $c'$ in $\Gamma$} when
    \[
    L_c(v) = \begin{cases}
        L(v) \setminus \{c\} \cup \{c'\} & \mbox{if $c \in L(v)$,}\\
        L(v)& \mbox{otherwise, and}
    \end{cases}
    \]
    \[
    R_c(e,u,v) = \begin{cases}
        (c',c') & \mbox{if $R(e,u,v)=(c,c)$,}\\
        (c',c_v)& \mbox{if $R(e, u,v)=(c,c_v)$ where $c_v \neq c$,}\\
        (c_u,c')& \mbox{if $R(e,u,v)=(c_u,c)$ where $c_u \neq c$,}\\
        R(e,u,v) & \mbox{otherwise.}
    \end{cases}
    \]

\begin{lemma}
\label{switching}
Let $Q$ be a polar graph, let $f : V(Q) \rightarrow \mathbb{N}$, and let $r \in V(Q)$. Let $C \subseteq \mathbb{N}$ with $|C|=m_{Q,f}(r)$. Then there exists an $f$-polar assignment $\Gamma_C$ for $Q$ such that $M_{Q,\Gamma_C}(r)=C$.
\end{lemma}
\begin{proof}
    Let $\Gamma=(L,R)$ be an $f$-polar assignment such that $|M_{Q,\Gamma}(r)|=m_{Q,f}(r)$.
    Let $\Gamma_c=(L_c,R_c)$ be obtained from the colour swap of $c$ and $c'$ in $\Gamma$.
    Observe that since $c' \notin L(v)$ for any $v \in V(Q)$, we have that $\Gamma_c$ is an $f$-polar assignment for $Q$, and $c$ is new for $\Gamma_c$. 
    Note also that there is a bijection $\sigma$ between $\Gamma$-colourings of $Q$, and $\Gamma_c$-colourings of $Q$, that exchanges any $c$ for $c'$. 
    It follows that $c' \in M_{Q,\Gamma_c}(v)$ if and only if $c \in M_{Q,\Gamma}(v)$. 

    We now describe how to construct $\Gamma_C$.
    First, for each $c \in C$ such that $c$ is not new for $\Gamma$, choose some new $c' \notin C$, and perform the above operation.
    Having done this, we obtain a polar assignment $\Gamma'$ for $Q$ such that every $c \in C$ is new for $\Gamma'$. 
    This allows us to perform the above operation $|C|$ more times, each time exchanging a colour in $C$ with a colour in $M_{Q,\Gamma'}(r) \setminus C$, until we obtain $\Gamma_C$ where $M_{Q,\Gamma_C}(r)=C$, as required.
\end{proof}

The next two lemmas are simple, but are useful throughout for dealing with polarised edges or parallel edges.
For vertices $u$ and $v$ in a polar graph $Q$, we write $e_Q(u,v)$ to denote the number of edges that join $u$ and $v$, that is, $e_Q(u,v) = |\{e \in E(Q) : e=uv\}|$.  When $Q$ is clear from context, we write $e(u,v)$ rather than $e_Q(u,v)$.

\begin{lemma}
\label{cdlemma1}
Let $Q$ be a polar graph with polar assignment $\Gamma=(L,R)$, let $v \in V(Q)$, and let $e=uv \in F(Q)$.
For distinct $a,b \in L(v)$, let $(Q-v,\Gamma_a)$ and $(Q-v,\Gamma_b)$ be the $(v,a)$- and $(v,b)$-colour deletions from $(Q,\Gamma)$ respectively, where $\Gamma_a = (L_a,R_a)$ and $\Gamma_b = (L_b,R_b)$.
Then either $|L_a(u)| > |L(u)|-e(u,v)$ or $|L_b(u)| > |L(u)|-e(u,v)$.
\end{lemma}
\begin{proof}
    Let $R(e,u,v)=(c_u,c_v)$.
    By \cref{full colour deletion def}, $|L_a(u)| \geq |L(u)|-e(u,v)$ and $|L_b(u)| \geq |L(u)|-e(u,v)$, as each edge between $u$ and $v$ removes at most one colour from $L(u)$.
    If $c_v \neq a$, then $e$ does not remove any colours from $L(u)$ in the $(v,a)$-colour deletion. Similarly, if $c_v \neq b$, then $e$ does not remove any colours from $L(u)$ in the $(v,b)$-colour deletion. Since $a$ and $b$ are distinct, 
    either $|L_a(u)| > |L(u)|-e(u,v)$ or $|L_b(u)| > |L(u)|-e(u,v)$, as required.
\end{proof}

\begin{corollary}
\label{cdlemma2}
Let $Q$ be a polar graph with polar assignment $\Gamma=(L,R)$.
Let $u,v \in V(Q)$ such that $e(u,v) \ge 2$.
For distinct $a,b \in L(v)$, let $(Q-v,\Gamma_a)$ and $(Q-v,\Gamma_b)$ be the $(v,a)$- and $(v,b)$-colour deletions from $(Q,\Gamma)$ respectively, where $\Gamma_a = (L_a,R_a)$ and $\Gamma_b = (L_b,R_b)$.
Then either $|L_a(u)|>|L(u)|-e(u,v)$ or $|L_b(u)|>|L(u)|-e(u,v)$.
\end{corollary}
\begin{proof}
    If any edge between $u$ and $v$ is polar, then we obtain the result by \cref{cdlemma1}. If no edges between $u$ and $v$ are polar, then $L_a(u)=L(u) \setminus \{a\}$, but since $e(u,v)>1$, we have that $|L_a(u)|>|L(u)|-e(u,v)$, as required.
\end{proof}

\subsection{\texorpdfstring{$1$}{1}-joins and polar assignment unions}

We now extend the notions of graph union and list assignment union to polar graphs and polar assignments.
We also describe how to construct a polar graph that is not $k$-choosable from polar graphs that are $k$-restricted, and how to construct a polar graph with a $k$-polar assignment that simulates a polar graph with an $f$-assignment, for a function $f$ bounded by $k$.

We first define the union of two polar graphs.  For our purposes, it is sufficient to define this operation for polar graphs with no edges in common.
Let $Q_0$ and $Q_1$ be two polar graphs with $E(Q_0) \cap E(Q_1) = \emptyset$.
The {\em union} of $Q_0$ and $Q_1$, denoted $Q_0 \cup Q_1$, is the polar graph with $V(Q_0 \cup Q_1)=V(Q_0) \cup V(Q_1)$, $E(Q_0 \cup Q_1)=E(Q_0) \cup E(Q_1)$ and $F(Q_0 \cup Q_1)=F(Q_0) \cup F(Q_1)$.
In particular, when $Q_0$ and $Q_1$ are polar graphs such that $E(Q_0) \cap E(Q_1) = \emptyset$ and $|V(Q_0) \cap V(Q_1)| = 1$, then we say that $Q_0 \cup Q_1$ is the \emph{$1$-join} of $Q_0$ and $Q_1$.

We omit the straightforward proof of the next lemma.
\begin{lemma}
\label{1-joins unrestricted}
    Let $Q_0$ and $Q_1$ be $k$-unrestricted polar graphs with $E(Q_0) \cap E(Q_1) = \emptyset$ and $|V(Q_0) \cap V(Q_1)| = 1$.
    Then the $1$-join of $Q_0$ and $Q_1$ is $k$-unrestricted.
\end{lemma}

\begin{lemma}
\label{1joinscu}
    Let $Q$ be a $k$-restricted polar graph with $|V(Q)| \ge 3$. Then $Q$ contains a $2$-connected polar subgraph that is $k$-restricted.
\end{lemma}
\begin{proof}
    Assume the result is false, and let $Q$ have the minimum number of vertices among all counterexamples.
    Then $Q$ is not $2$-connected, so is either disconnected or a $1$-join of two non-null connected graphs.
    If $Q$ is disconnected, then some component of $Q$ is $k$-restricted, but not $2$-connected, contradicting that $Q$ is minimum-sized.
    So $Q$ is a $1$-join of two non-null connected graphs, $Q_1$ and $Q_2$, each with strictly fewer vertices than $Q$.
    If $Q_i$ is $k$-restricted for some $i \in \{0,1\}$, then it is $2$-connected, since $Q$ is minimum-sized, contradicting that $Q$ is a counterexample.
    Therefore both $Q_0$ and $Q_1$ are $k$-unrestricted, so $Q$ is $k$-unrestricted by \cref{1-joins unrestricted}, a contradiction.
\end{proof}

Usually, we will be dealing with polar graphs with associated polar assignments.  When taking the union of two such polar graphs, there is a natural polar assignment for their union, as we now describe.

\begin{definition}[polar assignment union]
Let $Q_0$ and $Q_1$ be polar graphs with polar assignments $\Gamma_0=(L_0,R_0)$ and $\Gamma_1=(L_1,R_1)$ respectively, such that $E(Q_0) \cap E(Q_1) = \emptyset$.
Let $Q = Q_0 \cup Q_1$.
We obtain a polar assignment $\Gamma=(L,R)$ for $Q$, which we call the \emph{polar assignment union} and denote $\Gamma_0 \cup \Gamma_1$, by defining $L$ to be the list assignment union of $L_0$ and $L_1$,
and $R(e,u,v)=\begin{cases}
    R_0(e,u,v) & \mbox{if $e \in E(Q_0)$}\\
    R_1(e,u,v) & \mbox{if $e \in E(Q_1)$}\\
\end{cases}.$
\end{definition}

It is easily seen that the union of polar graphs, and the union of polar assignments, are both commutative and associative.
Thus, for a set of polar graphs $\mathcal{Q}$ (with pairwise disjoint edge sets), or a set of polar assignments $\mathcal{L}$, we can unambiguously write $\bigcup_{Q \in \mathcal{Q}} Q$ or $\bigcup_{\Gamma \in \mathcal{L}} \Gamma$ respectively.

Note that if $\mathcal{Q}$ consists of the components or blocks of a polar graph $Q$, then the polar graphs in $\mathcal{Q}$ have pairwise disjoint edge sets, and $Q$ is the union of $\mathcal{Q}$.
Thus, one way to obtain a polar assignment is to start with a polar assignment for the connected components or blocks, and taking the union (for example, the list assignment for the graph in \cref{figblockcounter} can be obtained 
in this way, see \cref{figassignunion}).

\begin{figure}
    \centering
    \includegraphics[scale=0.7]{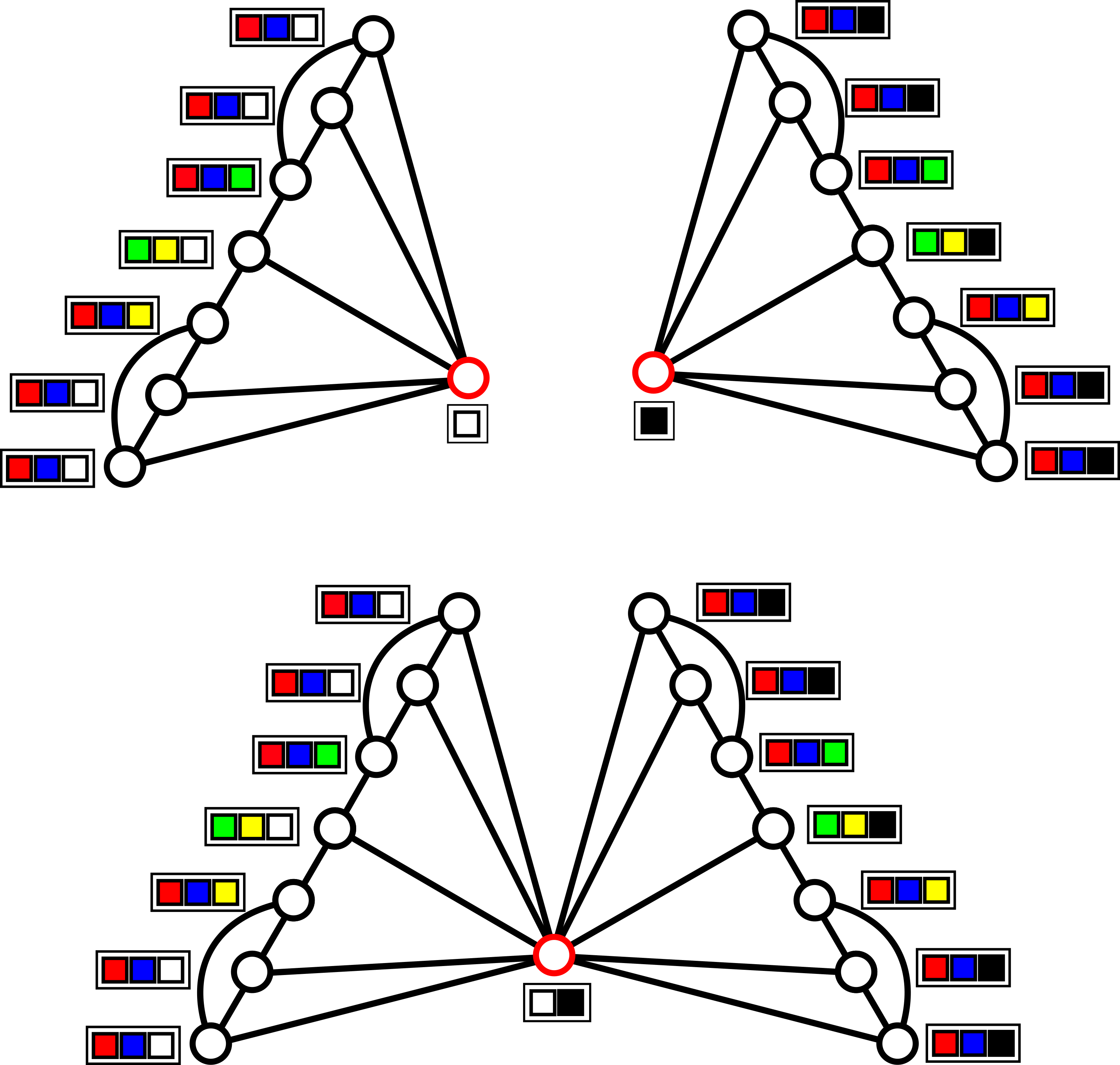}
    \caption{An example of a polar assignment union of two copies of a polar graph we call $T_3$ (with no polarised edges), that have a single vertex in common, shown in red.} 
    \label{figassignunion}
\end{figure}

We again omit the straightforward proof of the following lemma.

\begin{lemma}
\label{1-joins bad}
Let $Q$ be the 1-join of polar graphs $Q_0$ and $Q_1$. Let $\Gamma_0$ be a polar assignment for $Q_0$ and let $\Gamma_1$ be a polar assignment for $Q_1$ such that both $(Q_0,\Gamma_0)$ and $(Q_1,\Gamma_1)$ are not colourable. Then $Q$ is not $(\Gamma_0 \cup \Gamma_1)$-colourable.
\end{lemma}

\Cref{1-joins bad}, combined with \cref{switching}, allows us to construct polar graphs that are not $k$-choosable from $k$-restricted polar graphs that are $k$-choosable, using $1$-joins.
We now describe this construction.
Let $Q$ be a $k$-restricted polar graph, so there exists $v \in V(Q)$ with $m_{Q,k}(v) > 0$.
For simplicity, we assume $m_{Q,k}(v)=1$.
Let $\mathcal{Q}=\{Q_1,Q_2,\dotsc,Q_k\}$ be a set consisting of $k$ polar graphs isomorphic to $Q$, such that for distinct $i,j \in [k]$, we have $V(Q_i) \cap V(Q_j)=\{v\}$ and $E(Q_i) \cap E(Q_j) = \emptyset$.
%
For each $i \in [k]$, there exists a $k$-polar assignment~$\Gamma_i^+$ for $Q_i$ such that $m_{Q_i,\Gamma_i^+}(v)=1$.
By \cref{switching}, we may also assume that for distinct $i,j \in [k]$, we have $M_{Q_i,\Gamma_i^+}(v) \cap M_{Q_j,\Gamma_j^+}(v)=\emptyset$.
For each $i \in [k]$, let $\Gamma_i^+=(L_i^+,R_i)$, and let $\Gamma_i$ be the polar assignment $(L_i,R_i)$ where $L_i(u) = \begin{cases}
    L_i^+(u) & \mbox{if $u \in V(Q_i) \setminus \{v\}$}\\
        M_{Q_i,\Gamma_i^+}(u) & \mbox{if $u=v$.}
    \end{cases}$.
    Now, it is easily seen that, for each $i \in [k]$, the polar graph $Q_i$ is not $\Gamma_i$-colourable, and $|L_i(v)|=1$.
    Let $Q' = \bigcup_{i \in [k]} Q_i$, 
    and let $\Gamma'=(L',R')$ be the polar assignment union $\bigcup_{i \in [k]}\Gamma_t$.
    It follows from \cref{1-joins bad} that $Q'$ is not $\Gamma'$-colourable.
    Moreover, $L'$ is a $k$-list assignment, so $\Gamma'$ certifies that $Q'$ is not $k$-choosable.

    A similar technique can be employed to construct polar graphs with $k$-polar assignments that simulate polar graphs with $f$-polar assignments, when $f$ is bounded by $k$.  We describe this construction next, before giving an example.

\begin{definition}[$(Q',f,Q,v,k)$-top-up]
    \label{topupdef}
    Let $Q$ and $Q'$ be polar graphs, let $v \in V(Q)$, let $k \in \mathbb{N}$, and let $f:V(Q') \rightarrow \mathbb{N}$ be a function such that $f(u) \leq k$ for all $u \in V(Q')$.
    For each $u \in V(Q')$, let $\mathcal{Q}_u$ be a set of $k-f(u)$ polar graphs isomorphic to $Q$, each of which has the vertex $v$ labelled by $u$, but is otherwise vertex disjoint with each other graph in $\mathcal{Q}_u$, and with $Q'$. 
    Moreover, whenever $u$ and $u'$ are distinct vertices in $V(Q')$, then $Q_u$ and $Q_{u'}$ are disjoint for any $Q_u \in \mathcal{Q}_u$ and $Q_{u'} \in \mathcal{Q}_{u'}$.
    Let
    \[
    \mathcal{U}=\{Q'\} \cup \bigcup_{u \in V(Q')} \mathcal{Q}_u.
    \]
    Then the {\em $(Q',f,Q,v,k)$-top-up} is $\bigcup_{U \in \mathcal{U}} U$.
\end{definition}

    Note that, although we have defined this operation for polar graphs $Q$ and $Q'$, if neither $Q$ nor $Q'$ has any polarised edges, then the $(Q',f,Q,v,k)$-top-up has no polarised edges.  Thus, we can also consider, as a special case, a top-up operation for graphs.

As an example, the polar graph in \cref{ncp2fig} is a $(Q_6,f,T_3,u,3)$-top-up, where $Q_6$ and $f$ are given in \cref{ncp1fig}, and $T_3$ is given in \cref{figassignunion} (where the red vertex is $u$). The next lemma shows that the polar graph in \cref{ncp2fig} is not $3$-choosable.


\begin{lemma}
\label{topuplemma}
Let $k \in \mathbb{N}$, let $Q$ and $Q'$ be polar graphs where $Q$ has $v \in V(Q)$ such that $m_{Q,k}(v)=1$, and let $f:V(Q')\rightarrow \mathbb{N}$ be a function such that $f(u)\leq k$ for all $u \in V(Q')$.
Then the $(Q',f,Q,v,k)$-top-up is $k$-choosable if and only if $Q'$ is $f$-choosable.
\end{lemma}
\begin{proof}
    Let $Q^+$ be the $(Q',f,Q,v,k)$-top-up.
    First assume that $Q'$ is $f$-choosable and let $\Gamma=(L,R)$ be a $k$-polar assignment for $Q^+$.
    We seek a $\Gamma$-colouring of $Q'$ that will extend to a $\Gamma$-colouring of $Q^+$.
    Let $u \in V(Q')$.
    Since $m_{Q_u,k}(u)=1$ for each copy $Q_u$ of $Q$ containing $u$, there is at most one $c \in L(u)$ such that a $\Gamma$-colouring $\phi$ of $Q'$ where $\phi(u)=c$ does not extend to a $\Gamma$-colouring of $Q^+[V(Q') \cup V(Q_u)]$.
    When such a $c$ exists, if we remove $c$ from $L(u)$, then any colouring with the remaining polar assignment will extend to $Q^+[V(Q') \cup V(Q_u)]$.
    Therefore, we can remove at most $k-f(u)$ colours from $L(u)$, for each $u \in V(Q')$, and obtain a polar assignment $\Gamma_f=(L_f,R_f)$ such that any $\Gamma_f$-colouring of $Q'$ extends to a $\Gamma$-colouring of $Q^+$.
    Now $|L_f(u)| \ge k - (k-f(u)) = f(u)$ for each $u$,
    and therefore, since $Q'$ is $f$-choosable, $Q'$ is $\Gamma_f$-colourable, and so $Q^+$ is $\Gamma$-colourable.

    Next assume that $Q'$ is not $f$-choosable, and let $\Gamma_f=(L_f,R_f)$ be an $f$-polar assignment for $Q'$ such that $Q'$ is not $\Gamma_f$-colourable.
    We will construct a $k$-polar assignment~$\Gamma^+$ for $Q^+$ such that $Q^+$ is not $\Gamma^+$-colourable.
    First, for each copy $Q_{u,i}$ of $Q$ containing $u \in V(Q')$ for $i \in [k-f(u)]$, let $\Gamma_{u,i}'=(L_{u,i}',R_{u,i})$ be a $k$-polar assignment for $Q_{u,i}$ such that $|M_{Q_{u,i},\Gamma_{u,i}'}(u)|=1$ and the sets in $$\{M_{Q_{u,i},\Gamma_{u,i}'}(u): u \in V(Q') \textrm{ and } i \in [k-f(u)]\} \cup \{L_f(u)\}$$ are pairwise disjoint; this is possible by \cref{switching}.
    Then, let $\Gamma_{u,i}=(L_{u,i},R_{u,i})$ be the polar assignment for $Q_{u,i}$ where $L_{u,i}$ is obtained from $L_{u,i}'$ by removing all colours from $L_{u,i}'(u)$ except those in $M_{Q_{u,i},\Gamma_{u,i}'}(u)$.
    So, for each $u \in V(Q')$ and $i \in [k-f(u)]$, we have $L_{u,i}(u)=M_{Q_{u,i},\Gamma_{u,i}'}(u)$, which is a singleton, since $m_{Q_{u,i},k}(u)=1$.
    However, since $u$ is $\Gamma_{u,i}'$-restricted on the colour in $L_{u,i}(u)$, we see that $Q_{u,i}$ is not $\Gamma_{u,i}$-colourable.
    Now $Q^+$ can be obtained by a sequence of $1$-joins on $Q'$ and the $Q_{u,i}$'s. %
    Let $\Gamma^+=(L^+,R^+)$ be the polar assignment union of $\Gamma_f$ and each member of $\{\Gamma_{u,i} : u \in V(Q') \textrm{ and } i \in [k-f(u)]\}$. %
    Then $\Gamma^+$ is a $k$-polar assignment, and it follows from \cref{1-joins bad} that $Q^+$ is not $\Gamma^+$-colourable.
    So $Q^+$ is not $k$-choosable.
\end{proof}

Finally, for polar graphs $Q_0$ and $Q_1$ with polar assignments $\Gamma_0=(L_0,R_0)$ and $\Gamma_1=(L_1,R_1)$ respectively, we say that $\Gamma_0$ and $\Gamma_1$ are \emph{conflicting} (or \emph{non-conflicting}) if 
$L_0$ and $L_1$ are conflicting (or non-conflicting, respectively).

\section{Degree choosability for polar graphs}
\label{dcpolarsec}

Recall the notion of degree-choosability, as seen in \cref{dcsec}.
A polar assignment~$(L,R)$ for a polar graph $Q$ is a {\em degree-polar assignment} if $|L(u)| \geq d(u)$ for all $u \in V(Q)$.
We say a polar graph $Q$ is {\em degree-choosable} if it is $\Gamma$-colourable for every degree-polar assignment $\Gamma$.
\Cref{dcthm,baddegreeassign} characterise the graphs that are not degree-choosable and the degree-list assignments for which they are not colourable.
In this section we generalise these results to polar graphs.

Let $Q$ be a polar graph.
If the underlying graph of $Q$ is a Gallai tree, and every polarised edge of $Q$ belongs to a $K_2$-block, then we call $Q$ a {\em polar Gallai tree}.
We say that a $K_2$-block of $Q$ is a \emph{polar $K_2$-block} if the unique edge in the block is polarised.
For a polar $K_2$-block $B$, we say that a $1$-polar assignment $\Gamma_B=(L_B,R_B)$ for $B$ is \emph{polar conforming} if $R_B(e,u,v)=(c_u,c_v)$ where $L_B(u)=\{c_u\}$ and $L_B(v)=\{c_v\}$.

A degree-polar assignment $\Gamma=(L,R)$ for $Q$ is \emph{bad} if $Q$ is a polar Gallai tree and $\Gamma$ can be obtained as the union, taken over each block $B$ of $Q$, of uniform $d(B)$-polar assignments whenever the block is not a polar $K_2$-block, and of polar-conforming $1$-polar assignments otherwise, that are pairwise non-conflicting.
Intuitively, a bad degree-polar assignment is a bad degree-list assignment for a polar Gallai tree except that, for each polarised edge (which must appear in a $K_2$-block), the colours are given by $R$ (rather than some uniform colouring for the $K_2$-block).
For an example, see \cref{baddpa}.
It is easily seen that every polar Gallai tree has a bad degree-polar assignment.

\begin{figure}
  \centering
  \includegraphics[scale=0.7]{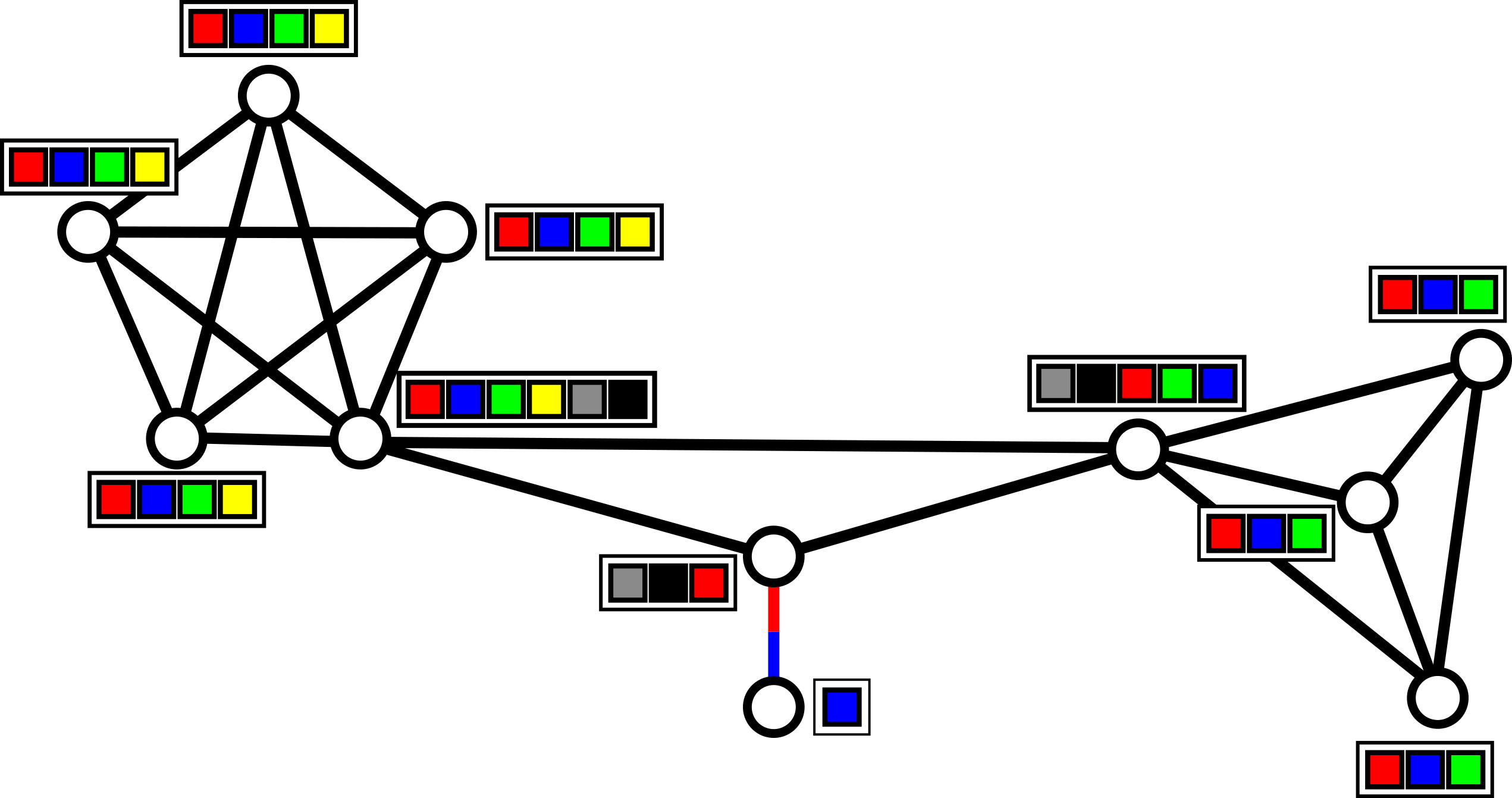}
  \caption{A polar Gallai tree with a bad degree-polar assignment.} 
  \label{baddpa}
\end{figure}

Our main result of this section is the following:

\begin{theorem}
\label{dcpolarthm}
    Let $Q$ be a connected polar graph with degree-polar assignment $\Gamma$.
    Then $Q$ is not $\Gamma$-colourable if and only if $Q$ is a polar Gallai tree and $\Gamma$ is bad.
\end{theorem}

Our proof approach is broadly similar to that used by Erd\H os, Rubin, and Taylor to prove \cref{dcthm}.
We start with some simple observations.
\begin{lemma}
    \label{dcponedirection}
    Let $Q$ be a polar Gallai tree.
    If $\Gamma$ is a bad degree-polar assignment for $Q$, then $Q$ is not $\Gamma$-colourable.
\end{lemma}
\begin{proof}
    Let $\Gamma$ be a bad degree-polar assignment for $Q$.
    By definition, $\Gamma$ is a union of bad degree-polar assignments $\Gamma_B$ for each block $B$ of $Q$.  It is easily checked that $B$ is not $\Gamma_B$-colourable for any such $B$.  The lemma then follows from \cref{1-joins bad}.
\end{proof}


\begin{lemma}
    \label{polar paths}
    Let $Q$ be a non-null polar graph whose underlying graph is a tree, and let $\Gamma=(L,R)$ be a degree-polar assignment for $Q$. If $|L(v)| > d(v)$ for some $v \in V(Q)$, then $Q$ is $\Gamma$-colourable.
\end{lemma}
\begin{proof}
    Let $v_1 \in V(Q)$ such that $|L(v_1)| > d(v_1)$, and let $G=G(Q)$.
    We let $(v_1,v_2,\dotsc,v_n)$ be a search ordering of $G$ starting at $v_1$, that is, an ordering of $V(G)$ such that $G[\{v_1,\dotsc,v_i\}]$ is connected for all $i \in [n]$.
    We greedily construct a $\Gamma$-colouring for $Q$ starting at $v_n$, then $v_{n-1}$, and so on, where, for each $i \ge 2$, as $v_i$ has a neighbour in $\{v_1,\dotsc,v_{i-1}\}$, and $|L(v_i)| \ge d_G(v_i) > d_{G[v_i,v_{i+1},\dotsc,v_n]}(v_i)$, there is always at least one colour in $L(v_i)$ that we can choose for $v_i$ while maintaining a $\Gamma$-colouring for $Q[\{v_i,v_{i+1},\dotsc,v_n\}]$.
    Finally, we can also greedily colour $v_1$, since $|L(v_1)| > d(v_1)$, thereby obtaining a $\Gamma$-colouring of $Q$.
\end{proof}

\begin{lemma}
\label{restriction subgraphs}
    Let $Q$ be a polar graph.  If $Q$ has an induced polar subgraph that is degree-choosable, then $Q$ is degree-choosable.
\end{lemma}
\begin{proof}
    Suppose $Q$ has an induced polar subgraph $Q'$ that is degree-choosable.
    Our proof is by induction on $t=|V(Q)| - |V(Q')|$.
    Clearly the result holds when $t=0$.
    Suppose $t \ge 1$, and let $\Gamma = (L,R)$ be a degree-list assignment for $Q$.
    Then there exists $v \in V(Q) \setminus V(Q')$.
    We choose any $c \in L(v)$.  Then, considering the $(v,c)$-colour deletion from $Q$, we obtain a polar graph that still contains $Q'$ as an induced polar subgraph, together with a degree-polar assignment.  The lemma follows by induction.
\end{proof}

We next consider the case where the underlying graph is a Gallai tree consisting of a single block.

\begin{lemma}
\label{polar cycles+}
Let $Q$ be a simple non-null polar graph with a degree-polar assignment~$\Gamma$ that is not bad. If $Q$ is cycle or a complete graph, then $Q$ is $\Gamma$-colourable
\end{lemma}
\begin{proof}
    Let $\Gamma=(L,R)$ and assume that $\Gamma$ is not bad.
    First, let $Q$ be a cycle.
    If $F(Q) = \emptyset$, then the result follows from \cref{dcthm,baddegreeassign}.
    So assume that $Q$ contains a polarised edge $e=uv$ with $R(e,u,v)=(c_u,c_v)$.
    Since $d(u)=2$ and $\Gamma$ is a degree-polar assignment, $|L(u)|\geq 2$, and therefore there is some $c_u' \in L(u) \setminus \{c_u\}$.
    Let $(Q-u,\Gamma')$ be the $(u,c_u')$-colour deletion from $(Q,\Gamma)$, and let $\Gamma'=(L',R')$.
    Note that since $c_u' \neq c_u$, we have that $L'(v)=L(v)$.
    Therefore, $\Gamma'$ is a degree-polar assignment for the path $Q-u$, but since $|L'(v)|>d_{Q-u}(v)$, we have that $Q-u$ is $\Gamma'$-colourable by \cref{polar paths}, and therefore $Q$ is $\Gamma$-colourable, as required.

    Next, assume that $G(Q) \cong K_2$, and let $V(Q)=\{u,v\}$.
    In this case $d(u)=d(v)=1$, so $|L(u)| \geq 1$ and $|L(v)| \geq 1$.
    Clearly $Q$ is $\Gamma$-colourable if either $|L(u)|>1$ or $|L(v)|>1$,
    so we may assume that $|L(u)|=|L(v)|=1$.
    If $e=uv$ is not polarised, then $\Gamma$ is not uniform (as $\Gamma$ is not bad), so $Q$ is $\Gamma$-colourable.
    On the other hand, if $e=uv$ is polarised, then
    $\Gamma$ is not polar conforming for the polar $K_2$-block $Q$, so again $Q$ is $\Gamma$-colourable.

    It remains only to consider the case where $Q \cong K_n$ for $n \geq 4$.
    We may assume that $Q$ contains a polarised edge $e=uv$, for otherwise the result follows from \cref{dcthm,baddegreeassign}.
    We proceed by induction on $n$, where the base case, $n=3$, holds by the foregoing.
    So assume $n \ge 4$ and let $\Gamma=(L,R)$ be a degree-polar assignment for $Q$.
    There exists a vertex $w \in V(Q) \setminus \{u,v\}$.
    Let $c_w \in L(w)$, and let $(Q-w,\Gamma')$ be the $(w,c_w)$-colour deletion from $(Q,\Gamma)$, with $\Gamma'=(L',R')$.
    Observe that $G(Q-w) \cong K_{n-1}$, and that $|L'(x)|\geq n-2$ for all $x \in V(G(Q-w))$, so $\Gamma'$ is a degree-polar assignment for $Q-w$, where $Q-w$ contains a polarised edge.
    The result follows by induction.
\end{proof}

A {\em theta graph} is a graph obtained by identifying the endpoints of three paths each with at least one edge and two with at least two edges.
We will use the following lemma, which appears in \cite{ERT1979}.

\begin{lemma}
\label{theorem R}
    Let $G$ be a $2$-connected graph that is not an odd cycle or a complete graph. Then $G$ contains an induced subgraph that is isomorphic to an even cycle or a theta graph.
\end{lemma}

The previous lemma motivates considering the case where the underlying graph is a theta graph.

\begin{lemma}
\label{even theta}
Let $Q$ be a polar graph with degree-polar assignment $\Gamma$. If $G(Q)$ is a theta graph, then $Q$ is $\Gamma$-colourable.
\end{lemma}
\begin{proof}
    Assume that $G(Q)$ is a theta graph, and let $\Gamma=(L,R)$.
    Let $u$ and $v$ be the distinct vertices of $Q$ with three internally disjoint paths between them.
    Since at most one of these paths has length 1, there exists $w \in V(Q) \setminus \{u,v\}$ adjacent to $u$.
    If $e=uw$ is not polarised, then, since $|L(u)| \geq d(u)=3$ and $|L(w)| \geq d(v)= 2$, there is some $c_u \in L(u)$ such that $|L(w) \setminus \{c_u\}| \geq 2$.
    Similarly, if $e$ is polarised and $R(e,u,w)=(c_u',c_w)$, then, since $|L(u)| \geq 3$, there is some $c_u \in L(u)$ with $c_u' \neq c_u$.
    Let $(Q-u,\Gamma')$ be the $(u,c_u)$-colour deletion from $(Q,\Gamma)$, with $\Gamma'=(L',R')$.
    Then $|L'(w)| \geq 2$, since if $e$ is not polarised, then $L'(w)=L(w) \setminus \{c_u\}$; and if $e$ is polarised, then $L'(w)=L(w)$.
    Therefore $Q-u$ is a tree and $|L'(w)| > d_{Q-u}(w)=1$, so $Q-u$ is $\Gamma'$-colourable by \cref{polar paths}, and therefore $Q$ is $\Gamma$-colourable.
\end{proof}

We require one more lemma that handles when $Q$ is not simple.

\begin{lemma}
\label{non simple}
Let $Q$ be a polar graph with degree-polar assignment $\Gamma$. If $Q$ is not simple, then $Q$ is $\Gamma$-colourable.
\end{lemma}
\begin{proof}
    By \cref{restriction subgraphs}, it suffices to consider the case where $|V(Q)|=2$.
    Let $V(Q) = \{u,v\}$ with $e(u,v)>1$. Then $d(u) \geq 2$, and so, by \cref{cdlemma2}, there is some $c \in L(v)$ such that the $(v,c)$-colour deletion from $Q$ leaves $u$ with more than $|L(u)|-e(u,v)$ colours. That is, there will be at least one colour left to colour $u$ with, so $Q$ is $\Gamma$-colourable.
\end{proof}

We now prove the main result of this section.

\begin{proof}[Proof of \cref{dcpolarthm}]
    One direction is given by \cref{dcponedirection}.
    For the other direction, we assume that either $Q$ is not a polar Gallai tree, or $\Gamma$ is not bad, and will show that $Q$ is $\Gamma$-colourable.
    If $Q$ is not simple, then it is $\Gamma$-colourable by \cref{non simple}.
    So we may assume that $Q$ is simple.
    If $G(Q)$ is not a Gallai tree, then it contains a block that is not a complete graph or an odd cycle.
    Therefore, by \cref{theorem R}, $G(Q)$ contains an even cycle or a theta graph as an induced subgraph.
    By \cref{restriction subgraphs,even theta,polar cycles+}, it follows that $Q$ is $\Gamma$-colourable, as required.

    It now suffices to show that if $G(Q)$ is a Gallai tree and $\Gamma$ is not bad, then $Q$ is $\Gamma$-colourable.
    We prove this by induction on the number of blocks in $Q$.
    If $Q$ has only one block, then this follows from \cref{polar cycles+}.
    So assume that $Q$ has at least two blocks.


    Let $\Gamma=(L,R)$, and let $B$ be a leaf block of $G(Q)$.
    Then $B$ contains a unique cut vertex $v$ of $Q$.
    We say that the leaf block $B$ is \emph{unambiguous (with respect to $\Gamma$)} if $B$ contains no polarised edges, $\Gamma|(B-v)$ is uniform, and $L(u) \subseteq L(v)$ for all $u \in V(B) \setminus \{v\}$.
    We say a leaf block is \emph{ambiguous} if it is not unambiguous.
    Observe that when $Q$ has an ambiguous leaf block with respect to $\Gamma$, then $\Gamma$ is not bad.

    First, we assume $B$ is an unambiguous leaf block.
    Let $Q' = Q - (V(B) \setminus \{v\})$, so $Q$ is the $1$-join of $B$ and $Q'$.
    Then there exist degree-polar assignments $\Gamma_B$ and $\Gamma_{Q'}$ for $B$ and $Q'$ respectively such that $\Gamma = \Gamma_B \cup \Gamma_{Q'}$ and $\Gamma_B$ is uniform.
    Note that $\Gamma_{Q'}$ is not bad, since $\Gamma_B$ is bad but $\Gamma$ is not bad.
    Now $B$ is either an odd cycle or a complete graph, and $|L(v)| > d_B(v)$, so $B$ is $\Gamma|B$-colourable by \cref{polar cycles+}.
    Let $\phi$ be a $\Gamma|B$-colouring of $B$.
    It follows that the $(B-v,\phi|_{V(B-v)})$-colour deletion from $(Q,\Gamma)$ is $(Q',\Gamma_{Q'})$.
    As $Q'$ is $\Gamma_{Q'}$-colourable, by induction, we see that $Q$ is $\Gamma$-colourable, as required.

    We may now assume that $Q$ has no unambiguous leaf blocks.
    As before, let $B$ be a leaf block of $G(Q)$, let $v$ be the cut vertex of $Q$ in $B$, let $Q' = Q - (V(B) \setminus \{v\})$, so $Q$ is the $1$-join of $B$ and $Q'$, and let $\Gamma=(L,R)$.
    Now $B$ is either an odd cycle or a complete graph, and $|L(v)| > d_B(v)$, so $B$ is $\Gamma|B$-colourable by \cref{polar cycles+}.
    Let $\phi$ be a $\Gamma|B$-colouring of $B$.
    Let $(Q',\Gamma')$ be the $(B-v,\phi|_{V(B-v)})$-colour deletion from $(Q,\Gamma)$, with $\Gamma'=(L',R')$.
    Then $\Gamma'$ is a degree-polar assignment for $Q'$.
    Let $B'$ be a leaf block of $Q$ that is distinct from $B$.
    If $B$ and $B'$ are the only two blocks of $Q$, then, as $B'$ is ambiguous with respect to $\Gamma$, we have that $\Gamma'$ is not bad for $Q'$.
    Otherwise, $B'$ is also a leaf block in $Q'$ and it is ambiguous with respect to $\Gamma'$ so, again, $\Gamma'$ is not bad for $Q'$.
    It follows that $Q'$ is $\Gamma_{Q'}$-colourable, by induction, and therefore $Q$ is $\Gamma$-colourable, as required.
\end{proof}

\section{One higher-degree vertex: choosability}
\label{dcpolarhdsec}

We have seen that Gallai trees are the only connected graphs that are not degree-choosable (see \cref{dcthm}), and polar Gallai trees are the only connected polar graphs that are not degree-choosable (see \cref{dcpolarthm}).
In this section, we prove an analogue for polar graphs where a single vertex is allowed to have a list that is smaller than its degree.

We first describe the exceptional structures, which are analogous to Gallai trees for connected graphs, or polar Gallai trees for connected polar graphs.
A \emph{rooted polar graph} is a pair $(Q,h)$ where $Q$ is a polar graph and $h \in V(Q)$ is a distinguished vertex called the \emph{root}.
We will define a class $\mathcal{N}_k$ of rooted polar graphs, for each $k \ge 2$. The exceptional structures for this setting 
are members of $\mathcal{N}_k$, where the root $h$ corresponds to the unique vertex with a list of size $k$ that can be smaller than $d(h)$.
For a rooted polar graph $(Q,h)$ that is not colourable for any appropriate polar assignment, colour deleting the vertex $h$, using any colour in its list, will result in a degree-polar assignment for $Q-h$ that is not colourable. So $Q-h$ is a polar Gallai tree.
Further properties of $Q-h$ will be captured in the upcoming definition of a ``balanced Gallai tree''. 

For a polar graph $Q$, recall that $\mathcal{B}_Q$ denotes the set of blocks of $Q$. 
Throughout this section, when $v \in V(Q)$, we write $\mathcal{B}_{Q}(v)$ (or just $\mathcal{B}(v)$ when $Q$ is clear from context) to denote the set of blocks of $Q$ that contain $v$.
For a block $B \in \mathcal{B}_Q$, we also write $I(B)$ to denote the internal vertices of $G(Q)$ contained in $V(B)$.

\begin{definition}[balanced Gallai tree]
\label{bgt rules}
Let $Q$ be a polar Gallai tree and let $U \subseteq V(Q)$.
We say a function $f : \mathcal{B}_Q \rightarrow \{-1,0,1\}$ is a \emph{balancing function} for $(Q,U)$ if:
    \begin{enumerate}[label=\textbf{\textup{(G\arabic*)}}]
        \item for each $B \in \mathcal{B}_Q$ with $I(B) \neq \emptyset$, either
            \begin{itemize}
                \item $f(B)=-1$ and $I(B) \subseteq U$, or
                \item $f(B)=0$ and $I(B) \cap U = \emptyset$;
            \end{itemize}\label{G1}
        \item for each cut vertex $v$ of $Q$, either
            \begin{itemize}
                \item $f(B)=0$ for each $B \in \mathcal{B}(v)$, and $v \notin U$;
                \item there exists $B^- \in \mathcal{B}(v)$ such that $f(B^-)=-1$, and $f(B) = 0$ for each $B \in \mathcal{B}(v) \setminus \{B^-\}$, and $v \in U$; or
                \item there exist $B^-,B^+ \in \mathcal{B}(v)$ such that $f(B^-)=-1$ and $f(B^+)=1$, and $f(B) = 0$ for each $B \in \mathcal{B}(v) \setminus \{B^-,B^+\}$, and $v \notin U$; and
            \end{itemize}\label{G2}
        \item If $F(B) \neq \emptyset$ for some $B \in \mathcal{B}_Q$, then $f(B)=0$.\label{G3}
    \end{enumerate}
    We say $(Q,U)$ is a {\em balanced Gallai tree} if $(Q,U)$ has a balancing function.
\end{definition}

We will see (in \cref{uniquegalaxy}) that for a balanced Gallai tree, there is a unique balancing function.
Let $(Q,U)$ be a balanced Gallai tree, let $f$ be a balancing function for $(Q,U)$, and let $B$ be a block of $Q$.
We say $B$ is \emph{normal} if $f(B)=0$, we say $B$ is \emph{light} if $f(B)=-1$, and we say $B$ is \emph{heavy} if $f(B)=1$.



Let $k \ge 2$ be an integer.
We now define $\mathcal{N}_k$ to be the class consisting of all rooted polar graphs $(Q,h)$ such that the following hold:
\begin{enumerate}[label=\textbf{(\textup{N\arabic*})}]
    \item $(Q-h,N_Q(h))$ is a balanced Gallai tree,\label{N1}
    \item if $B$ is a light block of $(Q-h,N_Q(h))$, then $d_B(v) \geq k-1$ for all $v \in V(B)$, and\label{N2}
    \item $Q$ is simple and no polarised edges of $Q$ are incident with $h$.\label{N3}
\end{enumerate}

\begin{figure}[b]
    \centering
    \includegraphics[width=13.5cm]{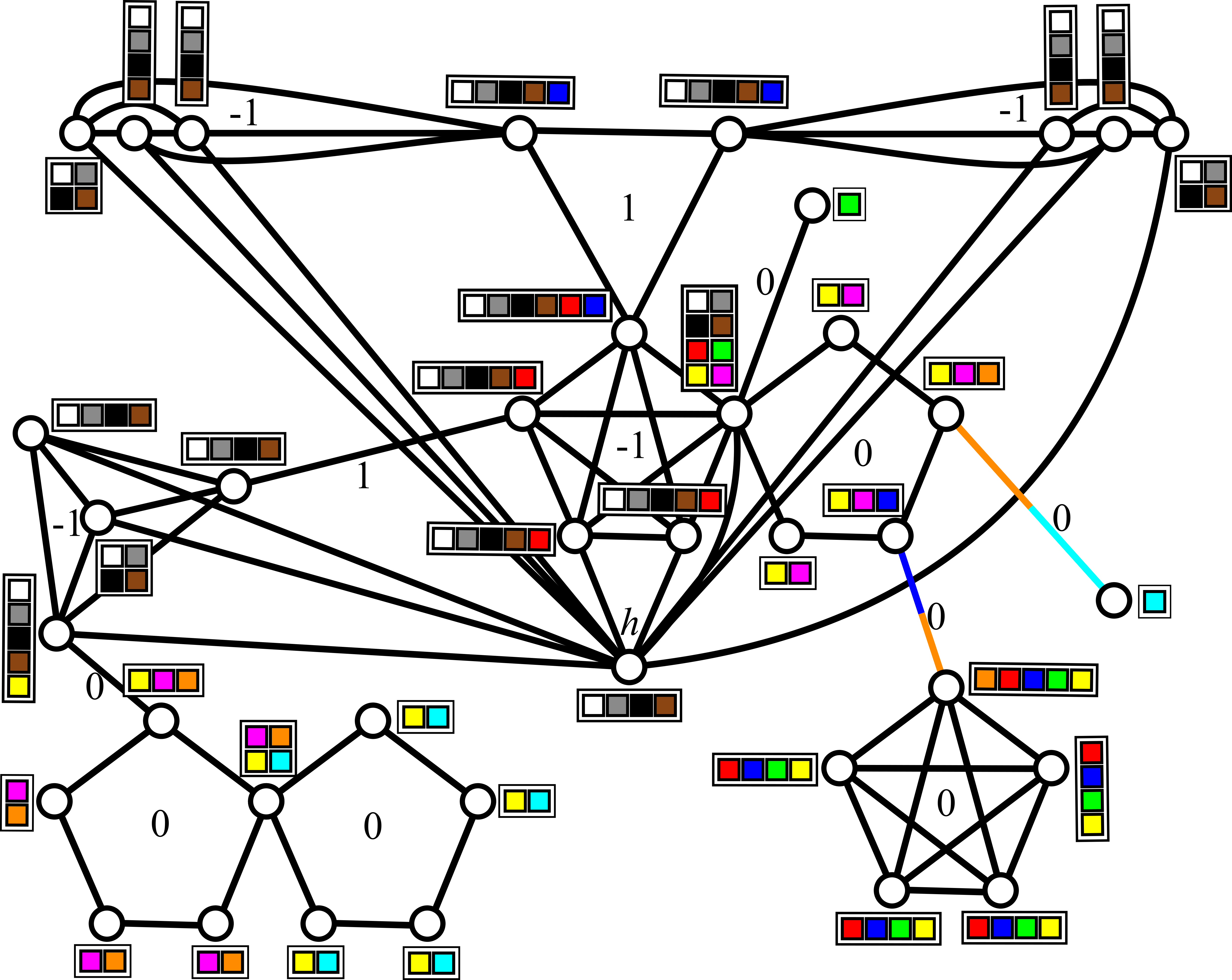}
    \caption{A rooted polar graph $(Q,h)$ belonging to $\mathcal{N}_4$, together with an $h$-bad polar assignment $\Gamma$.
    There are four light blocks and two heavy blocks, labeled $-1$ and $1$ respectively, while the remaining normal blocks are labelled $0$.}
    \label{graph in N4}
\end{figure}

To get some intuition for the rooted polar graphs in $\mathcal{N}_k$, we describe how to construct such a rooted polar graph.
Let $Q$ be a polar Gallai tree.
First, assign some subset $\mathcal{S}^+$ of the blocks of $Q$ to be heavy, where no two blocks in $\mathcal{S}^+$ are adjacent and, for each $B \in \mathcal{S}^+$, there are no internal vertices in $B$, and $B$ has no polarised edges.
    Then, we add blocks to $\mathcal{S}^-$ in such a way that, for each block $B \in \mathcal{S}^+$ and each cut vertex $v$ in $B$, the vertex $v$ belongs to precisely one block in $\mathcal{S}^-$;
 no two blocks in $\mathcal{S}^-$ are adjacent; and no block in $\mathcal{S}^-$ has any polarised edges.
We define the balancing function such that the blocks in $\mathcal{S}^-$ are light, and the blocks in $\mathcal{B}_Q \setminus (\mathcal{S}^- \cup \mathcal{S}^+)$ are normal.
Now let $Q'$ be the polar graph obtained from $Q$ by adding a single vertex~$h$, where $h$ is incident with each vertex that belongs to a block in $\mathcal{S}^-$ but does not belong to any block in $\mathcal{S}^+$, via a single edge that is not polarised.
We also require that $\min_{B \in \mathcal{S}^-} d(B) \ge k-1$.
Then, it is easily checked that $(Q', h) \in \mathcal{N}_k$.

Finally, we describe the properties of a ``bad'' polar assignment for a rooted polar graph in $\mathcal{N}_k$.

\begin{definition}[$h$-bad polar assignment]
    \label{hbaddef}
Let $(Q,h)$ be a rooted polar graph in $\mathcal{N}_k$.
Let $H$ be the polar graph with vertex set $\{h\} \cup N_Q(h)$ and edge set consisting of all edges of $Q$ incident with $h$.
Let $\mathcal{B}_{(Q,h)}=\mathcal{B}_{Q-h} \cup \{H\}$.
By \ref{N1}, each block of $Q-h$ is either normal, light, or heavy.
A polar assignment~$\Gamma$ for $Q$ is \emph{$h$-bad} if it can be obtained as the union of a polar assignment $\Gamma_B$ for each $B \in \mathcal{B}_{(Q,h)}$ where the $\Gamma_B$'s are pairwise non-conflicting, and each $\Gamma_B$ satisfies:
\begin{enumerate}
        \item If $B=H$, then $\Gamma_B=(L_B,\emptyset)$ where $|L_B(h)| = k$ and $L_B(u)=\emptyset$ for all other $u \in V(H)$.
        \item If $B$ is a normal block of $Q-h$, then $\Gamma_B$ is a bad degree-polar assignment for $B$.
        \item If $B$ is a heavy block of $Q-h$, then $\Gamma_B=(L_B,\emptyset)$ where $L_B$ is a uniform list assignment with $|L_B(u)|=d_B(u)-1$ for all $v \in V(B)$. 
        \item If $B$ is a light block of $Q-h$, then $\Gamma_B=(L_B,\emptyset)$ where $L_B$ is a uniform list assignment with $|L_B(u)|=d_B(u)+1$ and $L_H(h) \subseteq L_B(u)$ for all $u \in V(B)$.
\end{enumerate}
\end{definition}

Note that $B$ does not contain any polarised edges in cases (i), (iii), and (iv), by \ref{G3} and \ref{N3}.
Moreover, the final property of (iv) is possible by \ref{N2}.

See \cref{graph in N4} for an example of a rooted polar graph $(Q,h)$ in $\mathcal{N}_4$, together with an $h$-bad polar assignment for $Q$.
It is not difficult to see the following:

\begin{lemma}
    \label{hbadexistencelemma}
    Suppose, for some integer $k \ge 3$, that the rooted polar graph $(Q,h)$ is in $\mathcal{N}_k$.
    Then there exists an $h$-bad polar assignment for $Q$.
\end{lemma}

Our main result of this section is the following:
\begin{theorem}
\label{dcpolarhdthm}
    Let $Q$ be a polar graph with polar assignment $\Gamma=(L,R)$.
    Let $h \in V(Q)$ such that $Q-h$ is connected, $|L(h)| \geq 2$, and $|L(v)| \geq d(v)$ for all $v \in V(Q) \setminus \{h\}$.
    Then $Q$ is not $\Gamma$-colourable if and only if $(Q,h) \in \mathcal{N}_{|L(h)|}$ and $\Gamma$ is $h$-bad.
\end{theorem}

\noindent
Note that \cref{dcpolarhdthm} does not require that $|L(h)| < d(h)$ (in which case \cref{dcpolarthm} also applies), but does require that $h$ is not a cut vertex. 
For any polar Gallai tree $Q$, we have that $(Q,h) \in \mathcal{N}_{d(h)}$ for every internal vertex $h$ of $Q$.

We first show that a balanced Gallai tree has a unique balancing function.

\begin{lemma}
\label{uniquegalaxy}
Let $(Q,U)$ be a balanced Gallai tree, 
and let $f_0$ and $f_1$ be balancing functions for $(Q,U)$.
Then $f_0=f_1$.
\end{lemma}
\begin{proof}
    We first prove the following claim:

    \begin{claim}
    \label{determined claim}
    Let $v$ be a cut vertex of $Q$ and let $B$ be a block in $\mathcal{B}(v)$ such that for all $B' \in \mathcal{B}(v) \setminus \{B\}$ we have $f_0(B')=f_1(B')$.
        Then $f_0(B)=f_1(B)$.
    \end{claim}
    \begin{subproof}
        Let $i \in \{0,1\}$.
        We will show that $f_i(B)$ is determined by the values of $f_i$ on $\mathcal{B}(v) \setminus \{B\}$.
        First assume that $v \in U$.
        Then, by \ref{G2}, if there exists $B' \in \mathcal{B}(v) \setminus \{B\}$ with $f_i(B')=-1$ then $f_i(B)=0$, and if there is no such block then $f_i(B)=-1$.
        Next assume that $v \notin U$.
        By \ref{G2}, there are two cases to consider.
        If $f_i(B')=0$ for all $B' \in \mathcal{B}(v) \setminus \{B\}$, then we are in the first case, so $f_i(B)=0$.
        Otherwise, there is some block $B_1 \in \mathcal{B}(v) \setminus \{B\}$ with $f_i(B_1) \in \{-1,1\}$, and we are in the second case.
        Therefore, either there is a block $B_2 \in \mathcal{B}(v) \setminus \{B_1,B\}$ such that $\{f_i(B_1),f_i(B_2)\} = \{-1,1\}$, in which case $f_i(B)=0$; or for every block $B' \in \mathcal{B}(v) \setminus \{B_1,B\}$ we have $f_i(B')=0$, in which case we choose $f_i(B) \in \{-1,1\}$ such that $\{f_i(B_1),f_i(B)\} = \{-1,1\}$. 
    \end{subproof}

    Let $B$ be a block of $Q$ containing an internal vertex $v$.
    By \ref{G1}, if $v \in U$, then $f_0(B)=f_1(B)=-1$, and if $v \notin U$, then $f_0(B)=f_1(B)=0$.
    Therefore $f_0(B)=f_1(B)$ for any block $B$ of $Q$ containing an internal vertex.
    In particular, $f_0(B)=f_1(B)$ for any leaf block $B$ of $Q$.

    Towards a contradiction, suppose $f_0 \neq f_1$.
    Let $$\mathcal{B}' = \{B \in \mathcal{B}_Q : f_0(B) \neq f_1(B)\}.$$
    Let $H$ be a connected component of $Q\big[\bigcup_{B' \in \mathcal{B}'} V(B')\big]$, 
    and let $B''$ be a leaf block of $H$.
    Now $f_0(B'') \neq f_1(B'')$, by construction, so $B''$ is not a leaf block of $Q$.
    Let $v \in V(B'')$ such that $v$ is a cut vertex in $Q$, but $v$ is not a cut vertex in $H$.
    Then $B$ is the unique block in $\mathcal{B}(v)$ for which $f_0(B) \neq f_1(B)$.
    But this contradicts \cref{determined claim}.
\end{proof}

We now state some properties of balanced Gallai trees, which follow easily from \cref{bgt rules}.

\begin{lemma}
    \label{list}
    Let $(Q,U)$ be a balanced Gallai tree and let $v \in V(Q)$.
    \begin{enumerate}
        \item If $v \in U$, then $v$ is not in any heavy blocks, and $v$ is in precisely one light block.\label{V not in sync}\label{V means one async}\label{V only async}
        \item If $v \not\in U$ and is in at least one light or heavy block, then $v$ is in exactly one light block and exactly one heavy block.\label{several conditions}
        \item If $v$ is in a heavy block, then $v$ is a cut vertex and is also in a light block.\label{sync only cut}
        \item If $v$ is in distinct blocks $B$ and $B'$, then at most one of $B$ and $B'$ is light and at most one of $B$ and $B'$ is heavy.\label{no neighbours}
        \item If there are no light blocks, then all blocks are normal.\label{all fixed}
    \end{enumerate}
\end{lemma}







Using \cref{hbaddef,list}, and the definition of $\mathcal{N}_k$, it is not difficult (but somewhat tedious) to prove the following.  We omit the proof.

\begin{lemma}
    Let $(Q,h)$ be a rooted polar graph in $\mathcal{N}_k$, for some $k \ge 2$, and let $\Gamma=(L,R)$ be an $h$-bad polar assignment for $Q$. Then $|L(v)|=d(v)$ for each $v \in V(Q) \setminus \{h\}$.
\end{lemma}

We now prove the main result of this section.

\setcounter{theorem}{3}

\begin{proof}[Proof of \cref{dcpolarhdthm}]
    Let $k = |L(h)|$.
    We first assume that $(Q,h)$ is in $\mathcal{N}_k$ and $\Gamma$ is $h$-bad, and show that $Q$ is not $\Gamma$-colourable.
    By \ref{N1}, $Q-h$ is a polar Gallai tree.
    For $i \in [k]$, let $(Q-h,\Gamma_i)$ be the $(h,c_i)$-colour deletion from $(Q,\Gamma)$.
    It suffices to show that $\Gamma_i$ is a bad polar assignment for $Q-h$ for each $i \in [k]$.

    Let $\{B_1,B_2,\dotsc,B_t\}$ be the set of blocks of $Q-h$.
    Recall that $\Gamma$, the $h$-bad polar assignment for $Q$, can be obtained as a union of polar assignments $\Gamma_{B_s}$ for $s \in [t]$, conforming with \cref{hbaddef}.
    Fix some $i \in [k]$.
    For each $s \in [t]$, we will construct a bad polar assignment $\Gamma_{i,s}$ for $B_s$ such that $\Gamma_i = \bigcup_{s \in [t]} \Gamma_{i,s}$.
    We let $\mathcal{F}$, $\mathcal{S}^-$, and $\mathcal{S}^+$ denote the set of all normal, light, and heavy blocks of $Q-h$, respectively.

    Let $s \in [t]$.
    If $B_s \in \mathcal{F}$, then $\Gamma_{B_s}$ is a bad degree-polar assignment for $B_s$, in which case we let $\Gamma_{i,s}=\Gamma_{B_s}$.
    Suppose $B_s \in \mathcal{S}^+ \cup \mathcal{S}^-$.
    Then $F(B_s)=\emptyset$ by \ref{G3}.
    So we let $\Gamma_{i,s}=(L_{i,s},\emptyset)$ where,
    if $B_s \in \mathcal{S}^+$, then $L_{i,s}(u)=L_{B_s}(u) \cup \{c_i\}$ for all $u\in V(B_s)$;
    whereas if $B_s$ is in $\mathcal{S}^-$, then $L_{i,s}(u)=L_{B_s}(u) \setminus \{c_i\}$ for all $u \in V(B_s)$.
    When $u$ is in a heavy block $B_s$, then $|L_{B_s}(u)| = d_B(u) - 1$ by \cref{hbaddef}(iii).
    By \cref{list}(ii), $u$ is also in a light block $B'$, so, by \cref{hbaddef}(iv), $c_i \in L(h) \subseteq L_{B'}(u)$.  Since $L_{B_s}$ and $L_{B'}$ are non-conflicting, $c_i \notin L_{B_s}(u)$.  Thus, $|L_{i,s}(u)| = |L_{B_s}(u)| + 1 = d_B(u)$.
    When $u$ is in a light block~$B_s$, then $|L_{B_s}(u)| = d_B(u) + 1$ by \cref{hbaddef}(iv), and, by \cref{hbaddef}(iv), $c_i \in L(h) \subseteq L_{B_s}(u)$.
    So $|L_{i,s}(u)| = |L_{B_s}(u)|-1 = d_B(u)$.
    In either case, since $L_{B_s}$ is a uniform list assignment (by \cref{hbaddef}), $L_{i,s}$ is a bad degree-list assignment for $B_s$.
    Since the polar assignments $\Gamma_{B_s}$ are non-conflicting for $s \in [t]$, and by \cref{list}\ref{no neighbours}, the polar assignments $\Gamma_{i,s}$ are non-conflicting for $s \in [t]$.
    So $\bigcup_{s \in [t]} \Gamma_{i,s}$ is a bad polar assignment for $Q-h$,
    and it is easily checked that this polar assignment coincides with $\Gamma_i$. 
    This proves one direction.

    We now assume that $Q$ is not $\Gamma$-colourable, and show that $(Q,h) \in \mathcal{N}_k$ and $\Gamma$ is $h$-bad. 

    \begin{claim}
    \label{fundamental claim}
    If $c \in L(h)$ and $(Q-h,\Gamma_c)$ is the $(h,c)$-colour deletion from $(Q,\Gamma)$, then $Q-h$ is a polar Gallai tree and $\Gamma_c$ is bad.
    \end{claim}
    \begin{subproof}
        Since $Q$ is not $\Gamma$-colourable, $Q-h$ is not $\Gamma_c$-colourable.
        Since $|L(v)| \ge d(v)$ for each $v \in V(Q)\setminus \{h\}$, it follows that
        $\Gamma_c$ is a degree-polar assignment for $Q-h$.
        Since $Q-h$ is connected, \cref{dcpolarthm} implies that $Q-h$ is a polar Gallai tree and $\Gamma_c$ is bad, as required.
    \end{subproof}

    By \cref{fundamental claim}, $Q-h$ is simple.
    To satisfy \ref{N3}, we require the following:

    \begin{claim}
        \label{no restrict h claim}
        There is no polarised edge, and no pair of parallel edges, incident with $h$.
    \end{claim}
    \begin{subproof}
        First let $e=hv$ be a polarised edge incident with $h$.
        By \cref{cdlemma1} there is some $c \in L(h)$ such that the $(h,c)$-colour deletion $(Q-h,\Gamma_c)$ from $(Q,\Gamma)$ leaves $v$ with a list of size more than $|L(v)|-e(h,v) \ge d_{Q-h}(v)$,
        so $\Gamma_c$ is not bad, contradicting \cref{fundamental claim}.
        Similarly, suppose that $h$ has a neighbour $v$ such that $e(h,v) > 1$.
        By \cref{cdlemma2}, there is some $c \in L(h)$ such that the $(h,c)$-colour deletion leaves $v$ with a list of size more than $|L(v)|-e(h,v)$, again contradicting \cref{fundamental claim}. 
    \end{subproof}

    By \cref{fundamental claim,no restrict h claim}, \ref{N3} holds for $(Q,h)$.
    Next we work towards showing \ref{N1} holds, that is, showing that $(Q-h,N_Q(h))$ is a balanced Gallai tree.

    \begin{claim}
    \label{containment claim}
    Let $v \in N_Q(h)$. Then $L(h) \subseteq L(v)$.
    \end{claim}
    \begin{subproof}
        Let $c \in L(h)$ and let $(Q-h,\Gamma_c)$ be the $(h,c)$-colour deletion from $(Q,\Gamma)$, with $\Gamma_c=(L_c,R_c)$.
        By \cref{no restrict h claim}, $hv$ is not polarised.
        If $c \notin L(v)$, then $L_c(v)=L(v)$, so $|L_c(v)| > d_{Q-h}(v)$, and therefore $L_c$ is not bad, contradicting \cref{fundamental claim}.
        Hence $c \in L(v)$.
    \end{subproof}

    Let $L(h) = \{c_1,\dotsc,c_k\}$.
    For each $i \in [k]$, let $(Q-h,\Gamma_i)$ be the $(h,c_i)$-colour deletion from $(Q,\Gamma)$, with $\Gamma_i=(L_i,R_i)$.
    By \cref{fundamental claim}, $Q-h$ is a polar Gallai tree and $\Gamma_i$ is a bad polar assignment for $Q-h$, so it is a union of bad polar assignments for each block of $Q-h$.
    Let $\{B_1,B_2,\dotsc,B_t\}$ be the set of blocks of $Q-h$ and, for each $s \in [t]$, let $\Gamma_{i,s}=(L_{i,s},R_{i,s})$ be a bad polar assignment for $B_s$ such that $\Gamma_i = \bigcup_{s \in [t]} \Gamma_{i,s}$.

    Let $s \in [t]$.
    We call $B_s$ {\em airy} if there is some $C_s \subseteq \mathbb{N}$ with $L(h) \subseteq C_s$ such that $L_{i,s}(v)=C_s \setminus \{c_i\}$ for all $v \in V(B_s)$ and all $i \in [k]$.
    We call $B_s$ {\em weighty} if there is some $C_s \subseteq \mathbb{N}$ with $L(h) \cap C_s = \emptyset$ such that $L_{i,s}(v)=C_s \cup \{c_i\}$ for all $v \in V(B_s)$ and all $i \in [k]$.
    We call $B_s$ {\em standard} if $L_{i,s}=L_{j,s}$ for all $i,j \in k$.
    Observe that a block of $Q-h$ can have at most one of the properties: airy, weighty, or standard. (The notions of airy, weighty, and standard blocks of $Q-h$ will eventually align with the notion of light, heavy, and normal blocks in a balanced Gallai tree $(Q-h,N_Q(h))$, respectively.)

    The next claim is a straightforward consequence of the definition of colour deletion, which we will use often.
    We then prove a series of claims that describe properties of airy, standard, and weighty blocks of $Q-h$.
    \begin{claim}
        \label{cdimmediate}
        If $u \in V(Q-h) \setminus N_Q(h)$, then $L_i(u)=L_j(u)$ for all $i,j \in [k]$.
    \end{claim}

    \begin{claim}
        \label{first async claim}
        Let $u \in I(B_s)$, for $s \in [t]$. Then $B_s$ is airy if and only if $u \in N_Q(h)$.
    \end{claim}
    \begin{subproof}
        Suppose $u \in N_Q(h)$.
        Then $L(h) \subseteq L(u)$, by \cref{containment claim}, and $L_i(u)=L(u) \setminus \{c_i\}$ for all $i \in [k]$.
        Since $u$ is not a cut vertex, $B_s$ is the only block of $Q-h$ containing $u$, 
        so $L_{i,s}(u)=L_i(u)$.
        Now $L_{i,s}(u)=L(u) \setminus \{c_i\}$ and $L_{i,s}$ is a bad degree-list assignment for $B_s$, so it is uniform.
        Thus, letting $C_s=L(u)$, we have $L_{i,s}(v)=C_s \setminus \{c_i\}$ for all $v \in V(B_s)$ and all $i \in [k]$, so $B_s$ is airy.

        For the other direction, suppose $B_s$ is airy.
        Let $i \in [k]$.
        Then there is some $C_s \subseteq \mathbb{N}$ with $L(h) \subseteq C_s$ such that $L_{i,s}(v)=C_s \setminus \{c_i\}$ for all $v \in V(B_s)$.
        Since $u$ is only in one block of $Q-h$, we have $L_i(u)=L_{i,s}(u)$.
        Since $L_{i,s}(u)=C_s \setminus \{c_i\}$ for all $i \in [k]$, we have $L_i(u) \neq L_j(u)$ for any $j \in [k] \setminus \{i\}$. %
        Hence $u \in N_Q(h)$, by \cref{cdimmediate}.
    \end{subproof}


    \begin{claim}
        \label{no polar edges claim}
        If $F(B_s) \neq \emptyset$, for some $s \in [t]$, then $B_s$ is standard.
    \end{claim}
    \begin{subproof}
        Assume that $B_s$ contains a polarised edge $e=uv$, for some $s \in [t]$.
        Since $Q-h$ is a polar Gallai tree, $B_s$ is a $K_2$-block of $Q-h$.
        Let $i,j \in [k]$.
        Since $\Gamma_{i,s}$ and $\Gamma_{j,s}$ are bad polar assignments for $B_s$, and $R_i(e,u,v)=R_j(e,u,v)$, we have that $L_{i,s}=L_{j,s}$, implying that $B_s$ is standard.
    \end{subproof}

    \begin{claim}
        \label{first sync claim}
        Suppose $v \in V(B_s)$ for some $s \in [t]$ such that $B_s$ is weighty.
        Then $v$ is a cut vertex.
    \end{claim}
    \begin{subproof}
        Towards a contradiction, assume that $v$ is an internal vertex of $B_s$.
        Since $B_s$ is weighty, \cref{first async claim} implies that $v \notin N_Q(h)$.
        Therefore, by \cref{cdimmediate}, $L_i(v)=L_j(v)$ for all $i,j \in [k]$.
        Since $v$ is an internal vertex, $L_{i,s}(v)=L_i(v)=L_j(v)=L_{j,s}(v)$, so $B_s$ is not weighty, a contradiction.
        We deduce that $v$ is a cut vertex, as required.
    \end{subproof}

    \begin{claim}
    \label{classification claim}
        Each block of $Q-h$ is either airy, weighty, or standard.
    \end{claim}
    \begin{subproof}
        We call a block of $Q-h$ {\em sporadic} if it is not airy, weighty, or standard.
        It suffices to show that there are no sporadic blocks in $Q-h$.
        Assume that $Q-h$ has at least one sporadic block.
        First, we observe that there is some $s \in [t]$ and some $u \in V(B_s)$ such that $B_s$ is the unique sporadic block containing $u$.
        (To see this, consider a leaf block of an arbitrary component in the polar graph obtained by restricting $Q-h$ to the vertices in sporadic blocks.)

        Since $B_s$ is not standard, there exists a vertex $v \in V(B_s)$ such that $L_{i,s}(v) \neq L_{j,s}(v)$ for some $i,j \in [k]$.
        By \cref{no polar edges claim}, $B_s$ has no polarised edges, and so any bad polar assignment for $B_s$ is uniform.
        Thus $L_{i,s}(v') \neq L_{j,s}(v')$ for any $v' \in V(B_s)$. 
        In particular, $L_{i,s}(u) \neq L_{j,s}(u)$.
        If $u$ is an internal vertex of $Q-h$, then $L_i(u) \neq L_j(u)$, so $u \in N_Q(h)$ by \cref{cdimmediate}, implying that $B_s$ is airy by \cref{first async claim}, a contradiction.
        So $u$ is a cut vertex of $Q-h$.

        Let $i,j \in [k]$ such that $L_{i,s}(u) \neq L_{j,s}(u)$.
        Now $|L_{i,s}(u)|=|L_{j,s}(u)|$, so there exist $c \in L_{i,s}(u) \setminus L_{j,s}(u)$ and $c' \in L_{j,s}(u) \setminus L_{i,s}(u)$.
        Suppose that at least one of $c$ or $c'$ is not in $\{c_i,c_j\}$.
        Without loss of generality, $c \notin \{c_i,c_j\}$. 
        There are two cases.
        If $u \notin N_Q(h)$, then, for all $i,j \in [k]$, we have $L_i(u)=L_j(u)$ and so
        \[
            \bigcup_{B_t \in \mathcal{B}_{Q-h}(u)} L_{i,t}(u) =L_i(u)=L_j(u) = \bigcup_{B_t \in \mathcal{B}_{Q-h}(u)} L_{j,t}(u).
        \]
        On the other hand, if $u \in N_Q(h)$, then, by \cref{no restrict h claim,containment claim}, for all $i,j \in [k]$ we have that $L_i(u) \setminus \{c_j\} \cup \{c_i\}=L_j(u)$, and so
        \[
            \left( \bigcup_{B_t \in \mathcal{B}_{Q-h}(u)} L_{i,t}(u) \right) \setminus \{c_j\} \cup \{c_i\} =\bigcup_{B_t \in \mathcal{B}_{Q-h}(u)} L_{j,t}(u).
        \]
        In either case, since $c \in L_{i,s}(u) \setminus L_{j,s}(u)$, and (for the latter case) since $c \neq c_j$, there exists some $B_t \in \mathcal{B}_{Q-h}(u) \setminus \{B_s\}$ such that $c \in L_{j,t}(u) \setminus L_{i,t}(u)$.
        However, such a $B_t$ is not standard, and, since $c \notin \{c_i,c_j\}$, we have that $B_t$ is also neither airy nor weighty.
        Hence $B_t$ is sporadic, which contradicts the fact that $B_s$ is the only sporadic block that contains $u$.

        Now we may assume,
        for any $i,j \in [k]$ such that $L_{i,s}(u) \neq L_{j,s}(u)$,
        that $L_{i,s}(u) \symdiff L_{j,s}(u) = \{c_i,c_j\}$.
        Suppose $L_{i,s}(u) \neq L_{j,s}(u)$ for all distinct $i,j \in [k]$.
        In this case, we argue that $B_s$ is either airy or weighty.
        Towards a contradiction, say that, for distinct $i,j,j' \in [k]$, we have $c_i \in L_{i,s}(u)$ and $c_j \in L_{j,s}(u)$, 
        but $c_{j'} \in L_{i,s}(u)$ and $c_i \in L_{j',s}(u)$. 
        Since $L_{i,s}(u) \symdiff L_{j',s}(u)=\{c_i,c_{j'}\}$, we also have $c_i \notin L_{i,s}(u)$, a contradiction.
        It now follows that for all distinct $i,j \in [k]$ either $c_i \in L_{i,s}(u)$ and $c_j \in L_{j,s}(u)$, in which case $B_s$ is weighty; or $c_j \in L_{i,s}(u)$ and $c_i \in L_{j,s}(u)$, in which case $B_s$ is airy.

        Finally, suppose there are distinct $i,j,j' \in [k]$ such that $L_{i,s}(u) \neq L_{j,s}(u)$ but $L_{j,s}(u) = L_{j',s}(u)$.
        Then $L_{i,s}(u) \symdiff L_{j,s}(u) = L_{i,s}(u) \symdiff L_{j',s}(u)$.
        But $L_{i,s}(u) \symdiff L_{j,s}(u) = \{c_i,c_j\}$ and $L_{i,s}(u) \symdiff L_{j',s}(u) = \{c_i,c_{j'}\}$, where $c_j \neq c_{j'}$, since $j \neq j'$.
        This contradiction completes the proof of the claim.
    \end{subproof}


    \begin{claim}
        \label{at most one}
        Let $v \in V(Q-h)$. Then $v$ is in at most one airy block, and $v$ is in at most one weighty block.
    \end{claim}
    \begin{subproof}
        Suppose $v$ is in distinct airy blocks $B_s$ and $B_{s'}$ of $Q-h$, for $s,s' \in [t]$.
        By the definition of airy, $c_2 \in L_{1,s}(v) \cap L_{1,s'}(v)$.
        Then $L_{1,s}$ and $L_{1,s'}$ are conflicting, so $L_1$ is not bad, a contradiction.
        Similarly, if $v$ is in distinct weighty blocks $B_s$ and $B_{s'}$ of $Q-h$, for $s,s' \in [t]$, then $c_1 \in L_{1,s}(v) \cap L_{1,s'}(v)$, so $L_1$ is not bad, a contradiction.
    \end{subproof}

    \begin{claim}
    \label{cut vertex claim}
    Let $u$ be a cut vertex of $Q-h$. 
    \begin{itemize}
        \item If $u \in N_Q(h)$, then exactly one block of $\mathcal{B}_{Q-h}(u)$ is airy, and the rest are standard.
        \item If $u \notin N_Q(h)$, then either every block in $\mathcal{B}_{Q-h}(u)$ is standard; or exactly one is airy, exactly one is weighty, and the rest are standard.
    \end{itemize}
    \end{claim}
    \begin{subproof}
%
        Suppose $u \in N_Q(h)$.
        By \cref{no restrict h claim,containment claim}, $L_1(u)=L(u) \setminus \{c_1\}$ and $L_2(u)=L(u) \setminus \{c_2\}$.
        Therefore, there exists some $B_s$ in $\mathcal{B}_{Q-h}(u)$, for $s \in [t]$, such that $c_2 \in L_{1,s}(u) \setminus L_{2,s}(u)$.
        Clearly $B_s$ is not standard. 
        Also, each $B_{s'} \in \mathcal{B}_{Q-h}(u)$ is not weighty, for $s' \in [t]$, since $c_1 \notin L_1(u)$ and so $c_1 \notin L_{1,s'}(u)$.
        Hence $B_s$ is airy, and all other blocks in $\mathcal{B}_{Q-h}(u)$ are standard, by \cref{classification claim,at most one}.

        Now suppose $u \notin N_Q(h)$.
        Then $L_i(u)=L_j(u)$ for all $i,j \in [k]$, by \cref{cdimmediate}.
        If $u$ is not in any airy or weighty blocks of $Q-h$, then the claim holds, by \cref{classification claim}, so we assume that $u$ is in a block $B_s$ of $Q-h$ that is either airy or weighty, for some $s \in [t]$.
        If $B_s$ is airy, then, for all distinct $i,j \in [k]$, we have 
        $c_j \in L_{i,s}(u) \setminus L_{j,s}(u)$.
        On the other hand, if $B_s$ is weighty, then, for all distinct $i,j \in [k]$, we have 
        $c_i \in L_{i,s}(u) \setminus L_{j,s}(u)$.
        In either case, since $L_i(u)=L_j(u)$, we have
        \[
            \bigcup_{B_{s'} \in \mathcal{B}_{Q-h}(u)} L_{i,s'}(u) = 
            \bigcup_{B_{s'} \in \mathcal{B}_{Q-h}(u)} L_{j,s'}(u).
        \]
        Since there exists $c \in \{c_i,c_j\}$ such that $c \in L_{i,s}(u) \setminus L_{j,s}(u)$, there is some $B_{s'} \in \mathcal{B}_{Q-h}(u) \setminus \{B_s\}$ such that $c \in L_{j,s'}(u) \setminus L_{i,s'}(u)$, for $s' \in [t]$.
        Then $B_{s'}$ is not standard, so is either airy or weighty.
        It now follows from \cref{classification claim,at most one} that one of $B_s$ and $B_{s'}$ is airy and the other is weighty, and all other blocks of $Q-h$ are standard, as required.
    \end{subproof}
    
    We now show that $(Q-h,N_Q(h))$ is a balanced Gallai tree, conforming with \cref{bgt rules}.  Let $f : \mathcal{B}_Q \rightarrow \{-1,0,1\}$ be the function where $f(B) = -1$ if $B$ is airy, $f(B) = 0$ if $B$ is standard, and $f(B) = 1$ if $B$ is weighty (this is well defined by \cref{classification claim}).
    We let $\mathcal{S}^-$ be the set of airy blocks of $Q-h$, we let $\mathcal{S}^+$ be the set of weighty blocks of $Q-h$ and we let $\mathcal{F}$ be the set of standard blocks of $Q-h$. 

    By \cref{first sync claim}, 
    no block of $Q-h$ in $\mathcal{S}^+$ contains any internal vertices.
    Therefore, if $u$ is an internal vertex of a block $B$ of $Q-h$, then $B$ is in either $\mathcal{S}^-$ or $\mathcal{F}$.
    By \cref{first async claim}, if $u$ is an internal vertex in $N_Q(h)$, then $B$ is in $\mathcal{S}^-$ and otherwise $B$ is in $\mathcal{F}$.
    Therefore $(Q-h,N_Q(h))$ satisfies \ref{G1}.
    By \cref{cut vertex claim}, $(Q-h,N_Q(h))$ satisfies \ref{G2}, and, finally, by \cref{no polar edges claim}, $(Q-h,N_Q(h))$ satisfies \ref{G3}.
    Thus $f$ is a balancing function for $(Q-h,N_Q(h))$, so $(Q-h,N_Q(h))$ is a balanced Gallai tree, and $(Q,h)$ satisfies \ref{N1}.
    By \cref{uniquegalaxy}, the airy, standard, and weighty blocks of $Q-h$ coincide with the light, normal, and heavy blocks of $Q-h$, respectively.  Henceforth, we use that the light, normal, and heavy blocks of $Q-h$ satisfy the defining properties of airy, standard, and weighty blocks of $Q-h$, respectively.

    We next consider \ref{N2}.
    Let $u$ be a vertex of a light block $B_s$, for $s \in [t]$.
    Then there is some $C_s \subseteq \mathbb{N}$ with $L(h) \subseteq C_s$ such that $L_{i,s}(v) = C_s \setminus \{c_i\}$ for all $v \in V(B_s)$ and all $i \in [k]$.
    Since $L(h) \subseteq C_s$, we have that $|C_s| \geq |L(h)|=k$, so $|C_s \setminus \{c_i\}| \geq k-1$ for all $i \in [k]$.
    Since $L_{i,s}(u)=C_s \setminus \{c_i\}$, we have that $|L_{i,s}(u)| \geq k-1$ for all $i \in [k]$.
    Since $L_{i,s}$ is a bad degree-list assignment for $B_s$, we have that $|L_{i,s}(u)|=d_{B_s}(u)$, so $d_{B_s}(u) \geq k-1$.
    Hence $(Q,h)$ satisfies \ref{N2}.
    This completes the proof that $(Q,h) \in \mathcal{N}_k$.

    It remains to prove that $\Gamma$ is an $h$-bad polar assignment for $Q$.
    To see this, we show that $\Gamma = \bigcup_{B \in \mathcal{B}_{(Q,h)}} \Gamma_B$, where each $\Gamma_B$ is a polar assignment for $B$ satisfying \cref{hbaddef}(i)--(iv), and the $\Gamma_B$'s are pairwise non-conflicting.
    If $B_s$ is normal for some $s \in [t]$, then, since $L_{i,s}=L_{j,s}$ for all $i,j \in [k]$, we let $\Gamma_{B_s}=\Gamma_{i,s}$ for any $i \in [k]$, which is a bad degree-polar assignment for $B_s$.
    By \cref{no restrict h claim,no polar edges claim}, every other subgraph $B \in \mathcal{B}_{(Q,h)}$ has no polarised edges, so it suffices to find a list assignment $L_B$ for $B$.
    For $H$, we simply let $L_H(h)=L(h)$ and $L_H(u)=\emptyset$ for all $u \in V(H)-h$. 
    Now consider $B_s$ for some $s \in [t]$.
    If $B_s$ is light or heavy, then there is some $C_s \subseteq \mathbb{N}$ such that either
    $L(h) \subseteq C_s$ and $L_{i,s}(v)=C_s \setminus \{c_i\}$ for all $v \in V(B_s)$ and all $i \in [k]$, or
    $L(h) \cap C_s = \emptyset$ and $L_{i,s}(v)=C_s \cup \{c_i\}$ for all $v \in V(B_s)$ and all $i \in [k]$, respectively.
    In either case, let $L_{B_s}$ be the uniform list assignment that assigns $C_s$ to each vertex of $B_s$.
    If $B_s$ is heavy, then since $L_{i,s}$ is a bad degree-list assignment for $B_s$ for $i \in [k]$, we have that $|C_s|=d_{B_s}(u)+1$ for all $u \in V(B_s)$. %
    Similarly, if $B_s$ is light, then since $L_{i,s}$ is a bad degree-list assignment for $B_s$ for $i \in [k]$, we have that $|C_s|=d_{B_s}(u)-1$ for all $u \in V(B_s)$. 
    Now each $B \in \mathcal{B}_{(Q,h)}$ satisfies (i)--(iv) of \cref{hbaddef}.

    It remains only to show that the $\Gamma_{B_s}$'s are pairwise non-conflicting.
    Since $L_H(u)=\emptyset$ for all $u \in V(H)-h$, clearly $\Gamma_H$ and $\Gamma_{B_s}$ are non-conflicting for each $s \in [t]$.
    Therefore, let $B_s$ and $B_{s'}$ be adjacent blocks of $Q-h$, and let $L_s$ and $L_{s'}$ be their respective list assignments, such that $L_s(u) \cap L_{s'}(u) \neq \emptyset$ for some $u \in V(B_s) \cap V(B_{s'})$.
    If $B_s$ and $B_{s'}$ are normal, then, since $L_{i,s}(u)=L_s(u)$ and $L_{i,s'}(u)=L_{s'}(u)$ for any $i \in [k]$, we have $L_{i,s}(u) \cap L_{i,s'}(u) \neq \emptyset$, contradicting that $\Gamma_i$ is a bad polar assignment for $Q-h$.
    If $B_s$ is heavy and $B_{s'}$ is normal, then, for any $i \in [k]$, we have $L_{i,s'}(u) \cap C_s \neq \emptyset$, where $L_{i,s}(u)=C_s \cup \{c_i\}$, so $L_{i,s}(u) \cap L_{i,s'}(u) \neq \emptyset$, a contradiction.
    Since $B_s$ and $B_{s'}$ are not both heavy, by \cref{list}\ref{no neighbours}, we may now assume without loss of generality that $B_s$ is light.
    There exists $c_i \in L_s(u) \cap L_{s'}(u)$, for $i \in [k]$. %
    Let $j \in [k] \setminus \{i\}$.
    Since $B_s$ is light, $c_i \in L_{j,s}(u)$.
    By \cref{list}\ref{no neighbours}, $B_{s'}$ is heavy or normal, so $L_{s'}(u) \subseteq L_{j,s'}(u)$, and thus $c_i \in L_{j,s'}(u)$.
    Therefore, $L_{j,s}(u) \cap L_{j,s'}(u) \neq \emptyset$, a contradiction. 
    We deduce that the members of $\{\Gamma_B : B \in \mathcal{B}_{(Q,h)}\}$ are pairwise non-conflicting, as required. 
\end{proof}


\section{One higher-degree vertex: \texorpdfstring{$k$}{k}-restricted vertices}
\label{k restricted sec}

The previous section concerned polar graphs with a polar assignment that was ``close'' to being a degree-polar assignment, in the sense that at most one vertex has degree larger than the size of its list.
We now focus on $k$-polar assignments that are ``close'' to being a degree-polar assignment, in the sense that the $k$-polar assignment is for a polar graph having just one vertex with degree more than $k$.
Using \cref{dcpolarhdthm} from the previous section, we can obtain a characterisation of when such polar graphs are $k$-choosable.
In this section, we will prove a necessary property for such a polar graph to have at least one $k$-restricted vertex (see \cref{k restricted theorem}).
This leads to a characterisation of when a graph in the class has a $k$-restricted vertex (see \cref{krestcor}).
%
In order to prove \cref{restrictedthm}, we do not require the full generality of \cref{k restricted theorem}; the special case where $Q$ is $2$-connected suffices (see \cref{krtcorollary}).  We prove the more general result as it may be of independent interest.

For a graph $G$, we say that $H$ is a \emph{Gallai tree component of $G$} if $H$ is a connected component of $G$ that is a Gallai tree.
Let $Q$ be a polar graph.
We say that a polar subgraph $Q'$ of $Q$ is a \emph{Gallai tree component of $Q$} if $G(Q')$ is a Gallai tree component of $G(Q)$.
Note that when $Q'$ is a Gallai tree component of $Q$, the polar subgraph $Q'$ might not be a polar Gallai tree, as it may have polarised edges that are not in $K_2$-blocks.
We say that a polar subgraph $Q'$ of $Q$ is a \emph{polar Gallai tree component of $Q$} if $Q'$ is a Gallai tree component of $Q$ and $Q'$ is a polar Gallai tree.
We also say that $Q$ \emph{contains a (polar) Gallai tree component} if there exists a polar subgraph of $Q$ that is a (polar) Gallai tree component (respectively).
To simplify notation, we also say that $Q$ is a Gallai tree if $G(Q)$ is a Gallai tree (where $Q$ might or might not be a polar Gallai tree).

\begin{theorem}
  \label{k restricted theorem}
  Let $Q$ be a polar graph and let $k$ be an integer with $k \geq 3$.
  Let $h \in V(Q)$ such that $d(h) > k$ but $d(u)\leq k$ for all $u \in V(Q) \setminus \{h\}$.
  If $Q$ is $k$-restricted, then there is some $w \in V(Q)$ such that $Q-w$ contains a Gallai tree component $Q'$ 
  such that $d_Q(v) = k$ for all $v \in V(Q')$.
\end{theorem}
\begin{proof}
  Assume that $Q$ is $k$-restricted.
  Let $\Gamma = (L,R)$ be a $k$-polar assignment for $Q$ where $Q$ is $\Gamma$-restricted.
  Then $Q$ has a $\Gamma$-restricted vertex $r$,
  that is, there exists $r \in V(Q)$ with $c \in L(r)$ such that, letting $(Q-r,\Gamma')$ be the $(r,c)$-colour deletion from $(Q,\Gamma)$, we have that $Q-r$ is not $\Gamma'$-colourable.
  Let $V_k = \{v \in V(Q) : d_Q(v)=k\}$.
  Towards a contradiction, assume that there is no $w \in V(Q)$ such that $Q-w$ contains a Gallai tree component whose vertex set is contained in $V_k$.
  Under this assumption, we prove the following sequence of claims.

  \begin{claim}
    \label{h neq v}
    We have that $h \neq r$.
  \end{claim}
  \begin{subproof}
    For every $u \in V(Q) \setminus \{h\}$, we have that $|L(u)|=k\geq d_Q(u)$.
    Let $(Q-h,\Gamma'')$ be a colour-deletion of $h$ from $(Q,\Gamma)$, with $\Gamma''=(L'',R'')$.
    Then, for all $u \in V(Q) \setminus \{h\}$, we have $|L''(u)|\geq d_{Q-h}(u)$ where, as $d_Q(u) \le k = |L(u)|$, we have $|L''(u)|=d_{Q-h}(u)$ only if $u \in V_k$.
    Therefore $\Gamma''$ is a degree-polar assignment for $Q-h$.
    Towards a contradiction, suppose $h=r$.
    Then, by \cref{dcpolarthm}, $Q-h$ contains a polar Gallai tree component $T$, and $\Gamma''$ is a bad degree-polar assignment for $T$.
    Thus $|L''(u)|=d_{Q-h}(u)$ for all $u \in V(T)$.
    Therefore $V(T) \subseteq V_k$, a contradiction. 
    We deduce $h \neq r$, as required.
  \end{subproof}

  By \cref{h neq v}, $h \neq r$.
  Let $Q'=Q-\{r,h\}$ and let $\Gamma' = (L',R')$.
  For each $a \in L'(h)$, let $(Q',\Gamma_a)$ be the $(h,a)$-colour deletion from $(Q-r,\Gamma')$.
  Note that
  $\Gamma_a$ is a degree-polar assignment for $Q'$, for each $a\in L'(h)$.
  If $Q'$ is $\Gamma_a$-colourable, then $Q-r$ is $\Gamma'$-colourable, a contradiction.
  Therefore, by \cref{dcpolarthm}, for every $a \in L'(h)$ there is a polar Gallai tree component $T_a$ of $Q'$ such that $\Gamma_a$ is a bad degree-polar assignment for $T_a$.
  Then any such $T_a$ has $V(T_a) \subseteq V_k$. 

  \begin{claim}
    \label{only one}
    There exist distinct $a,b \in L'(h)$ such that there is a polar Gallai tree component $T$ of $Q'$, where $\Gamma_a$ and $\Gamma_b$ are both bad degree-polar assignments for $T$.
  \end{claim}
  \begin{subproof}
    Towards a contradiction, assume no such $T$ exists.
    Then, for every $a \in L'(h)$, there is a polar Gallai tree component $T_a$ of $Q'$ and a bad degree-polar assignment $\Gamma_a$ for $T_a$, where the members of $\{T_a : a \in L'(h)\}$ are pairwise distinct.
    Therefore, $Q'$ has at least $|L'(h)|$ polar Gallai tree components.
    By assumption, $Q-h$ does not contain a Gallai tree component whose vertex set is contained in $V_k$.
    For each $T_a$, we have $V(T_a) \subseteq V_k$, so $T_a$ is not a component of $Q-h$.
    That is, each $T_a$ contains at least one vertex in $N_Q(r)$.

    If $h \notin N_Q(r)$, then $|L'(h)|=k$, so there are $k$ such $T_a$ components; 
    whereas if $h \in N_Q(r)$, then $|L'(h)| \in \{k-1,k\}$, so there are at least $k-1$ such $T_a$ components.
    In either case, since $d_Q(r) \leq k$ and each $T_a$ contains a vertex in $N_Q(r)$, the vertex~$r$ has precisely one neighbour in each $T_a$ (and, in the latter case, one neighbour that is $h$), and $r \in V_k$.
    Then $Q[V(T_a) \cup \{r\}]$ is a Gallai tree, where the edge between $r$ and $T_a$ forms a $K_2$-block.
    It follows that $Q[\{r\} \cup \bigcup_{a \in L'(h)} V(T_a)]$ is a Gallai tree component of $Q-h$ with all vertices in $V_k$, a contradiction.
  \end{subproof}

  Let $T$ be a polar Gallai tree component of $Q'$ described by \cref{only one}.
  Let $T^+=Q[V(T) \cup \{h\}]$.
  We obtain a polar assignment $\Gamma^+=(L^+,R^+)$ for $T^+$ from $\Gamma'$ by removing each colour $a \in L'(h)$ for which $\Gamma_a$ is not a bad degree-polar assignment for $T$.
  By \cref{only one}, $|L^+(h)| \geq 2$.
  Now $T^+$ is not $\Gamma^+$-colourable, by construction.
  We also have that $|L^+(u)| \geq d_{T^+}(u)$ for all $u \in V(T^+)\setminus \{h\}$, and $T$ is connected.
  Therefore, by \cref{dcpolarhdthm}, $(T^+,h) \in \mathcal{N}_{|L^+(h)|}$, and $\Gamma^+$ is an $h$-bad polar assignment for $T^+$.
%
  Moreover, $V(T) \subseteq V_k$.
  Note that 
  $|L^+(u)|=d_{T^+}(u)$ for all $u \in V(T^+) \setminus \{h\}$, 
  and, by \ref{N3}, there are no parallel edges between $h$ and a vertex of $T$.
  We continue to work towards a contradiction.

  \begin{claim}
    \label{fullblocks}
    Let $B$ be a block of $T$.  Then $d_B(u) = d_B(v) \le k-1$ for all $u,v \in V(B)$.
  \end{claim}
  \begin{subproof}
    Since $T$ is a Gallai tree, $d_B(u) = d_B(v)$ for all $u,v \in V(B)$. 
    Let $v \in V(B)$ and observe that $d_B(v) \le d_T(v) \le d_Q(v) \le k$.
    Suppose $d_B(v)=k$.
    Then $v$, and indeed every vertex in $B$, is an internal vertex of $T$, and $v \notin N_Q(r) \cup N_Q(h)$.
    So $B \cong K_{k+1}$ and $B$ is the only block of $T$.
    Thus $B$ is a Gallai tree component in $Q-h$ (and $Q-r$) whose vertex set is contained in $V_k$, a contradiction.
  \end{subproof}

  We call a block $B$ of $T$ {\em full} if $d_B(u)=k-1$ for all $u \in V(B)$.
  By \ref{N1}, each block of $T$ is light, normal, or heavy.
  We call a block {\em terminal} if it is a full light block, or a heavy $K_2$-block that shares a cut vertex with a full light block.

  \begin{claim}
    \label{max to min}
    Let $B$ be a full light block of $T$, and let $u$ be a cut vertex of $T$ in $V(B)$. Then $u \notin N_Q(r) \cup N_Q(h)$, and $u$ is in exactly one other block of $T$, which is a heavy $K_2$-block.
  \end{claim}
  \begin{subproof}
    Since $B$ is full, we have that $d_B(u)=k-1$.
    If $u \in N_Q(r) \cup N_Q(h)$, then $d_T(u) \leq k-1$, implying $u$ is not a cut vertex of $T$.
    Hence $u \notin N_Q(r) \cup N_Q(h)$.
    Therefore, by \ref{G2}, we have that $u$ is in a heavy block $S$ of $T$.
    Since $d_B(u)=k-1$ and $d_B(u)+d_S(u) \leq k$, we have that $d_S(u)=1$, so $S$ is a $K_2$-block, and $u$ is not in any other blocks of $T$.
  \end{subproof}

  \begin{claim}
    \label{not bad 2}
    Let $u \in V(T) \cap N_Q(r)$. Then there is some block $B$ of $T$ containing $u$ such that $B$ is not terminal.
  \end{claim}
  \begin{subproof}
    Assume that all blocks of $T$ that contain $u$ are terminal.
    Suppose $u$ is an internal vertex of $T$.
    Then $u$ is not in a heavy block, by \cref{list}\ref{sync only cut}, so $u$ is in a full light block.
    By \ref{N1} and \ref{G1}, this implies that $u \in N_Q(h)$.
    So $u \in N_Q(r) \cap N_Q(h)$, and therefore $d_T(u) \leq k-2$, contradicting that $u$ is in a full block.
    Now suppose $u$ is a cut vertex of $T$.
    Since every block of $T$ containing $u$ is terminal, either $u$ is in a heavy block, or $u$ is in a full light block.
    If $u$ is in a heavy block, then $u \notin N_Q(h)$, by \cref{list}\ref{V only async}; whereas if $u$ is in a full light block, then, by \cref{max to min}, $u \notin N_Q(h)$.
    Since $u \notin N_Q(h)$ in either case, \cref{list}\ref{several conditions} implies that $u$ is in exactly one light block $A$ and exactly one heavy block $S$.
    But then $d_A(u)=k-1$ and $d_S(u)=1$, so $d_A(u)+d_S(u)=k$ and therefore $u \notin N_Q(r)$, a contradiction.
  \end{subproof}

  \begin{claim}
    \label{some not maximal}
    There exists at least one light block of $T$ that is not full.
  \end{claim}
  \begin{subproof}
    Observe that $T$ has at least one vertex in $N_Q(h)$, for otherwise $T$ is a Gallai tree component of $Q-r$ with all vertices in $V_k$, contrary to our initial assumption.
    If $u \in N_Q(h) \cap V(T)$, then, by \cref{list}\ref{V means one async}, $u$ is in some light block of $T$.
    So $T$ has at least one light block.

    Assume that all light blocks of $T$ are full.
    By \cref{not bad 2}, if every block of $T$ is terminal, then $N_Q(r) \cap V(T) = \emptyset$, implying $T$ is a Gallai tree component of $Q-h$ with all vertices in $V_k$, a contradiction.
    So $T$ has a block that is not terminal.
    Now $T$ has both terminal and non-terminal blocks, so there exists a cut vertex $v$ of $T$ that is in both a terminal block~$B$ and a non-terminal block~$A$.
    By \cref{max to min}, if $B$ is a full light block, then $A$ is a heavy $K_2$-block and therefore is terminal, a contradiction.
    So $B$ is a heavy $K_2$-block.
    By \cref{list}\ref{V not in sync}, $v \notin N_Q(h)$, and therefore, by \cref{list}\ref{several conditions}, $v$ is in exactly one light block $A'$.
    However, since all light blocks of $T$ are full, $A'$ is full, and therefore, by \cref{max to min}, $A'$ and $B$ are the only blocks containing $v$, and so $v$ is not in a non-terminal block, a contradiction.
  \end{subproof}

  \begin{claim}
    \label{ntinter}
    $L^+(h) \neq L(h)$.
  \end{claim}
  \begin{subproof}
    Suppose that $L^+(h) = L(h)$, so $|L^+(h)|=k$, in which case $(T^+,h) \in \mathcal{N}_k$.
    By \cref{some not maximal}, there exists a light block $B$ of $T$ that is not full.
    Now, for any vertex $v$ of $B$, we have $d_B(v) \geq |L^+(h)|-1=k-1$ by \ref{N2}, and $d_B(v) \le k-1$ by \cref{fullblocks}, so $d_B(v)=k-1$, contradicting that $B$ is not full.
  \end{subproof}

  \begin{claim}
    \label{n is not in t}
    There is at least one vertex in $N_Q(r)$ that is not in $V(T)$.
  \end{claim}
  \begin{subproof}
    Assume that $N_Q(r) \subseteq V(T)$.
    In particular, $h \notin N_Q(r)$, so $L'(h)=L(h)$.
    By \cref{ntinter}, $L^+(h) \neq L'(h)$.
    Since $L^+(h) \subseteq L'(h)$, there exists $d \in L'(h) \setminus L^+(h)$, and a polar Gallai tree component $T_d$ of $Q'$, distinct from $T$, for which $V(T_d) \subseteq V_k$. 
    Now $N_Q(r) \cap V(T_d) = \emptyset$, so $T_d$ is a Gallai tree component of $Q-h$, a contradiction.
  \end{subproof}

  Now, since $d_Q(r) \leq k$, we have $|N_Q(r) \cap V(T)| \le k-1$ by \cref{n is not in t}.

  \begin{claim}
    \label{no complete graphs}
    If $|N_Q(r) \cap V(T)| = k-1$, then $Q[V(T) \cup \{r\}]$ is not a Gallai tree.
  \end{claim}
  \begin{subproof}
    Towards a contradiction, suppose that $|N_Q(r) \cap V(T)| = k-1$ and $Q[V(T) \cup \{r\}]$ is a Gallai tree.
    Then, since $d_Q(r) \leq k$ and by \cref{n is not in t}, $r \in V_k$ and $r$ has precisely one neighbour not in $T$.
%
%
    If there is an edge between $r$ and $h$, then $Q[V(T) \cup \{r\}]$ is a component of $Q-h$, so it is a Gallai tree component with all vertices in $V_k$, a contradiction.
    So there is no edge between $r$ and $h$, and therefore $L'(h)=L(h)$. 
    By \cref{ntinter}, $L^+(h) \neq L'(h)$.
    Since $L^+(h) \subseteq L'(h)$, there is some $d \in L'(h) \setminus L^+(h)$, and a polar Gallai tree component $T_d$ of $Q'$, distinct from $T$, for which $V(T_d) \subseteq V_k$.
    Now $T_d$ is not a component of $Q-h$, for otherwise $T_d$ is a Gallai tree component of $Q-h$.
    So $r$ has a neighbour $u \in V(T_d)$.
    Now $u$ is the unique neighbour of $r$ that is not in $T$, so $Q[V(T_d) \cup \{r\}]$ is a Gallai tree, where $\{r,u\}$ is the vertex set of a $K_2$-block in this Gallai tree.
    It follows that $Q[V(T) \cup V(T_d) \cup \{r\}]$ is a Gallai tree, since it can be obtained by identifying $r$ in the Gallai trees $Q[V(T) \cup \{r\}]$ and $Q[V(T_d) \cup \{r\}]$, with all vertices in $V_k$.
    But $Q[V(T) \cup V(T_d) \cup \{r\}]$ is a component of $Q-h$, so this is a Gallai tree component, a contradiction.
  \end{subproof}

  \begin{claim}
    \label{at least two}
    There are at least two blocks in $T$ that are not terminal.
  \end{claim}
  \begin{subproof}
    By \cref{some not maximal}, there exists a block $B$ of $T$ such that $B$ is light and not full.
    By \cref{fullblocks} and since $B$ is not full, $d_B(v) < k-1$ for each $v \in V(B)$.

    Assume, towards a contradiction, that $B$ is the only non-terminal block of $T$. 
    Let $m=k-1-d_B(u)$ for some arbitrary $u \in V(B)$ (recalling that $d_B(u)=d_B(v)$ for all $u,v \in V(B)$). 
    We show that $e_Q(r,v)=m$ for all $v \in V(B)$.
    Let $v \in V(B)$.
    If $v$ is an internal vertex of $T$, then, by \ref{G1}, $v \in N_Q(h)$ and so $d_{Q-h}(v) = k-1$.
    Since $d_{Q-h}(v) - e_Q(r,v) = d_{Q'}(v) = d_B(v)$, we have that $e_Q(r,v)=m$.
    On the other hand, if $v$ is a cut vertex of $T$, then $v$ is in at least one block other than $B$, and all such blocks are terminal.
    As a terminal block is either light or heavy, \cref{list}\ref{no neighbours} implies that there is exactly one block other than $B$ that contains $v$, and this block is a heavy $K_2$-block.
    Thus $d_T(v)=d_B(v)+1$.
    As $v \notin N_Q(h)$ by \cref{list}\ref{V not in sync}, $d_T(v)=k-e_Q(r,v)$, and so, once again, $e_Q(r,v)=m$.

    Note that $m \ge 1$, since $d_B(v) < k-1$ for all $v \in V(B)$.
    We claim that $B$ is a complete graph.
    Suppose that $B$ is an odd cycle with $|V(B)| \ge 5$.
    Then $m=k-3$, so $k \ge 4$.
    As $|V(B)| \ge 5$ and $e_Q(r,v) =k-3$ for all $v \in V(B)$, there are at least $5(k-3)$ edges between $r$ and vertices of $B$.
    But $5(k-3) > k$ for all $k \ge 4$, so this implies that $d_Q(r) > k$, a contradiction.
    So $B$ is a complete graph with $d_B(u)+1$ vertices.
    Therefore, there are at least $m(d_B(u)+1)=m(k-m)$ edges between $r$ and vertices of $B$.
    By \cref{n is not in t}, there are at most $k-1$ such edges.
    So $m(k-m) \le k-1$, implying $k(m-1) \le (m-1)(m+1)$, and thus $m=1$ or $m = k-1$.  But $m=k-1$ implies that $B$, and hence $T$, consists of a single vertex, a contradiction.
    So $m=1$, and thus $B \cong K_{k-1}$, where $e_Q(r,v)=1$ for each $v \in V(B)$.
    But then $r$ has no other neighbours in $T$, by \cref{n is not in t}, so $Q[V(T) \cup \{r\}]$ is a Gallai tree, contradicting \cref{no complete graphs}.
    We deduce that there are at least two non-terminal blocks in $T$.
  \end{subproof}

  \Cref{at least two} implies, in particular, that there are at least two blocks in $T$, and therefore every block of $T$ has at least one cut vertex.

  \begin{claim}
    \label{non cut vertex claim}
    Let $u$ be an internal vertex of $T$, and let $B$ be the block of $T$ containing $u$.
    \begin{enumerate}
      \item If $u \notin N_Q(r)$, then $u \in N_Q(h)$ and $B$ is a full light block.
      \item If $u \in N_Q(r)$, then either $u \in N_Q(h)$ and $B$ is light and not full, or $u \notin N_Q(h)$ and $B$ is a normal block.
    \end{enumerate}
  \end{claim}
  \begin{subproof}
    The block $B$ contains a cut vertex $v \neq u$.
    Then $d_T(v) > d_T(u)$ and so, since $d_Q(u)=d_Q(v)=k$, we have that $u \in N_Q(r) \cup N_Q(h)$.
    Suppose $u \notin N_Q(r)$. Then $u \in N_Q(h)$, so $B$ is light by \ref{G1}.
    As $u \in N_Q(h) \setminus N_Q(r)$, we have $d_T(u)=k-1$, so $d_B(u)=k-1$ and therefore $B$ is full. So (i) holds.

    Now suppose $u \in N_Q(r)$.
    By \ref{G1}, if $u \in N_Q(h)$ then $B$ is light, whereas if $u \notin N_Q(h)$ then $B$ is normal. 
    Moreover, when $u \in N_Q(r) \cap N_Q(h)$ we have $d_B(u) = d_T(u) \le k-2$, so $B$ is not full.
  \end{subproof}

  \begin{figure}
    \centering
    \includegraphics[width=10.5cm]{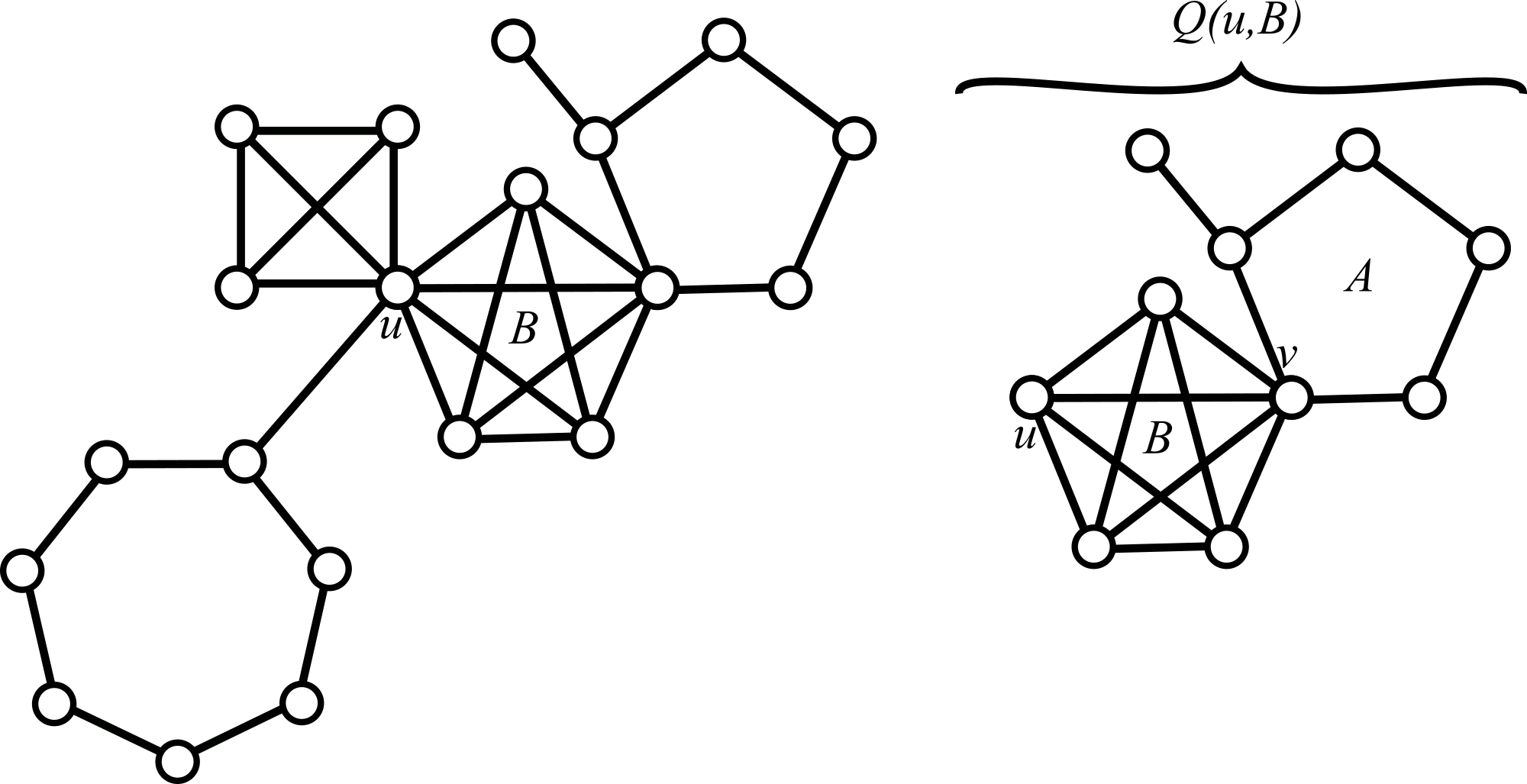}
    \caption{A polar Gallai tree $T$ with a cut vertex $u$ and a block $B$ containing $u$ (left); and $Q_T(u,B)$, with a block $A$ and cut vertex $v$ illustrated (right). Here $Q_T(v,A)$ does not contain $u$, so it is a subgraph of $Q_T(u,B)$.}
    \label{qtub}
  \end{figure}

  Let $u$ be a cut vertex of the polar Gallai tree $T$.
  Then $T$ is the union of some set of polar Gallai trees~$\mathcal{U}$, each of which contains $u$ but has no other vertex in common with a distinct member of $\mathcal{U}$.
  For each block $B$ of $T$ containing $u$, there is a unique polar graph in $\mathcal{U}$ that contains $B$; we denote this polar Gallai tree as $Q_T(u,B)$, see \cref{qtub}.
  Let $A$ be a block of $T$ with $u \notin V(A)$, and let $v$ be a cut vertex of $T$ in $A$.
  Observe that if $Q_T(v,A)$ does not contain $u$, then $Q_T(v,A)$ is a subgraph of $Q_T(u,B)$.


  \begin{claim}
    \label{surprise claim}
    Let $u$ be a cut vertex of $T$ in a block $B$. If $B$ is not terminal, then there is a block $A$ of $Q_T(u,B)$ such that either
    \begin{enumerate}
      \item $A$ is normal and there are at least $k-1$ edges between $r$ and $V(A)\setminus \{u\}$, or
      \item $A$ is light and there are at least $k-2$ edges between $r$ and $V(A) \setminus \{u\}$.
    \end{enumerate}
  \end{claim}
  \begin{subproof}
    The proof is by induction on the number of blocks of $Q_T(u,B)$.
    Since $B$ is a block of $Q_T(u,B)$, for the base case we have that $B$ is the only block of $Q_T(u,B)$.
    Therefore, for any $v \in V(B) \setminus \{u\}$, we have that $v$ is an internal vertex of $T$, and so, by \cref{non cut vertex claim} and since $B$ is not terminal, $v \in N_Q(r)$.
    First, suppose that there exists some $v' \in V(B) \setminus \{u\}$ such that $v' \notin N_Q(h)$, and let $m = e_Q(r,v')$.
    Then, by \cref{non cut vertex claim}, $B$ is normal and $v \notin N_Q(h)$ for all $v \in V(B) \setminus \{u\}$.
    Now $d_{B}(v)=d_T(v)=k-m$ and $e_Q(r,v)=m$ for all $v \in V(B) \setminus \{u\}$.
    Moreover, $|V(B)| \ge k-m+1$, so $|V(B)\setminus \{u\}| \ge k-m$.
    Hence, there are at least $(k-m)m$ edges between $r$ and $V(B) \setminus \{u\}$.
    Clearly $1 \le m$, and $m \leq k-1$ by \cref{n is not in t}.
    Thus $k(m-1) \ge (m+1)(m-1)=m^2-1$, implying $(k - m)m \ge k-1$, so (i) holds in this case.

    Suppose $v \in N_Q(h)$ for all $v \in V(B) \setminus \{u\}$.
    Then $B$ is light, by \cref{non cut vertex claim}. 
    Letting $m = e_Q(r,v')$ for an arbitrary $v' \in V(B) \setminus \{u\}$, we have that $d_{B}(v)=k-m-1$ and $e_Q(r,v)=m$ for all $v \in V(B) \setminus \{u\}$.
    So $|V(B) \setminus \{u\}| = k-m-1$.
    Thus there are at least $(k-m-1)m$ edges between $r$ and $V(B) \setminus \{u\}$.
    Clearly $m \ge 1$ and, since $d_{B}(v) \ge 1$, we have $m \le k-2$.
    Hence $k(m-1) \ge (m+2)(m-1)$, implying $(k-m-1)m \ge k-2$, so (ii) holds, and thus the claim holds when $Q_T(u,B)$ consists of one block.

    Now suppose that $Q_T(u,B)$ has $i$ blocks, for some $i \ge 2$, and the result holds for any cut vertex $u'$ and block $B'$ of $T$ such that $Q_T(u',B')$ has fewer than than $i$ blocks.
    Since $i \ge 2$, there is at least one cut vertex $x$ of $Q_T(u,B)$ in $B$, with $x \neq u$.
    Let $B_1$ be a block of $Q_T(u,B)$ that contains $x$ and is distinct from $B$.
    If $B_1$ is not terminal, then, as $Q_T(x,B_1)$ has fewer than $i$ blocks, by the induction assumption it contains a block $A$ such that either $A$ is normal and there are at least $k-1$ edges between $r$ and $V(A) \setminus \{x\}$, or $A$ is light and there are at least $k-2$ edges between $r$ and $V(A) \setminus \{x\}$.
    Since $Q_T(x,B_1)$ does not contain $u$, it is a subgraph of $Q_T(u,B)$, and so the claim holds in either case.

    Therefore, we may assume that all other blocks that contain $x$ are terminal.
    In particular, $B_1$ is terminal.
    If $B_1$ is a full light block, then by \cref{max to min} we have that $B$ is a heavy $K_2$-block, implying $B$ is terminal, a contradiction.
    Hence $B_1$ is a heavy $K_2$-block.
    By \cref{list}\ref{no neighbours}, $x$ is in at most one heavy block.
    It follows that $B$ and $B_1$ are the only blocks that contain $x$, and, by \cref{list}\ref{sync only cut}, $B$ is light.
    By \cref{list}\ref{V not in sync}, $x \notin N_Q(h)$.
    If $x \notin N_Q(r)$, then $k=d_T(x)=d_{B}(x)+d_{B_1}(x)$, with $d_{B_1}(x)=1$, so $d_{B}(x)=k-1$, implying $B$ is a full light block, contradicting that $B$ is not terminal.
    So $x \in N_Q(r)$.
    Let $m=e_Q(x,r)$. 
    We have that $d_B(x)=k-1-m$, so $|V(B)| \ge k-m$.

    Now let $x' \in V(B) \setminus \{u\}$.  If $x'$ is a cut vertex of $Q_T(u,B)$, then $x'$ has a single neighbour in $Q_T(u,B) \setminus V(B)$ (as $x$ was an arbitrary cut vertex of $Q_T(u,B)$ in $B$ in the foregoing argument).  On the other hand, if $x'$ is an internal vertex of $Q_T(u,B)$, then $x' \in N_Q(h)$, by \ref{G1} and since $B$ is light.
    Thus $d_{Q-r}(x') = k-m$ for all such $x'$, so there are $m$ edges between $r$ and each $x' \in V(B) \setminus \{u\}$.  As $|V(B)| \ge k-m$, there are at least $(k-m-1)m$ edges between $r$ and $V(B) \setminus \{u\}$.
    As before, $(k-m-1)m \ge k-2$, so (ii) holds for the block $B$, as required.
    The result follows by induction.
  \end{subproof}

  By \cref{at least two}, $T$ has distinct blocks $B$ and $B'$ that are not terminal, and there is a cut vertex $u_B \in V(B)$ that separates $V(B) \setminus \{u_B\}$ from $V(B') \setminus \{u_B\}$, and a cut vertex $u_B' \in V(B')$ that separates $V(B') \setminus \{u_B'\}$ from $V(B) \setminus \{u_B'\}$.
  Note that if $B$ and $B'$ share a cut vertex, then $u_B=u_B'$; otherwise, $Q_T(u_B,B)$ and $Q_T(u_B',B')$ are vertex disjoint.

  By \cref{surprise claim}, there is a block $A$ of $Q_T(u_B,B)$ such that there is a set $Z$ of at least $k-2$ edges between $r$ and $V(A) \setminus \{u_B\}$, and there is a block $A'$ of $Q_T(u_B',B')$ such that there is a set $Z'$ of at least $k-2$ edges between $r$ and $V(A') \setminus \{u_B'\}$.
  Since $Q_T(u_B,B)$ and $Q_T(u_B,B')$ are vertex disjoint except for potentially $u_B$ and $u_B'$, $Z$ and $Z'$ are disjoint, so there are at least $2(k-2)$ edges between $r$ and $V(T)$.
  Since $d_Q(r) \le k$, we have $k \le 4$.
  Moreover, if $k=4$, then $N_Q(r) \subseteq V(T)$, which contradicts \cref{n is not in t}.

  It remains to examine the case when $k=3$.
  By \cref{n is not in t}, $N_Q(r) \not\subseteq V(T)$, so there are at most two edges between $r$ and $V(T)$.
  Therefore, there is exactly one edge between $r$ and $V(A) \setminus \{u_B\}$ and exactly one edge between $r$ and $V(A') \setminus \{u_B'\}$.
  By \cref{surprise claim}, $A$ and $A'$ are light.
  Let $z \in V(A) \cap N_Q(r)$.
  If $z$ is an internal vertex of $T$, then $A$ is a $K_2$-block by \cref{non cut vertex claim}.
  Otherwise, $z$ has at least one neighbour in $V(T) \setminus V(A)$, so, as $d_Q(z) = 3$, we have $d_A(z) =1$, and again $A$ is a $K_2$-block.
  So $A$ and, similarly, $A'$ are light $K_2$-blocks.
  In particular, $A$ and $A'$ are not terminal.
  There is a unique path in the block-cut graph of $G(T)$ between $A$ and $A'$.
  Let $T_B$ be the induced subgraph of $T$ consisting of all vertices in the blocks of this path; then $T_B$ consists of a path of blocks, each of which is also a block in $T$, and where $A$ and $A'$ are the leaf blocks of $T_B$.
  Let $V(A) = \{v_A,u_A\}$ and $V(A') = \{v'_A,u'_A\}$ where $u_A$ and $u_A'$ are cut vertices of $T_B$.
  Observe that the only vertex in $Q_T(u_A,A)$ that neighbours $r$ in $Q$ belongs to $V(A)$.
  By \cref{surprise claim}, as $A$ is not terminal, $v_A \in N_Q(r)$.
  Similarly, $v'_A \in N_Q(r)$.

  \begin{claim}
    \label{endgameclaim}
    Each block of $T_B$ is a $K_2$-block that is either light or normal.
    Moreover, each cut vertex of $T_B$ is in two blocks: one light, and one normal.
  \end{claim}
  \begin{subproof}
    Let the blocks of $T_B$ be $B_1,B_2,\dotsc,B_m$, where $B_1=A$, $B_m=A'$, and $B_j$ and $B_{j-1}$ share a cut vertex $u_j$ for each $j \in \{2,\dotsc,m\}$.
    We first show that each $B_j$, for $j \in \{1,2,\dotsc,m\}$, is a light or normal $K_2$-block.
    Towards a contradiction, suppose there exists some such $j$ such that $B_j$ is not a light or normal $K_2$-block; choose $i$ to be the smallest such $j$ (see \cref{endgameclaimfig}).
    Since $A$ is light, $i \ge 2$.
    So $B_i$ has a cut vertex $u_i$ that separates $V(B_i) \setminus \{u_i\}$ from $V(A) \setminus \{u_i\}$, and $u_i$ belongs to one other block, $B_{i-1}$, which is a light or normal $K_2$-block.

    \begin{figure}[b]
      \centering
      \includegraphics[width=10.5cm]{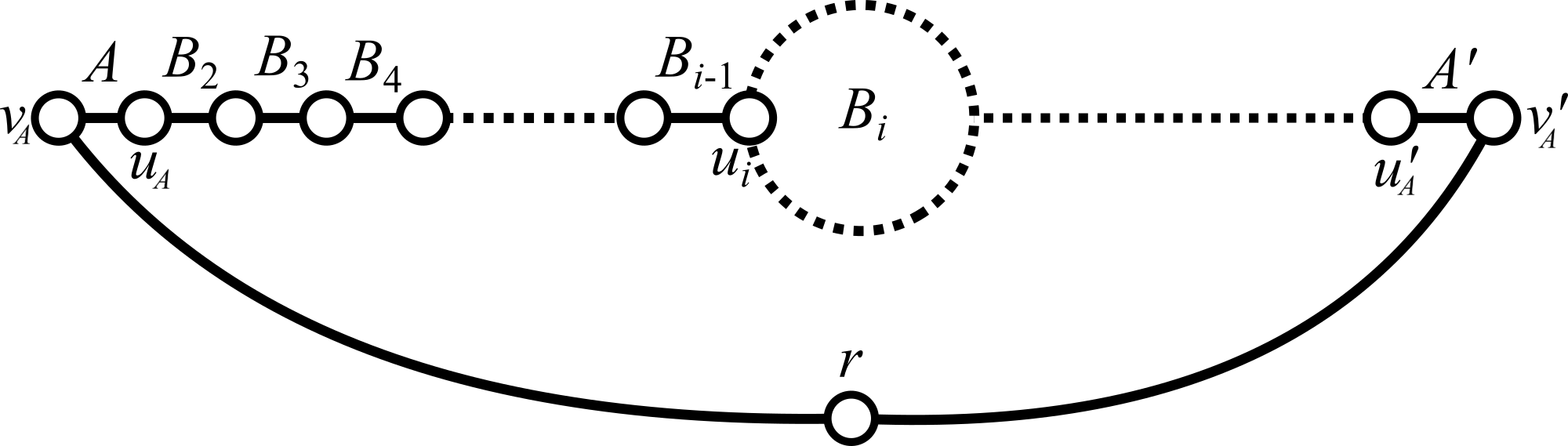}
      \caption{The graph induced by $V(T_B) \cup \{r\}$, with the blocks $B_1,\dotsc,B_m$ of $T_B$ illustrated. The block $B_i$ is the first block that is not a light or normal $K_2$-block.}
      \label{endgameclaimfig}
    \end{figure}

    Recall that $u_i \notin N_Q(r)$.
    So either $u_i \in N_Q(h)$ and $d_T(u_i)=2$, or $u_i \notin N_Q(h)$ and $d_T(u_i)=3$.
    Suppose we are in the first case.
    Since $B_{i-1}$ is a $K_2$-block, $B_i$ is also a $K_2$-block.
    By \ref{G2}, if $B_{i-1}$ is light then $B_i$ is normal, and if $B_{i-1}$ is normal then $B_i$ is light.
    So $B_i$ is a $K_2$-block that is light or normal, a contradiction.
    So we may assume $u_i \notin N_Q(h)$ and $d_T(u_i)=3$.
    To simplify notation, let $d_i=d_{B_i}(u_i)$ and $d_{i-1}=d_{B_{i-1}}(u_{i-1})$.
    Then $d_{i-1}=1$ and, as $u_i \notin N_Q(h)$, we have $d_i \in \{1,2\}$.

    Suppose $d_i=1$.
    Then, since $d_T(u_i)=3$ and $d_{i-1}=1$, there is some third block $B_i'$ that contains $u_i$, and which is also a $K_2$-block.
    By \ref{G2}, either $B_i$, $B_{i-1}$, and $B_i'$ are all normal; or one is light, one is heavy, and one is normal.
    By assumption, $B_i$ is heavy, so $B_i'$ is either light or normal, and so is not terminal.
    Therefore, by \cref{surprise claim}, there is a block $A_i$ of $Q_T(u_i,B_i')$ such that there is an edge between $r$ and a vertex of $V(A_i) \setminus \{u_i\}$.
    But then $r$ has a neighbour in $T$ distinct from $v_A$ and $v_A'$, a contradiction.

    Now we may assume that $d_i=2$.
    Then the only blocks of $T$ that contain $u_i$ are $B_i$ and $B_{i-1}$.
    By \ref{G2}, if $B_{i-1}$ is normal then $B_i$ is normal, and if $B_{i-1}$ is light then $B_i$ is heavy.
    In any case, $B_i$ is not light.
    Since $d_i=2$, we have $|V(B_i)| \ge 3$.
    As $B_i \neq A'$, we choose $x \in V(B_i) \setminus \{u_i,u_{i+1}\}$, and have $x \notin N_Q(r)$.
    If $x \in N_Q(h)$, then $d_T(x)=2=d_i=d_{B_i}(x)$, implying $x$ is an internal vertex of $T$, so $B_i$ is light by \cref{non cut vertex claim}, a contradiction.
    Therefore $x \notin N_Q(r) \cup N_Q(h)$ and $d_T(x)=3$.
    Since $d_{B_i}(x)=2$, this means there is another block $B_x$ containing $x$, and $B_x$ is a $K_2$-block.
    Then the only blocks containing $x$ are $B_i$ and $B_x$, where $B_i$ is not light, so, by \ref{G2}, $B_x$ is not heavy.
    Thus $B_x$ is not terminal, and so, by \cref{surprise claim}, $Q_T(x,B_x)$ contains a vertex in $N_Q(r)$, distinct from $v_A$ and $v_A'$, a contradiction.
    Hence all blocks of $T_B$ are light or normal $K_2$-blocks, as required.

    Now, by \cref{endgameclaim}, $T_B$ is a path with leaves $v_A$ and $v_A'$.
    Let $x \in V(T_B) \setminus \{v_A,v_A'\}$.
    By the foregoing, $x$ belongs to two $K_2$-blocks in $T_B$ that are light or normal.
    It remains to show that $x$ belongs to one light block and one normal block.
    We have $d_Q(x)=3$ and $d_{T_B}(x)=2$, and $x \notin N_Q(r)$.
    By \ref{G2}, if $x \in N_Q(h)$, then $x$ belongs to a light block and every other block containing $x$ is normal, as required.
    So assume $x \notin N_Q(h)$.
    Then $x$ belongs to another block $X$ of $T$ that is not in $T_B$, and $X$ is a $K_2$-block.
    Suppose $X$ is not heavy.
    Then $X$ is not terminal, and so, by \cref{surprise claim}, there is a block $A'$ of $Q_T(x,X)$ that has a vertex in $N_Q(r)$, distinct from $v_A$, and $v_A'$, a contradiction.
    So $X$ is heavy.
    But then, by \ref{G2}, one of the blocks of $T_B$ containing $x$ is light and the other normal, as required.
  \end{subproof}

  Now $T_B$ alternates between light $K_2$-blocks and normal $K_2$-blocks, starting and ending with light blocks.
  Hence $T_B$ has an odd number of blocks.
  Therefore $Q[V(T_B) \cup \{r\}]$ is an odd cycle, and hence $Q[V(T) \cup \{r\}]$ is a Gallai tree, with $|N_Q(r) \cap V(T)| = k-1$, contradicting \cref{no complete graphs}. 
\end{proof}

For graphs (that is, when we have no polarised edges), we obtain the following characterisation:

\begin{theorem}
  \label{krestcor}
  Let $G$ be a simple graph and let $k$ be an integer with $k \geq 3$.
  Let $h \in V(G)$ such that $d(h) > k$ but $d(u) \leq k$ for all $u \in V(G) \setminus \{h\}$.
  The graph $G$ is $k$-restricted if and only if there is some $r \in V(G)$ such that $G-r$ contains a Gallai tree component $H$ such that $d_G(v)=k$ for all $v \in V(H)$.
\end{theorem}
\begin{proof}
  One direction holds by \cref{k restricted theorem}.
  For the other direction, let $V_k=\{v \in V(G) : d(v)=k\}$, and assume that there is some $r \in V(G)$ such that $G-r$ contains a Gallai tree component $H$ with $V(H) \subseteq V_k$.
  Then each $v \in V(H)$ has $k$ neighbours in $G$, where at most one, $r$, is not in $V(H)$.
  So each $v \in V(H)$ has $d_H(v) \in \{k-1,k\}$.
  Let $L$ be a bad degree-list assignment for $H$, and note that $|L(v)|=k$ if $v \notin N_G(r)$ and $|L(v)|=k-1$ if $v \in N_G(r)$.
  Let $c \in \mathbb{N} \setminus \bigcup_{v \in N_G(r)} L(v)$, and let $L'$ be a $k$-list assignment for $G$ obtained by letting $L'(v)=L(v) \cup \{c\}$ if $v \in N_G(r) \cap V(H)$, letting $L'(v)=L(v)$ if $v \in V(H) \setminus N_G(r)$, and letting $L'(v)$ be an arbitrary list of size $k$, but with $c \in L'(r)$.
  Then, by construction, the $(r,c)$-colour deletion from $(G,L')$ results in a graph with a Gallai tree component having a bad degree-list assignment, and is therefore not colourable.
  Therefore $r$ is $L'$-restricted on $c$, so $G$ has a $k$-restricted vertex, as required.
\end{proof}

We are primarily interesting in the case where the polar graph is $2$-connected, in which case we obtain a slightly stronger outcome.

\begin{corollary}
    \label{krtcorollary}
    Let $Q$ be a $2$-connected polar graph and let $k$ be an integer with $k \geq 3$.
    Let $h \in V(Q)$ such that $d(h) > k$ but $d(u)\leq k$ for all $u \in V(Q) \setminus \{h\}$.
    If $Q$ is $k$-restricted, then $G(Q-h)$ is a Gallai tree and $d_Q(v)=k$ for all $v \in V(Q) \setminus \{h\}$.
\end{corollary}
\begin{proof}
    By \cref{k restricted theorem}, there is some $w \in V(Q)$ such that $G(Q-w)$ contains a Gallai tree component~$H$, where $d_Q(u)=k$ for all $u \in V(H)$.
    Since $Q$ is $2$-connected, $Q-w$ is connected, so $H=G(Q-w)$, and $h=w$.
    The result follows.
\end{proof}

\section{The class \texorpdfstring{$\mathcal{T}_k^-$}{Tk-}}

Let $k \ge 3$.
For a graph $H$ with maximum degree at most $k$ and minimum degree strictly less than $k$, recall that the \emph{almost $k$-regular extension} of $H$, denoted $H^+_k$, is the graph obtained from $H$ by adding a vertex $h$ and, for each vertex $v \in V(G)$, adding $k-d_H(v)$ edge(s) in parallel between $h$ and $v$.
%
We define $\mathcal{T}_k^-$ to be the class of graphs consisting of the almost $k$-regular extension of a Gallai tree, with maximum degree at most $k$ and minimum degree less than $k$, containing no $K_2$-blocks.
See \cref{t5graphs} for some examples of graphs in $\mathcal{T}_5^-$.

\begin{figure}
    \centering
    \includegraphics[scale=0.5]{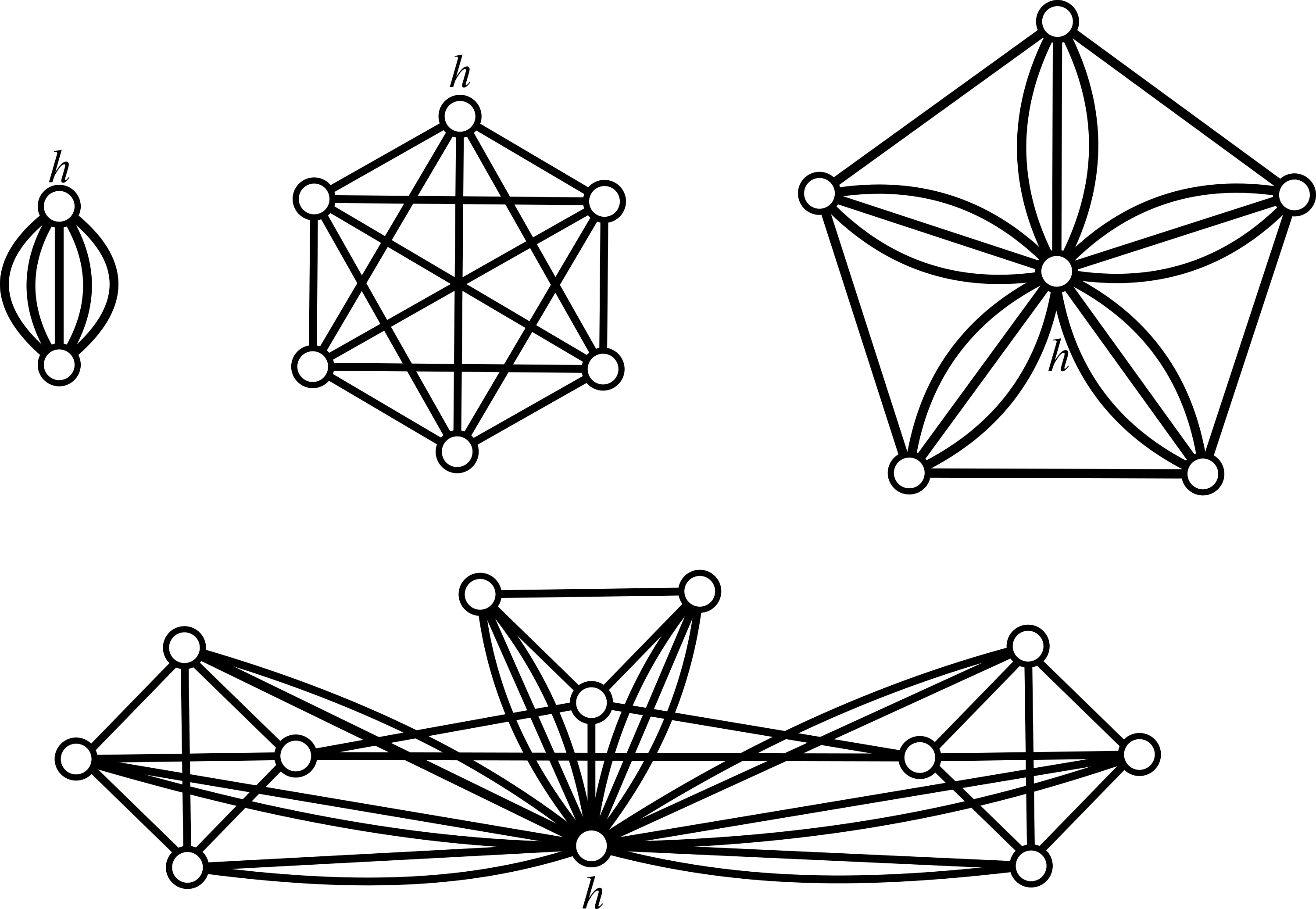}
    \caption{Some graphs in $\mathcal{T}_5^-$.}
    \label{t5graphs}
\end{figure}

In the next section, we close $\mathcal{T}_k^-$ under Haj\'os joins in order to obtain $\mathcal{T}_k$.
Note that this definition of $\mathcal{T}_k$ slightly differs from that given in the introduction, due to the exclusion of $K_2$-blocks when defining $\mathcal{T}_k^-$.
However, later we will see \cref{Hajos joins give k2}, which implies that if $G$ is the almost $k$-regular extension of a Gallai tree (even one with $K_2$-blocks), then $G \in \mathcal{T}_k$.
So the two definitions do in fact coincide.

In this section, we first prove some structural properties of graphs in $\mathcal{T}_k^-$.
We then consider polar odd wheels, which are a class of polar graphs whose underlying graphs are in $\mathcal{T}_3^-$.
We next characterise when a polar graph whose underlying graph is in $\mathcal{T}_k^-$ is $k$-restricted, and consider some properties of such polar graphs when there is a vertex that is $k$-restricted on more than one colour (see \cref{tkminusone}).
Finally, we introduce the notion of $k$-joinable polar graphs, and show that a polar graph whose underlying graph is in $\mathcal{T}_k^-$ is $k$-joinable.

\subsection{Structural properties}

Suppose $G \in \mathcal{T}_k^-$, for $k \ge 3$.
We say that $h \in V(G)$ is a \emph{hub} of $G$ if $G$ is isomorphic to the almost $k$-regular extension of $G-h$, where $G-h$ is a Gallai tree with maximum degree at most $k$ and minimum degree less than $k$.
Clearly $G$ has at least one hub, but it is possible that it has more than one.
We say that $G$ is \emph{$k$-symmetric} if $G \cong K_{k+1}$ or $G \cong I_k$, where $I_k$ is the graph consisting of two vertices with $k$ edges between them.
When $G$ is $k$-symmetric, then $G$ is the almost $k$-regular extension of $G-v$ for any $v \in V(G)$; that is, every vertex of $G$ is a hub.
The next lemma asserts that when $G$ is not $k$-symmetric, it has just one hub.
Moreover, we prove a further useful property when $G$ is not $k$-symmetric.

\begin{lemma}
    \label{no parallel edges lemma}
    Let $G$ be a graph in $\mathcal{T}_k^-$, for $k \ge 3$, that is not $k$-symmetric.
    \begin{enumerate}
        \item $G$ has a unique hub $h \in V(G)$.
        \item Moreover, if $G$ is not an odd wheel, then $e_G(u,h) \ge 2$ for every internal vertex $u$ of $G-h$.
    \end{enumerate}
\end{lemma}
\begin{proof}
    Let $h$ be a hub of $G$, so $(G-h)_k^+ \cong G$.
    Then every vertex of $G$ except perhaps $h$ has degree $k$ in $G$.
    So clearly (i) holds if $G$ is not $k$-regular.
    In particular, this is the case if 
    $G$ is an odd wheel that is not a complete graph, so henceforth we assume that $G$ is not an odd wheel.
    Furthermore, if $|V(G)|=2$, then $G \cong I_k$, a contradiction.
    So $|V(G)| \ge 3$.

    Let $T=G-h$.  Then $T$ is a Gallai tree with $|V(T)| \ge 2$.
    It follows that $T$ has at least two internal vertices.
    Note that, since a Gallai tree is simple, any parallel edges of $G$ are incident with $h$.  Thus (i) holds if $G$ has distinct parallel classes.
    It now suffices to show that (ii) holds, as this implies that $G$ has distinct parallel classes, which in turn implies that (i) holds.
    
    Suppose $T$ consists of a single block.
    Then every vertex in $T$ is an internal vertex, and has the same degree $d$.
    Note that $d < k$, since the minimum degree of $T$ is less than $k$.
    If $d=k-1$, then $T$ is an odd cycle when $k=3$, and $T \cong K_k$ when $k > 3$.
    In the former case $G$ is an odd wheel, and in the latter case $G \cong K_{k+1}$; either case is contradictory.
    So $d \le k-2$, in which case there are at least two edges between $h$ and each vertex in $T$, so (ii) holds. 

    We may now assume that $T$ has more than one block.
    Then each block of $T$ contains at least one cut vertex of $T$.
    Let $u$ be an internal vertex of a block $B$ of $T$, and let $x$ be a cut vertex in $B$.
    Then $d_T(u) = d_B(u) = d_B(x) < d_T(x) \le k$, 
    so $d_T(u) \le k-1$.
    If $d_T(u) = k-1$, then $d_B(x) = k-1$ and $d_T(x) = k$, so $x$ belongs to a $K_2$-block, a contradiction.
    So $d_T(u) \le k-2$.
    Then there are at least two parallel edges between $u$ and $h$ in $G$, so (ii) holds.
\end{proof}

\begin{lemma}
    \label{tkcliques}
    Let $G \in \mathcal{T}_k^-$, for $k \ge 3$, and let $U \subseteq V(G)$ such that $G[U] \cong K_k$. Then $G \cong K_{k+1}$.
\end{lemma}
\begin{proof}
    Since $G \in \mathcal{T}_k^-$, we have that $G \cong H_k^+$ where $H$ is a Gallai tree with no $K_2$-blocks, having maximum degree $k$ and minimum degree less than $k$.
    Let $h$ be a hub of $G$.

    We first prove the lemma for the case where $h \notin U$.
    So assume $h \notin U$.
    Then $U \subseteq V(H)$, and so 
    $U$ is contained in a block $B$ of $H$.
    Suppose $U$ is properly contained in $V(B)$. Then there exists $v \in V(B) \setminus U$ such that $v$ has a neighbour $u$ in $U$, and therefore $u$ has degree at least $k$ in $H$. But, as $H$ is a Gallai tree with maximum degree at most $k$, this implies that $H$ consists of a single block where each vertex has degree~$k$, so has minimum degree~$k$, a contradiction.
    We deduce that $U = V(B)$.
    If $B$ is the only block of $H$, then $G \cong K_{k+1}$, as required.
    So we may assume that $B$ contains a cut vertex that is in some block $B'$ distinct from $B$.
    Since each vertex in $U$ has degree $k-1$ in $B$, and $H$ has maximum degree $k$, we have that $B'$ is a $K_2$-block, a contradiction.
    So the lemma holds when $h \notin U$.

    Now assume that $h \in U$.
    Note that $G[U \setminus \{h\}] \cong K_{k-1}$, where there is precisely one edge between $h$ and $u$ for each $u \in U \setminus \{h\}$.
    This implies that each $u \in U \setminus \{h\}$ has degree $k-1$ in $H$, so has precisely one neighbour $v_u$ in $V(H) \setminus U$.
    Since $H$ has no $K_2$-blocks, the block containing $u$ and $v_u$ contains some vertex $w_u$ that neighbours both $u$ and $v_u$.
    Since $v_u$ is the unique neighbour of $u$ that is in $V(H) \setminus U$, we have $w_u \in U$.
    Therefore $v_u$ is in the same block of $H$ as $u$ and $w_u$, that is, the block containing $U \setminus \{h\}$.
    It now follows that the vertices $U'=U \setminus \{h\} \cup \{v_u\}$ form a block of $H$, with $G[U'] \cong K_{k}$, and $h \notin U'$.  By the previous paragraph, $G \cong K_{k+1}$, as required.
\end{proof}


\begin{lemma}
\label{k edge con lemma}
Let $G$ be a graph in $\mathcal{T}_k^-$, for $k \ge 3$. Then $G$ is $k$-edge-connected.
\end{lemma}
\begin{proof}
    Let $G = T_k^+$ for some Gallai tree $T$ with maximum degree at most $k$ and minimum degree less than $k$, and let $h$ be a hub of $T$.  
    It suffices to show that if $v \in V(G) \setminus \{h\}$, then there are $k$ edge-disjoint paths between $v$ and $h$ in $G$.
    Suppose this is not the case for some $v \in V(G) \setminus \{h\}$.
    We assume that, among all such counterexamples, $G$ has the minimum number of vertices.

    First assume that $T$ consists of only one block.
    Let $S \subseteq E(G)$ with $|S|<k$ such that $v$ and $h$ are in different components of $G\ba S$.
    Clearly, if $|V(T)|=1$, then $G \cong I_k$, so no such $S$ exists, a contradiction.
    Therefore we may assume that $|V(T)| \ge 2$.
    Since $T$ is either a complete graph or odd wheel, every vertex has the same degree $d_T(v)$.
    By definition, there are $k-d_T(v)$ edges between $v$ and $h$, and so $S$ contains these $k-d_T(v)$ edges.
    Then, since each vertex in $N_T(v)$ is also in $N_G(h)$, such an edge cut also contains at least one edge incident with each vertex in $N_T(v)$.
    As there are $d_T(v)$ such vertices, $|S| \ge k-d_T(v)+d_T(v)=k$, a contradiction.

    So we may assume that $T$ has more than one block, in which case $T$ has a cut vertex~$x$.
    Let $T_0$ and $T_1$ be Gallai trees such that $V(T_0) \cap V(T_1) = \{x\}$ and $V(T_0) \cup V(T_1) = V(T)$.
    Let $G_i$ be the almost $k$-regular extension of $T_i$ with hub $h$, for $i \in \{0,1\}$.
    Then, for each $i\in \{0,1\}$, we have $G_i \in \mathcal{T}_k^-$, with $V(G_i) \subset V(G)$.
    We also note that each edge in $G_0$ corresponds to an edge in $G$ except for an extra $d_{T_1}(x)$ edges between $x$ and $h$ in $G_0$.
    Similarly, each edge in $G_1$ corresponds to an edge in $G$ except for an extra $d_{T_0}(x)$ edges between $x$ and $h$ in $G_1$.
    Without loss of generality, assume that $v \in V(G_0)$.
    By the minimality of $G$, there exists a set $\mathcal{P}_0$ of $k$ edge-disjoint paths between $v$ and $h$ in $G_0$, at most $d_{T_1}(x)$ of which use edges that are not in $G$.
    Similarly, there exists a set $\mathcal{P}_1$ of $k$ edge-disjoint paths between $x$ and $h$ in $G_1$, at most $d_{T_0}(x)$ of which use edges that are not in $G$.
    Therefore, there are at least $k-d_{T_0}(x)$ paths in $\mathcal{P}_1$ that use edges in $G$.
    Note that $k-d_{T_0}(x)=k-(d_T(x)-d_{T_1}(x))=k-d_T(x)+d_{T_1}(x) \ge d_{T_1}(x)$. 
    By taking the union of the paths in $\mathcal{P}_0$ that use edges of $G$, of which there are at least $k-d_{T_1}(x)$, and the paths in $\mathcal{P}_1$ that use edges of $G$, of which there are at least $d_{T_1}(x)$, we deduce that there are $k$ edge-disjoint paths between $v$ and $h$ in $G$, as required.
\end{proof}

\subsection{Polar odd wheels}
\label{powsec}

We say that a polar graph $Q$ is a \emph{polar odd wheel} if $G(Q)$ is an odd wheel.
Recall that the class of odd wheels is contained in $\mathcal{T}_3^-$.
In this section, we 
characterise when a polar odd wheel is $3$-restricted, consider properties of a $3$-polar assignment for which a vertex is $3$-restricted on more than one colour, and show that any proper polar subgraph of a polar odd wheel is $3$-unrestricted.
We start with some lemmas that describe how $3$-restricted vertices can appear in polar odd wheels.

Let $G$ be an odd wheel. 
Recall that a dominating vertex $h$ of $G$ is called a \emph{hub}, so $G-h$ is an odd cycle.
Note that either $G \cong K_4$ and every vertex is dominating, in which case we choose the hub $h$ arbitrarily, or $G$ has a unique hub $h$.
Suppose that $v \in V(G) \setminus \{h\}$.
For an edge $xy$ of $G-h$, we have that $G'=(G-h)\ba xy$ is a path with even length, so the distances $d_{G'}(x,v)$ and $d_{G'}(y,v)$ have the same parity.
If $d_{G'}(x,v)$ and $d_{G'}(y,v)$ are both even, then we say that $xy$ has {\em even displacement to $v$} in $G-h$.

\begin{figure}[hb]
    \centering
    \begin{subfigure}{0.4\textwidth}
        \centering
        \includegraphics[scale=0.5]{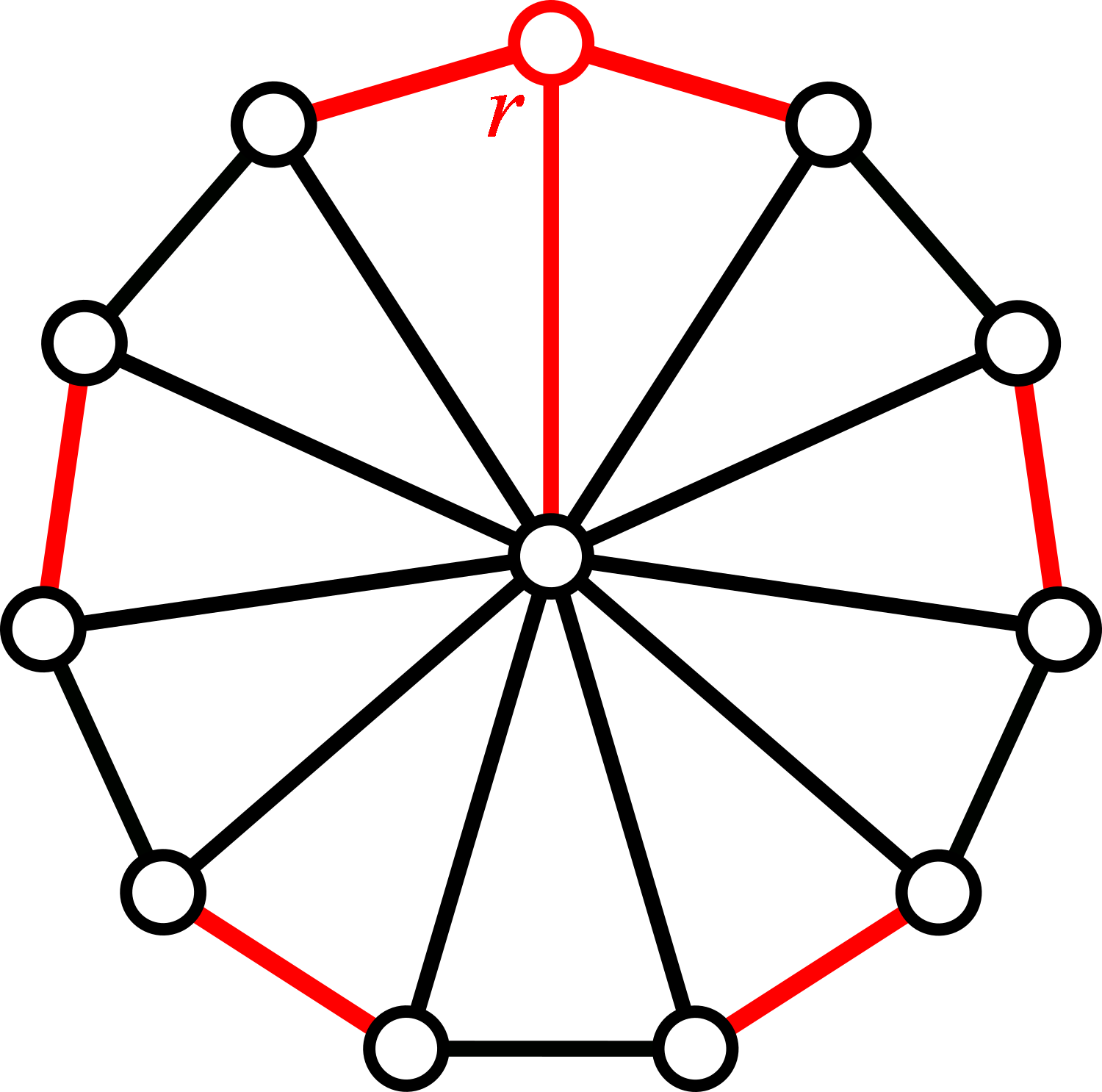}
        \subcaption{An odd wheel $G$ with a vertex $r$, where $\Ev_G(r)$ consists of the red edges.}
        \label{pow1fig}
    \end{subfigure}\quad
    \begin{subfigure}{0.4\textwidth}
        \centering
        \includegraphics[scale=0.45]{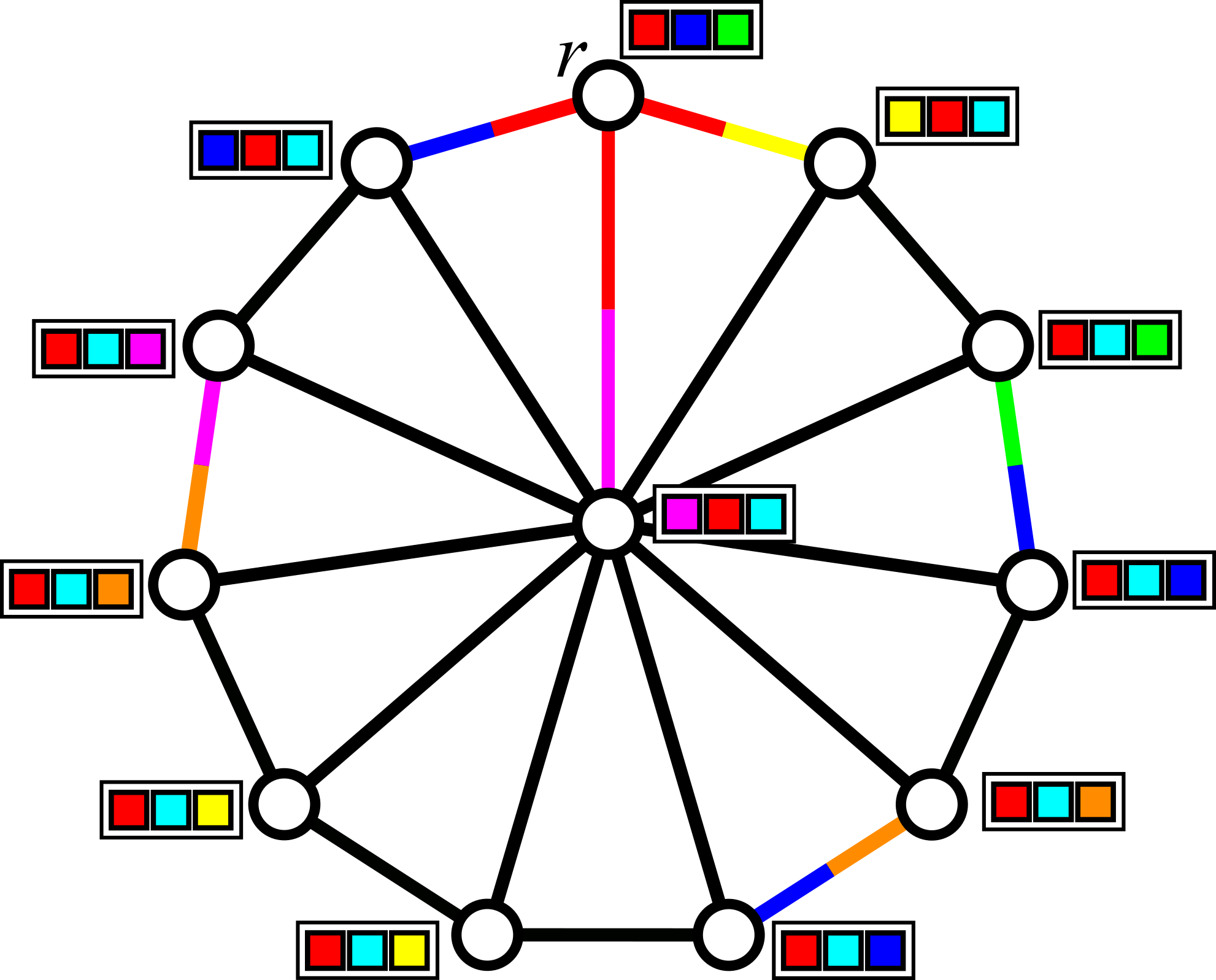}
        \subcaption{A polar odd wheel $Q$ with a $3$-polar assignment $\Gamma$. Here the vertex $r$ is $\Gamma$-restricted on red, and $F(Q) \subseteq \Ev_Q(r)$.}
        \label{pow2fig}
    \end{subfigure}
    \caption{A (polar) odd wheel with a vertex $r$.}
\end{figure}

We now define, for each $v \in V(G) \setminus \{h\}$, a set of edges of $G$, denoted $\Ev_G(v)$, as follows:
if $v$ is dominating, then $\Ev_G(v)$ is the set of edges incident to $v$;
whereas if $v$ is not dominating, then, letting $h$ be the dominating vertex of $G$, we define $$\Ev_G(v) = \{vh\} \cup \{e \in E(G-h) : e \textrm{ has even displacement to $v$ in $G-h$}\}.$$ 
Notice that if $v$ is incident with an edge $e \in E(G)$, then $e \in \Ev_G(v)$.
See \cref{pow1fig} for an example.
For a polar graph $Q$, we write $\Ev_Q$ rather than $\Ev_{G(Q)}$.

\begin{lemma}
    \label{clsu2}
    Let $Q$ be a polar odd wheel with $|V(Q)| > 4$ and hub $h$, and let $r \in V(Q) \setminus \{h\}$.
    Suppose $(Q-r,h) \in \mathcal{N}_2$.
    Then $(Q-\{r,h\},V(Q-\{r,h\}))$ is a balanced Gallai tree where the blocks of $Q-\{r,h\}$ alternate between light and normal, with the leaf blocks being light, and $F(Q) \subseteq \Ev_Q(r)$.
\end{lemma}
\begin{proof}
    Since $(Q-r,h) \in \mathcal{N}_2$, \ref{N1} holds, that is, $(Q-\{r,h\},V(Q-\{r,h\}))$ is a balanced Gallai tree.
    Since $Q-\{r,h\}$ is an odd path of length at least three, $Q-\{r,h\}$ consists of two degree-$1$ vertices that are internal vertices, whereas the remaining vertices are cut vertices that each belong to two $K_2$-blocks.
    By \ref{G1} and \ref{G2}, each internal vertex is in a light block, and each cut vertex is in both a light and normal block.
    Thus the blocks of $Q-\{r,h\}$ alternate between light and normal.
    By \ref{G3}, polarised edges can only appear in normal blocks of a balanced Gallai tree, so $F(Q) \subseteq \Ev_Q(r)$.
\end{proof}

\begin{lemma}
    \label{odd wheels in G2}
    Let $Q$ be a polar odd wheel with $|V(Q)| > 4$ and hub $h$, and let $r \in V(Q) \setminus \{h\}$.
    Let $\Gamma = (L,R)$ be a $3$-polar assignment for $Q$ with $c \in L(r)$, let $(Q-r,\Gamma_c)$ be the $(r,c)$-colour deletion from $(Q,\Gamma)$, and suppose that $Q-r$ is not $\Gamma_c$-colourable.
    Then $(Q-r,h) \in \mathcal{N}_2$ and $\Gamma_c$ is $h$-bad.
    In particular, letting $\Gamma_c=(L_c,R_c)$, there exist distinct $c_1,c_2 \in \mathbb{N}$ such that $\{c_1,c_2\} \subseteq L_c(u)$ for all $u \in V(Q-r)$, where $|L_c(u)| = 2$ if $u \in N_Q(r)$, and $|L_c(u)| = 3$ otherwise.
\end{lemma}
\begin{proof}
    Observe that $d_Q(h) \ge 5$, and $|L(v)| = d_Q(v)=3$ for all $v \in V(Q) \setminus \{h\}$.
    Let $\Gamma_c = (L_c,R_c)$.
    For each vertex $v \in V(Q) \setminus \{r,h\}$, there is at most one edge between $v$ and $r$, so $|L_c(v)| \ge d_{Q-r}(v)$.

    Suppose $Q-r$ is not $\Gamma_c$-colourable.
    Since $Q-\{r,h\}$ is connected, \cref{dcpolarhdthm} implies that $(Q-r,h) \in \mathcal{N}_k$ for $k=|L_c(h)|$, and $\Gamma_c$ is $h$-bad.
    Note that there is $u \in N_Q(h) \cap N_Q(r)$ such that $d_{Q-\{r,h\}}(u)=1$, so $u$ is an internal vertex of $Q-\{r,h\}$ that belongs to a block~$B$.
    Since $u \in N_{Q-r}(h)$, the block $B$ is light, by \ref{N1} and \ref{G1}, and so, by \ref{N2}, $k=2$.
    Let $Q' = Q-\{r,h\}$.
    By \cref{clsu2}, $(Q',V(Q'))$ is a balanced Gallai tree where the blocks of $Q'$ alternate between light and normal, with the leaf blocks being light.
    Since $\Gamma_c$ is $h$-bad, it follows that there exist distinct $c_1,c_2\in \mathbb{N}$ such that $L_c(h)=\{c_1,c_2\}$, and $L_c(h) \subseteq L_c(u)$ for all $u \in V(Q')$, and $|L_c(u)|=2$ for $u \in N_Q(r) \setminus \{h\}$.
\end{proof}

\begin{lemma}
    \label{More than 4}
    \label{tidying up odd wheels}
    Let $Q$ be a polar odd wheel and let $r \in V(Q)$.
    All of the following hold:
    \begin{enumerate}
        \item The vertex $r$ is $3$-restricted if and only if $F(Q) \subseteq \Ev_Q(r)$.
        \item If $\Gamma$ is a $3$-polar assignment for $Q$ such that $r$ is $\Gamma$-restricted on more than one colour, then
            \begin{itemize}
                \item $F(Q)=\emptyset$,
                \item $M_{Q,\Gamma}(r) \subseteq L(u)$ for all $u \in V(Q-r)$, and
                \item $\Gamma|V(Q-r)$ is uniform.
            \end{itemize}
        \item The vertex $r$ is $3$-unrestricted in $Q'$ for any proper polar subgraph $Q'$ of $Q$.
    \end{enumerate}
\end{lemma}
\begin{proof}
    We first assume $r$ is dominating.
    Suppose that $F(Q) \subseteq \Ev_Q(r)$.
    Note that if $e \in F(Q)$, then $e \in \Ev_Q(r)$, and so $r$ is incident with $e$.
    We will construct a $3$-polar assignment $\Gamma=(L,R)$ for $Q$ such that $r$ is $\Gamma$-restricted.
    Let $L$ be a uniform $3$-list assignment with $c \in L(r)$, and define $R$ such that,
        for any $ur \in F(Q)$, we have $R(e,u,r)=(c,c)$.
    Let $(Q-r,\Gamma_c)$ be the $(r,c)$-colour deletion from $(Q,\Gamma)$.
    Then $Q-r$ is an odd cycle with no polarised edges and, for any vertex $u \in V(Q-r)$, we have $L_c(u)=L(u) \setminus \{c\}$.
    Therefore $\Gamma_c$ is a bad degree-polar assignment and $Q-r$ is not $\Gamma_c$-colourable, so $r$ is $3$-restricted.

    For the converse, assume that $F(Q) \not \subseteq \Ev_Q(r)$.
    Then there is some polarised edge $e=uw$, not incident with $r$.
    Let $\Gamma$ be a $3$-polar assignment for $Q$, and let $(Q-r,\Gamma_c)$ be the $(r,c)$-colour deletion from $(Q,\Gamma)$ for some $c \in L(r)$.
    Then $Q-r$ is an odd cycle with a polarised edge, and $\Gamma_c$ is a degree-polar assignment, so $Q-r$ is $\Gamma_c$-colourable by \cref{dcpolarthm}.
    Therefore, $Q$ is $3$-unrestricted, so (i) holds when $r$ is dominating.

    We continue to assume $r$ is dominating and next consider (ii).
    Suppose $\Gamma=(L,R)$ is a $3$-polar assignment for $Q$ such that $r$ is $\Gamma$-restricted on more than one colour; that is, $|M_{Q,\Gamma}(r)| \ge 2$.
    For $c \in M_{Q,\Gamma}(r)$, we again let $(Q-r,\Gamma_c)$ be the $(r,c)$-colour deletion from $(Q,\Gamma)$, with $\Gamma_c=(L_c,R_c)$.
    Then $Q-r$ is not $\Gamma_c$-colourable for such a $c$, where $Q-r$ is an odd cycle, and $\Gamma_c$ is a degree-polar assignment.
    By the foregoing, $F(Q) \subseteq \Ev_Q(r)$, so $F(Q-r)=\emptyset$.
    By \cref{cdlemma1}, if $Q$ has a polarised edge $ur$, then, as $|M_{Q,\Gamma}(r)| \ge 2$, there exists $c \in M_{Q,\Gamma}(r)$ such that $|L_c(u)| > 2$, so $\Gamma_c$ is not bad, implying $Q-r$ is $\Gamma_c$-colourable by \cref{dcpolarthm}, a contradiction.
    So $F(Q) = \emptyset$.
    Similarly, for each $c \in M_{Q,\Gamma}(r)$, we have $c \in L(u)$ for all $u \in V(Q-r) = N_Q(r)$.
    That is, $M_{Q,\Gamma}(r) \subseteq L(u)$ for all $u \in V(Q-r)$.
    Now, for each $c \in M_{Q,\Gamma}(r)$, we have that $\Gamma_c$ is uniform.
    It follows that $\Gamma|V(Q-r)$ is uniform, as required.
    So (ii) also holds when $r$ is dominating.

    We now assume that $r$ is not dominating.
    Then $Q$ has a unique dominating vertex $h \in V(Q) \setminus \{r\}$.
    Let $Q'=(Q-r)-h$ and observe that $Q'$ is an odd path of length at least one.

    Suppose $F(Q) \not \subseteq \Ev_Q(r)$.
    Then, by \cref{clsu2}, $(Q-r,h) \notin \mathcal{N}_2$.
    Now, for any $3$-polar assignment $\Gamma=(L,R)$ for $Q$ with $c \in L(r)$, we have, by \cref{odd wheels in G2}, that $Q-r$ is $\Gamma_c$-colourable for the $(r,c)$-colour deletion $(Q-r,\Gamma_c)$ from $(Q,\Gamma)$.
    Thus $r$ is $3$-unrestricted.

    Conversely, suppose $F(Q) \subseteq \Ev_Q(r)$.
    We have that $(Q',V(Q'))$ is a balanced Gallai tree, where the blocks of $Q'$ alternate between light and normal, with the leaf blocks being light, so \ref{N1} holds.
    Since each block of $Q'$ is a $K_2$-block, \ref{N2} is satisfied, and since $Q-r$ is simple and $F(Q) \subseteq \Ev_Q(r)$, \ref{N3} holds.
    So $(Q-r,h) \in \mathcal{N}_2$.
    By \cref{hbadexistencelemma}, there exists an $h$-bad polar assignment $\Gamma_2=(L_2,R_2)$ for $Q-r$.
    We extend $\Gamma_2$ to a $3$-polar assignment $\Gamma=(L,R)$ for $Q$ as follows.
    Let $L(r)=\{a,b,c\}$ where $c \notin L_2(u)$ for any $u \in N_Q(r)$, let $L(u)=L_2(u) \cup \{c\}$ for each $u \in N_Q(r)$, and let $L(w)=L_2(w)$ for all $w \in V(Q) \setminus (N_Q(r) \cup \{r\})$.
    Finally, we let $R(e,u,r)=(c,c)$ for any polarised edge $e=ur$ incident with $r$.
    It follows that the $(r,c)$-colour deletion from $(Q,\Gamma)$, namely $(Q-r,\Gamma_2)$, is not colourable, so $r$ is 
    $3$-restricted, as required.

    We continue to assume that $r$ is not dominating, and consider (ii).
    Let $c \in M_{Q,\Gamma}(r)$, so, by \cref{odd wheels in G2}, we have that 
    $L_c(h) \subseteq L_c(u)$ for all $u \in V(Q')$, and $|L_c(u)|=2$ for $u \in N_Q(r)$.
    This implies, by \cref{cdlemma1}, that $Q$ has no polarised edges incident with $r$.
    Thus $L_c(h)=L(h) \setminus \{c\}$.
    Since $|M_{Q,\Gamma}(r)| \ge 2$, it follows that $\Gamma|V(Q-r)$ is uniform, and $c \in L(u)$.
    So $M_{Q,\Gamma}(r) \subseteq L(u)$ for all $u \in V(Q-r)$.

    It remains to show that $F(Q)=\emptyset$.
    We have $F(Q) \subseteq \Ev_Q(r)$, and, by \cref{clsu2}, any polarised edge of $Q'$ belongs to a normal $K_2$-block of $Q'$.
    Let $e=uv$ be a polarised edge, belonging to a normal $K_2$-block $B'$ of $Q'$.  Then $u$ also belongs to a light block $B$.
    Consider distinct $a,b \in M_{Q,\Gamma}(r)$, and recall that $\Gamma_a=(L_a,R_a)$ and $\Gamma_b=(L_b,R_b)$ are the polar assignments obtained from $(r,a)$- and $(r,b)$-colour deletions, respectively.
    Since $\Gamma_a$ and $\Gamma_b$ are both $h$-bad, they can be obtained as the union of pairwise non-conflicting polar assignments as described in \cref{hbaddef}.
    Let $\Gamma_{B,a}$ and $\Gamma_{B',a}$ be such polar assignments for $B$ and $B'$ that appear in the union to obtain $\Gamma_a$, and similarly for $\Gamma_{B,b}$ and $\Gamma_{B',b}$ relative to $\Gamma_b$.
    Then $\Gamma_{B',a}$ and $\Gamma_{B',b}$ are bad degree-polar assignments for $B'$ and, letting $\Gamma_{B,a}=(L_{B,a},R_{B,a})$ and $\Gamma_{B,b}=(L_{B,b},R_{B,b})$, we have that $L_{B,a}$ and $L_{B,b}$ are uniform $2$-list assignments for $B$ such that $L_a(h) \subseteq L_{B,a}(u)$ and $L_b(h) \subseteq L_{B,b}(u)$.
    Since $L_{B,a}(u)=L(u) \setminus \{a\}$ and $L_{B,b}(u)=L(u) \setminus \{b\}$, where $u \in V(B) \cap V(B')$ and $\Gamma_{B,a}$ and $\Gamma_{B',a}$ (or $\Gamma_{B,b}$ and $\Gamma_{B',b}$) are non-conflicting,
    the polar assignments $\Gamma_{B',a}$ and $\Gamma_{B',b}$ are distinct.
    But since $e$ is polarised, $\Gamma_{B',a}$ and $\Gamma_{B',b}$ are polar-conforming, so they are equal, a contradiction.
    Thus there are no polarised edges in $\Ev_Q(r)$ except perhaps edges that are incident with $r$.
    The fact that no polarised edges are incident with $r$ follows from \cref{cdlemma1}.

    It remains to prove (iii) holds.  It suffices to prove 
    that $r$ is $3$-unrestricted in $Q'$ where
    $Q' = Q\ba e$ for some $e \in E(Q)$.
    So let $Q' = Q\ba e$ for some $e \in E(Q)$. 
    Towards a contradiction, let $\Gamma=(L,R)$ be a $3$-polar assignment for $Q'$ such that $r$ is $\Gamma$-restricted on $c$ in $Q'$ for some $c \in L(r)$.
    Let $(Q' - r,\Gamma_c)$ be the $(r,c)$-colour deletion from $(Q', \Gamma)$, so that $Q'-r$ is not $\Gamma_c$-colourable, with $\Gamma_c=(L_c,R_c)$.

    First, we assume that $r$ is dominating in $Q$.
    Then $\Gamma_c$ is a degree-polar assignment so, by \cref{dcpolarthm}, $\Gamma_c$ is bad.
    Suppose that $e$ is incident with $r$ in $Q$; let $e=ru$ for some vertex $u$.
    Then $|L_c(u)|=3$, but $d_{Q'-r}(u)=2$, and it follows that $\Gamma_c$ is not bad, a contradiction.
    So we may assume that $e$ is not incident with $r$.
    Then $Q'-r$ is a path.
    Let $v$ be an end of this path.
    Then $|L_c(v)| \ge 2$, since at most one colour was removed in the colour deletion, but $d_{Q'-r}(u)=1$, and it follows that $\Gamma_c$ is not bad, a contradiction.

    We now assume that $r$ is not dominating in $Q$.
    Let $h$ be a dominating vertex in $Q$, and note that $d_Q(h) \ge 5$.
First assume that $e=hr$.
In this case notice that $|L_c(h)| =3$ while $|L_c(u)| \geq d_{Q'-r}(u)$ for all $u \in V(Q'-r) \setminus \{h\}$.
We also have that $Q' -\{r,h\}$ is connected, so, by \cref{dcpolarhdthm}, $(Q' -r, h) \in \mathcal{N}_3$.
    By \ref{N1}, $(Q'-\{r,h\}, N_{Q'-r}(h))$ is a balanced Gallai tree.
    Let $u \in N_{Q'}(r) \cap N_{Q'}(h)$, so $u$ is an end in the path $Q'-\{r,h\}$.
    Then $u$ is in a light block of $Q'-\{r,h\}$ by \ref{G1}, however $u$ has degree one in this block, contradicting \ref{N2}.
We may now assume that $e \neq hr$.
Consider a connected component of $Q'-\{r,h\}$.  Let $u$ be a vertex of this component that is incident with $e$ in $Q$.  Then $d_{Q'}(u)=2$.
Let $\Gamma_h$ be a colour deletion of $h$ from $(Q'-r,\Gamma_c)$.
Then
    $\Gamma_h$ is a degree-polar assignment, but $|L_h(u)| > d_{Q'-\{r,h\}}(u)$.
    Therefore each component of $Q'-\{r,h\}$ is $\Gamma_h$-colourable, by \cref{dcpolarthm}, and the result follows.
\end{proof}

\subsection{Restricted vertices}

We next build on the results of the last section, by generalising them from polar odd wheels to any graph in $\mathcal{T}_k^-$.
More specifically, for a polar graph $Q$ with $G(Q) \in \mathcal{T}_k^-$, we characterise the $k$-restricted vertices; we show that if a vertex is $k$-restricted on more than one colour, then $Q$ is an odd wheel or complete graph with no polarised edges; and we show that any proper polar subgraph is $k$-unrestricted.

We require the notion of a central vertex.
For a polar odd wheel $Q$, recall the definition of $\Ev_Q$ given in \cref{powsec}.

\begin{definition}[central vertex]
    Let $Q$ be a polar graph with $G(Q) \in \mathcal{T}_k^-$, for $k \ge 3$.
    A vertex $v \in V(Q)$ is {\em central} if 
    \begin{itemize}
        \item $Q$ is a polar odd wheel and $F(Q) \subseteq \Ev_Q(v)$, or
        \item $Q$ is not a polar odd wheel, every parallel class has at most one edge that is not polarised, and all polarised edges are incident with $v$.
    \end{itemize}
\end{definition}

\begin{theorem}
\label{tkminusone}
Let $Q$ be a polar graph with $G(Q) \in \mathcal{T}_k^-$, for $k \geq 3$, and let $r \in V(Q)$. Both of the following hold:
    \begin{enumerate}
        \item The vertex $r$ is $k$-restricted if and only if $r$ is central.\label{tk a}
        \item If $\Gamma$ is a $k$-polar assignment for $Q$ such that $r$ is $\Gamma$-restricted on more than one colour, then\label{tk b}
            \begin{itemize}
                \item $F(Q)=\emptyset$,
                \item $G(Q)$ is an odd wheel if $k=3$, or a complete graph with $k+1$ vertices otherwise,
                \item $M_{Q,\Gamma}(r) \subseteq L(u)$ for all $u \in V(Q-r)$, and
                \item $\Gamma|V(Q-r)$ is uniform.
            \end{itemize}
        \item The vertex $r$ is $k$-unrestricted in $Q'$ for any proper polar subgraph $Q'$ of $Q$.\label{tk d}
    \end{enumerate}
\end{theorem}
\begin{proof}
    Let $h$ be the hub of $Q$ if $G(Q)$ is not $k$-symmetric; otherwise, each vertex of $Q$ is a hub, and we let $h = r$.
    Now, if $r \neq h$, then $G(Q)$ is not $k$-symmetric; whereas if $r=h$, then $G(Q)$ may or may not be $k$-symmetric.

    First we consider the case where $r=h$.
    Let $\Gamma=(L,R)$ be a $k$-polar assignment for $Q$, let $c$ be any colour in $L(r)$, and let $(Q-r,\Gamma_c)$ be the $(r,c)$-colour deletion from $(Q,\Gamma)$, where $\Gamma_c=(L_c,R_c)$.
    \begin{claim}
    \label{the real v=h claim}
    Suppose $r=h$ and $Q-r$ is not $\Gamma_c$-colourable.
    Then $Q-r$ is a polar Gallai tree, $\Gamma_c$ is a bad polar assignment for $Q-r$, and $|L_c(u)|= d_{Q-r}(u) =|L(u)|-e(u,r)$ for all $u \in V(Q-r)$.
    \end{claim}
    \begin{subproof}
        As $r=h$, we have $d_Q(u)=k=|L(u)|$ for all $u \in V(Q-r)$.
        Then 
        $|L_c(u)| \geq |L(u)|-e(u,r) = d_Q(u)-e(u,r) = d_{Q-r}(u)$, so
        $\Gamma_c$ is a degree-polar assignment for $Q-r$.
        Therefore, as $Q-r$ is not $\Gamma_c$-colourable, $Q-r$ is a polar Gallai tree and $\Gamma_c$ is bad, by \cref{dcpolarthm}.
        Since $\Gamma_c$ is bad, $|L_c(u)| = d_{Q-r}(u) = |L(u)|-e(u,r)$
        for all $u \in V(Q-r)$.
    \end{subproof}

    \begin{claim}
        \label{rha}
        When $r=h$, \ref{tk a} holds.
    \end{claim}
    \begin{subproof}
        If $Q$ is a polar odd wheel, then \ref{tk a} holds by \cref{More than 4}(i).  So we may assume that $G(Q)$ is not an odd wheel.

        Let $T = Q-r$.
    Assume that $r$ is central.
    Since $G(Q)$ is not an odd wheel, all polarised edges of $Q$ are incident with $r$ by the definition of a central vertex.
    Therefore $Q-r$ has no polarised edges and, since $r=h$, it is a polar Gallai tree.
    Thus there exists a bad degree-polar assignment $\Gamma_T=(L_T,R_T)$ for $T$.

    We extend $\Gamma_T$ to a $k$-polar assignment $\Gamma$ for $Q$ as follows.
    For any $u \in V(T)$, let $\{e_{u,1},e_{u,2},\dotsc,e_{u,k-d_T(u)}\}$ be the set of edges between $u$ and $r$ in $Q$, where $e_{u,1}$ is the unique non-polarised edge if such an edge exists.
    Let $C=\{c_1,c_2,\dotsc,c_k\}$ be a set of $k$ colours such that $C \cap L_T(u)=\emptyset$ for all $u \in V(T)$.
    For $i \le k$, we write $C_i$ to denote the subset $\{c_1,c_2,\dotsc,c_i\}$ of $C$.
    Let $\Gamma=(L,R)$ be the $k$-polar assignment for $Q$ where
    \[
    L(v)=\begin{cases}
        L_T(v) \cup C_{k-d_T(v)} & \mbox{if $v \in V(T)$}\\
        C_k & \mbox{if $v=r$,}
    \end{cases}
    \]
    and $R(e_{u,i},u,r)=(c_i,c_1)$ for each $u \in V(T)$ and each polarised edge $e_{u,i}$.
    It is easy to see that $\Gamma$ is a $k$-polar assignment and, as all polarised edges are incident with $r$, the polarisation $R$ is well defined.
    Notice that if $e_{u,1}$ is polarised, then $R(e_{u,1},u,r)=(c_1,c_1)$, and so, whether $e_{u,1}$ is polarised or not, the $(r,c_1)$-colour deletion from $(Q,\Gamma)$ is $(T,\Gamma_T)$.
    Since $T$ is not $\Gamma_T$-colourable,
    $r$ is $\Gamma$-restricted on $c_1$.
    This shows that if $r$ is central, then it is $k$-restricted.
    
    For the other direction, assume that $r$ is not central.
    If there are any polarised edges not incident with $r$, then, as $T$ has no $K_2$-blocks, $T$ is not a polar Gallai tree and so, by \cref{the real v=h claim}, $r$ is not $k$-restricted.
    Recall that $G(Q)$ is not an odd wheel.
    As $r$ is not central, we may also assume, by the foregoing, that there is some $u \in V(T)$ such that there are distinct non-polarised edges, $e_0$ and $e_1$, between $u$ and $r$.  But then $|L_c(u)| > d_T(u)$, as both $e_0$ and $e_1$ remove $c$ from $L(u)$.
    Hence, by \cref{the real v=h claim}, $r$ is not $k$-restricted.
    This shows that if $r$ is not central, then $r$ is not $k$-restricted, as required.
    \end{subproof}

    \begin{claim}
        \label{rhb}
        When $r=h$, \ref{tk b} holds.
    \end{claim}
    \begin{subproof}
        Let $C = M_{Q,\Gamma}(r)$, so $r$ is $\Gamma$-restricted on each colour in $C$, with $|C| \ge 2$.
        Let $u$ be adjacent to $r$ in $Q$.
        By \cref{the real v=h claim}, for each $c \in C$, the $(r,c)$-colour deletion removes $e(u,r)$ colours from $L(u)$.
        Thus, by \cref{cdlemma2}, $e(u,r)=1$.
        Moreover, if there is a polarised edge $e=ur$ with $R(e,u,r)=(c_u,c_r)$, then only $c_r$ removes a colour from $u$.
        Therefore, there are no polarised edges incident with $r$.
        Since $r$ is central, by \ref{tk a}, it now follows that $F(Q)=\emptyset$ and, by \cref{no parallel edges lemma}, either $G(Q)$ is an odd wheel if $k=3$, or a complete graph with $k+1$ vertices otherwise.
        In either case, $G(Q-r)$ is a Gallai tree consisting of a single block, so, by \cref{the real v=h claim}, $\Gamma_c$ is uniform for all $c \in C$.
        It now follows that $\Gamma|V(Q-r)$ is also uniform.
        We then note that, for each $c \in C$, we have $c \in L(u)$ for all $u \in N_Q(r)=V(Q-r)$, since an $(r,c)$-colour deletion removes a colour from $u$. So $C \subseteq L(u)$ for all $u \in V(Q-r)$, as required.
    \end{subproof}

    \begin{claim}
        \label{rhc}
        When $r=h$, \ref{tk d} holds.
    \end{claim}
    \begin{subproof}
        If $Q$ is a polar odd wheel, then \ref{tk d} holds by \cref{More than 4}(iii).  So we may assume that $G(Q)$ is not an odd wheel.
        Moreover, it suffices to prove \ref{tk d} in the case that $Q' = Q \ba e$ for some $e \in E(Q)$.
        We have two cases: either $e$ is incident with $r$, or $e$ is not incident with $r$.
        First assume $e$ is incident with $r$.  Then $Q'-r = Q-r$.
        We let $(Q'-r,\Gamma_c')$ be the $(r,c)$-colour deletion from $(Q',\Gamma)$, where $\Gamma_c'=(L_c',R_c')$.
        Towards a contradiction, assume that $Q'-r$ is not $\Gamma_c'$-colourable.
        Let $e=ur$ for some $u \in V(Q) \setminus \{r\}$.
        Then $\Gamma_c'$ is a degree-polar assignment for $Q'-r$.
        By \cref{dcpolarthm}, $\Gamma_c'$ is bad.
        Each edge of $Q'$ that is incident to $r$, say $rv$, removes at most one colour from $L(u)$.
        In particular, since $e$ is in $Q$ but not $Q'$, we have $|L_c'(u)| > k-e_Q(u,r) = d_{Q'-r}(u)$, implying $\Gamma_c'$ is not bad.
        Hence we may assume that $r$ is not incident with $e$.
        Let $Q''$ be the polar graph obtained from $Q$ by adding $e$ to $F(Q)$.
        We let $\Gamma''=(L,R'')$ where $R''|Q = R$ and $R''(e,u,v)=(a,a')$ such that $a\notin L(u)$ and $a' \notin L(v)$.
        It is clear that $\phi$ is a $\Gamma$-colouring of $Q'$ if and only if $\phi$ is a $\Gamma''$-colouring of $Q''$.
        Since $e$ is a polarised edge not incident with $r$, it follows that $r$ is $k$-unrestricted in $Q'$ by \cref{rha}.
        This proves the claim.
    \end{subproof}

    Now, by \cref{rha,rhb,rhc}, the theorem holds when $r=h$, so assume that $r \neq h$.
    Recall that, in this case, $G(Q)$ is not $k$-symmetric.
    If $Q$ is a polar odd wheel, then $k=3$, and 
    the theorem holds by \cref{More than 4}.
    So we may now assume that $G(Q)$ is not an odd wheel.

    \begin{claim}
        \label{existpara}
        There exists a vertex $u \in V(Q) \setminus \{h,r\}$ such that $e_Q(h,u) \ge 2$.
    \end{claim}
    \begin{subproof}
        Consider $T=Q-h$, so $G(T)$ is a Gallai tree.
        We will show that $T$ has at least two internal vertices.
        If $T$ consists of a single block, then that block contains at least two vertices, as otherwise $G(Q)$ is $k$-symmetric; so $T$ has at least two internal vertices in this case.
        Otherwise, when $T$ has more than one block, it has at least two leaf blocks, each of which contains an internal vertex (which is not in any other block).
        Thus, in either case, $T$ has at least two internal vertices, $u_1$ and $u_2$ say.
        Therefore, by \cref{no parallel edges lemma}(ii), $e(h,u_1) \ge 2$ and $e(h,u_2) \ge 2$.
        Now, either $r \neq u_1$ or $r \neq u_2$, so we let $u \in \{u_1,u_2\}$ such that $r \neq u$.
    \end{subproof}
    Let $u$ be as given by \cref{existpara}.
    If every parallel class in $Q$ has at most one non-polarised edge, then there is at least one polarised edge between $h$ and $u$, where $r \notin \{h,u\}$.
    This shows that $r$ is not central.

    We now show that $r$ is $k$-unrestricted.
    Let $\Gamma=(L,R)$ be a $k$-polar assignment for $Q$.  For any $c \in L(r)$, let $(Q-r,\Gamma_c)$ be the $(r,c)$-colour deletion from $(Q,\Gamma)$, where $\Gamma_c=(L_c,R_c)$.
    Again, let $T=Q-h$, so $G(T)$ is a Gallai tree with no $K_2$-blocks.
    We have two cases: either $r$ is a cut vertex of $T$, or $r$ is an internal vertex of $T$.
    First assume that $r$ is an internal vertex of $T$.
    Then $T-r=(Q-r)-h$ is connected.
    Since $T$ has no $K_2$-blocks, every vertex of $T$ has degree at least two.
    In particular, $r$ has degree at least two in $T$, so $e_Q(r,h) \le k-2$.
    Thus, for $c \in L(r)$, we have $|L_c(h)| \ge 2$.
    By \cref{dcpolarhdthm}, if $Q-r$ is not $\Gamma_c$-colourable, then $(Q-r,h) \in \mathcal{N}_{|L_c(h)|}$.
    But, by \cref{existpara}, there exists a vertex $u$ in $Q-r$ such that $e_{Q-r}(h,u) \ge 2$, so $Q-r$ is not simple, implying $(Q-r,h) \notin \mathcal{N}_{|L_c(h)|}$ by \ref{N3}.
    Therefore $r$ is not $k$-restricted.

    Next assume that $r$ is a cut vertex of $T$.
    In this case, $T-r$ is not connected, so we cannot apply \cref{dcpolarhdthm} directly; instead, we find a particular polar subgraph on which to apply \cref{dcpolarhdthm}.
    For any connected component $C$ of $T-r$, we consider the polar subgraph $Q[V(C) \cup \{r\}]$.
    Let $\mathcal{S}$ be the set of all such polar subgraphs; that is, let $\mathcal{S} = \{Q[V(C) \cup \{r\}] : C \textrm{ is a component of $T-r$}\}$.
    Note that for any polar graph in $\mathcal{S}$, the underlying graph is a Gallai tree, since $G(T)$ is a Gallai tree.
    Any two distinct polar graphs in $\mathcal{S}$ are edge-disjoint, and have only the vertex $r$ in common.
    Moreover, the union of all members of $\mathcal{S}$ is $T$.
    Since $T$ has no $K_2$-blocks, each member of $\mathcal{S}$ has no $K_2$-blocks.
    Thus $d_S(r) \ge 2$ for each $S \in \mathcal{S}$.
    It follows that $|\mathcal{S}| \le d_T(r)/2$.
    
    Suppose that $r$ is $k$-restricted on $c$.
    Then for every $d \in L_c(h)$, the $(h,d)$-colour deletion from $(Q-r,\Gamma_c)$ is not colourable.
    For $S \in \mathcal{S}$, let $S^+=Q[(V(S)-\{r\}) \cup \{h\}]$.
    Now $Q-r$ is the union of all members of $\{S^+ : S \in \mathcal{S}\}$, where any two members of this set have only the vertex $h$ in common.
    For each $d \in L_c(h)$ there exists some $S_d \in \mathcal{S}$ such that $h$ is $\Gamma_c$-restricted on $d$ in $(S_d)^+$.
    As $|L_c(h)| \ge k-e(r,h)$, and $e(r,h)=k-d_T(r)$, we have $|L_c(h)| \ge d_T(r)$.
    Since $|\mathcal{S}| \le d_T(r)/2$, there exists some $S' \in \mathcal{S}$ for which $h$ is $\Gamma_c$-restricted on at least two colours in $(S')^+$; let $c_1$ and $c_2$ be two such colours.
    We obtain a polar assignment $\Gamma'=(L',R_c)$ for $(S')^+$ by letting $L'(h)=\{c_1,c_2\}$ and $L'(v)=L_c(v)$ for $v \neq h$.
    Now $(S')^+-h$ is connected, $|L'(h)| =2$, and $(S')^+$ is not $\Gamma'$-colourable.
    Therefore, by \cref{dcpolarhdthm}, $((S')^+,h) \in \mathcal{N}_2$.
    However, there is a leaf block of $T$ that is contained in $(S')^+-h$, so, by \cref{no parallel edges lemma}, there is $u \in V((S')^+-h) \setminus \{r\}$ such that $e(u,h) \ge 2$. Therefore $((S')^+,h) \notin \mathcal{N}_2$.
    From this contradiction, we deduce that $r$ is not $k$-restricted.

    We have now shown, in the case that $r \neq h$, that $r$ is not central and $r$ is not $k$-restricted, so \ref{tk a} and \ref{tk b} hold.
    Moreover, since $r$ is not $k$-restricted in $Q$, it follows that $r$ is not $k$-restricted in $Q\ba e$ for any $e \in E(Q)$, and $\ref{tk d}$ follows, in the case that $r \neq h$.
    This completes the proof.
\end{proof}

\begin{corollary}
    \label{tk c}
    Let $Q$ be a polar graph with $G(Q) \in \mathcal{T}_k^-$, for $k \geq 3$, and let $\Gamma$ be a $k$-polar assignment for $Q$. Then
    $Q$ is not $\Gamma$-colourable if and only if 
    \begin{enumerate}[label=\rm(\Roman*)]
        \item $F(Q)=\emptyset$,
        \item $G(Q)$ is an odd wheel if $k=3$, or a complete graph with $k+1$ vertices otherwise, and
        \item $\Gamma$ is uniform.
    \end{enumerate}
\end{corollary}
\begin{proof}
    Let $h$ be the hub of $Q$, where $h$ is arbitrary if $G(Q)$ is $k$-symmetric.
    Suppose that $Q$ is not $\Gamma$-colourable, and let $\Gamma = (L,R)$.
    Then $h$ is $\Gamma$-restricted for each $c \in L(h)$, so $M_{Q,\Gamma}(h) = L(h)$.
    Now, by \cref{tkminusone}\ref{tk b}, we have that (I) and (II) hold, $L(h) = M_{Q,\Gamma}(h) \subseteq L(u)$ for $u \in V(Q-h)$, and $\Gamma|V(Q-h)$ is uniform.
    For any $u \in V(Q-h)$, we have $|L(h)| = |L(u)|$, so $L(h)=L(u)$, and therefore $\Gamma$ is uniform.
    The other direction follows from \cref{stthm}.
\end{proof}

\subsection{\texorpdfstring{$k$}{k}-joinability}

In this subsection, we introduce the notion of a $k$-joinable graph, and show that any polar graph whose underlying graph is in $\mathcal{T}_k^-$ is $k$-joinable.

We first require a notion of when a colour is ``forced'' at a vertex.

\begin{definition}[forces]
    Let $Q$ be a polar graph with distinct vertices $u,v \in V(Q)$, and let $\Gamma = (L,R)$ be a polar assignment of $Q$ with $c_v \in L(v)$ and $c_u \in L(u)$.
    We say that \emph{$(v,c_v)$ forces $(u,c_u)$} in $(Q,\Gamma)$ if, for any $\Gamma$-colouring $\phi$ of $Q$ with $\phi(v)=c_v$, we have $\phi(u)=c_u$.
    We also say that $(v,c_v)$ \emph{forces} $u$ in $(Q,\Gamma)$ if there exists $c \in L(u)$ such that $(v,c_v)$ forces $(u,c)$ in $(Q,\Gamma)$.
\end{definition}


\begin{lemma}
\label{forcing-lemma}
Let $Q$ be a polar graph with distinct $u,v \in V(Q)$.
Let $\Gamma=(L,R)$ be an $f$-polar assignment for $Q$ such that $(v,c_v)$ forces $u$ in $(Q,\Gamma)$ for some $c_v \in L(v)$.
Then, for any $c \in \mathbb{N} \setminus \{c_v\}$, there exists an $f$-polar assignment $\Gamma'$ such that $(v,c_v)$ forces $(u,c)$ in $(Q,\Gamma')$.
\end{lemma}
\begin{proof}
    By assumption, there exists $c_u \in \mathbb{N}$ such that $(v,c_v)$ forces $(u,c_u)$ in $(Q,\Gamma)$.
    We may assume that $c_u \neq c$, otherwise $\Gamma'=\Gamma$ is our desired $f$-polar assignment.
    Recall that we say $c' \in \mathbb{N}$ is {\em new for $\Gamma=(L,R)$} if $c' \notin L(v)$ for all $v \in V(Q)$, and $c' \notin R(e,u,v)$ for all $e=uv \in F(Q)$.
    We may also assume that $c$ is new for $\Gamma$, as if $c$ is not new, then we perform the colour swap of $c$ and $c'$ for some $c'$ that is new for $\Gamma$
    (where this operation does not affect $c_v$, since $c \neq c_v$).

    Now let
    \[
    L_u(x) = \begin{cases}
        L(x) \setminus \{c_u\} \cup \{c\} & \mbox{if $x=u$}\\
        L(x)& \mbox{otherwise,}
    \end{cases}
    \]
    and let $\Gamma_u = (L_u,R)$.
    Then it is not difficult to see that $\Gamma_u$ is an $f$-polar assignment for which $(v,c_v)$ forces $(u,c)$ in $(Q,\Gamma_u)$, as required.
\end{proof}

We would like to be able to talk about the colours that one vertex, $v$ say, forces on another vertex, $u$ say.
However, \cref{forcing-lemma} tells us that when we only know of a single colour $c_v \in L(v)$ such that $(v,c_v)$ forces $u$, we cannot say anything meaningful about the colours that are forced on $u$.
This motivates our next two definitions.

\begin{definition}[$k$-forces]
    Let $Q$ be a polar graph with polar assignment $\Gamma=(L,R)$, and distinct $u,v \in V(Q)$.
    Let $k$ be a positive integer.
    We say that $v$ {\em $k$-forces} $u$ in $(Q,\Gamma)$ if there exists $C \subseteq L(v)$ with $|C| \ge k$ such that $(v,c)$ forces $u$ in $(Q,\Gamma)$ for each $c \in C$.
\end{definition}

\begin{definition}[$k$-mimic]
    Let $Q$ be a polar graph with distinct $u,v\in V(Q)$.
    Let $\mathcal{L}$ be the set of all $k$-polar assignments $\Gamma$ for $Q$ such that $v$ $2$-forces $u$ in $(Q,\Gamma)$.
    We say that $u$ is a {\em $k$-mimic} of $v$ in $Q$ if, for all $\Gamma=(L,R) \in \mathcal{L}$ and all $c \in L(v)$ such that $(v,c)$ forces $u$ in $(Q,\Gamma)$, we have that $(v,c)$ forces $(u,c)$ in $(Q,\Gamma)$.
\end{definition}

Intuitively, if $u$ is a $k$-mimic of $v$, then whenever $v$ $2$-forces $u$, we have that $v$ forces $u$ to be the same colour as $v$.
Note that if
$v$ does not $2$-force $u$ for any $k$-polar assignment,
then $u$ is vacuously a $k$-mimic of $v$.

\begin{definition}[$k$-joinable]
\label{joinability def}
We say that a polar graph $Q$ is {\em $k$-joinable} if, for all $e=uv \in E(Q)$, we have that $Q\ba e$ is $k$-unrestricted and $u$ is a $k$-mimic of $v$ in $Q\ba e$.
If $G$ is a graph such that any polar graph $Q$ with $G(Q)=G$ is $k$-joinable, then we say that $G$ is {\em $k$-joinable}.
\end{definition}

As we will see in \cref{joinablesimulation}, for $k$-joinable graphs we can simulate one side of a Haj\'{o}s join using polarised edges.
We first work towards showing that if $G$ is $k$-polar-unrestricted, then $G$ is $k$-joinable.

\begin{lemma}
\label{unrestricted vacuous}
Let $G$ be a $k$-polar-unrestricted graph, let $Q$ be a polar graph with $G(Q)=G\ba e$ for some $e=uv \in E(G)$, and let $\Gamma=(L,R)$ be a $k$-polar assignment for $Q$.
Then there is no $c \in L(v)$ such that $(v,c)$ forces $u$ in $(Q,\Gamma)$.
\end{lemma}
\begin{proof}
    Towards a contradiction, assume there exists $c \in L(v)$ such that $(v,c)$ forces $(u,d)$ in $(Q,\Gamma)$, for some $d \in \mathbb{N}$.
    Consider the polar graph $Q' = (G,F(Q) \cup \{e\})$.
    Let $\Gamma'=(L',R')$ be a $k$-polar assignment for $Q'$ where $\Gamma'|Q=\Gamma$ and $R'(e,u,v)=(d,c)$.
    Note that since $(v,c)$ forces $(u,d)$ in $(Q,\Gamma)$, no $\Gamma'$-colouring $\phi$ of $Q'$ has $\phi(v)=c$, as in any such colouring we have $\phi(u)=d$, violating \ref{Gamma 2} on $e$.
    Hence $v$ is $\Gamma'$-restricted on $c$ in $Q'$, contradicting the fact that $G$ is $k$-polar-unrestricted.
\end{proof}

\begin{corollary}
\label{unrestricted joinable corollary}
If $G$ is a $k$-polar-unrestricted graph, then $G$ is $k$-joinable.
\end{corollary}
\begin{proof}
    Let $Q$ be a polar graph with underlying graph $G$,
    and let $e=uv \in E(Q)$.
    Since $Q$ is $k$-unrestricted, $Q \ba e$ is also $k$-unrestricted.
    Moreover, by \cref{unrestricted vacuous}, $u$ is vacuously a $k$-mimic of $v$.
    This shows that any polar graph $Q$ with underlying graph $G$ is $k$-joinable, as required.
\end{proof}

Note that it follows from the definition of $k$-joinable and \cref{unrestricted joinable corollary} that if $G$ is $k$-joinable, then $G \ba e$ is $k$-joinable.
So $k$-joinable graphs are closed under subgraphs.

We now prove the main result of this subsection.

\begin{theorem}
    \label{tkminustwo}
    Let $G$ be a graph in $\mathcal{T}_k^-$, for $k \geq 3$.  Then $G$ is $k$-joinable.
\end{theorem}
\begin{proof}
    Towards a contradiction, assume there exists a polar graph $Q$ with $G(Q)=G$ that is not $k$-joinable.
    By \cref{tkminusone}\ref{tk d}, $Q\ba z$ is $k$-unrestricted for any $z \in E(Q)$.
    Thus there exists an edge $e=ur$ such that $u$ is not a $k$-mimic of $r$ in $Q \ba e$.
    Therefore, there exists a $k$-polar assignment $\Gamma=(L,R)$ for $Q \ba e$ with distinct $c_0,c_1 \in L(r)$ such that $(r,c_0)$ forces $(u,d_0)$ and $(r,c_1)$ forces $(u,d_1)$ in $(Q\ba e,\Gamma)$, for $d_0,d_1 \in L(u)$, with $c_0 \neq d_0$.
    Let $(Q-r,\Gamma_i)$ be the $(r,c_i)$-colour deletion from $(Q\ba e,\Gamma)$ for $i \in \{0,1\}$, and let $\Gamma_i=(L_i,R_i)$.
    Then, let $\Gamma_{i}'=(L_i',R_i)$ be the polar assignment for $Q-r$ where $L_i'$ is obtained from $L_i$ by removing $d_i$ from $L_i(u)$; that is, $L_i'(u) = L_i(u) \setminus \{d_i\}$ and $L_i'(v) = L_i(v)$ for each $v \neq u$.
    Then $Q-r$ is not $\Gamma_{i}'$-colourable for $i \in \{0,1\}$, for otherwise there is a $\Gamma_i$-colouring of $Q-r$ where $u$ is not coloured $d_i$, contradicting that $(r,c_i)$ forces $(u,d_i)$.

    We can also obtain $(Q-r,\Gamma_i')$ as a colour deletion of a polar graph $Q'$ whose underlying graph is $G(Q)$.
    To see this, let $Q'$ be the polar graph with underlying graph $G(Q)$ and polarised edges $F(Q) \cup \{e\}$; that is, $Q$ is obtained from $Q$ by making $e$ a polarised edge (where if $e \in F(Q)$, then $Q'=Q$).
    Let $\Gamma'=(L',R')$ be the polar assignment for $Q'$ such that $R'(e,u,r)=(d_0,c_0)$ and $R'|Q = R$.
    Then $(Q-r,\Gamma_{0}')$ is the $(r,c_0)$-colour deletion from $(Q',\Gamma')$, so $r$ is $k$-restricted in $Q'$.
    Therefore, by \cref{tkminusone}\ref{tk a}, $r$ is central in $Q'$.

    Let $h$ be a hub of $Q$.
    If $G(Q)$ is not an odd wheel and is not $k$-symmetric, then $h$ is the only central vertex in $Q'$.
    So first we assume that $G(Q)$ is not an odd wheel and is not $k$-symmetric, and $r=h$.
    Now, for $i \in \{0,1\}$, we have that $\Gamma_i'$ is a degree-polar assignment for $Q-r$, and $Q-r$ is not $\Gamma_i'$-colourable, so, by \cref{dcpolarthm}, $\Gamma_i'$ is bad.
    Moreover, 
    $|L_{i}'(v)|=d_Q(v)-e_Q(v,r)$ for all $v \in V(Q-r)$.
    There are at least two internal vertices in $Q-r$, and so there is some internal vertex $v$ of $Q-r$ that is distinct from $u$.
    Then $e_Q(v,r) \geq 2$ by \cref{no parallel edges lemma}(ii).
    By \cref{cdlemma2}, either $|L_{0}'(v)|>d_Q(v)-e(v,r)$ or $|L_{0}'(v)|>d_Q(v)-e(v,r)$, so $Q-r$ is either $\Gamma_{0}'$- or $\Gamma_{0}'$-colourable, a contradiction. 

We may now assume that $Q$ is $k$-symmetric or an odd wheel.
If $Q \cong I_k$, then, by \cref{cdlemma2}, for some $i \in \{0,1\}$, we have $|L_i'(u)| > 0$, so $Q-r$ is $\Gamma_i'$-colourable, a contradiction.
Therefore we may assume that either $Q \cong K_{k+1}$ or $Q$ is an odd wheel.
If $r$ is a hub, then $\Gamma_i'$ is a degree-polar assignment for $Q-r$, so, by \cref{dcpolarthm}, $\Gamma_i'$ is bad.
As $Q-r$ consists of a single block, which is not a $K_2$-block, $\Gamma_i'$ is uniform.
By \cref{cdlemma1}, 
$r$ is not incident to a polarised edge.
Now, since $\Gamma_i'$ is uniform for $i \in \{0,1\}$, it follows that $c_0 = d_0$, a contradiction.

So we may assume that $Q$ is an odd wheel and $r$ is not the hub.
If $u$ is the hub, then let $w \in N_Q(r) \cap N_Q(u)$; otherwise, let $w$ be the hub.
In either case, $L_{i}'(w)=L(w)\setminus \{c_i\}$ and, by \cref{odd wheels in G2}, 
$L_{i}'(w) =L_{i}'(u)$. It then follows that $c_i=d_i$, a contradiction.
\end{proof}

\section{Haj\'os joins and the class \texorpdfstring{$\mathcal{T}_k$}{Tk}}

Let $\mathcal{T}_k$ be the smallest class of graphs that contains $\mathcal{T}_k^-$ and is closed under taking Haj\'{o}s joins, for $k \ge 3$.
    Note that since $\mathcal{T}_k^-$ contains odd wheels in the case when $k=3$, and complete graphs on $k+1$ vertices in the case when $k > 3$, we have that $\mathcal{H}_k \subseteq \mathcal{T}_k$.

    In this section, we begin with some structural properties of Haj\'os joins, and then consider some properties relating to polar list colouring.
    In particular, \cref{hajosextend} describes how polar list colourings can be used to simulate the colourability of the two graphs involved in a Haj\'os join.
    We then use these results to prove properties of graphs in $\mathcal{T}_k$.
    The main result in the section, \cref{tkes}, shows that the polar graphs with underlying graphs in $\mathcal{T}_k$ that are not $k$-choosable are those in $\mathcal{H}_k$ having no polarised edges.

\subsection{Haj\'{o}s joins and structural properties}

The first two lemmas are assumedly well known, but we were unable to find proofs in the literature; we omit a straightforward proof of the first, but provide a proof of the second for completeness.

\begin{lemma}
    \label{hajosconn0}
    Let $G$ be a Haj\'{o}s join of graphs $G_0$ and $G_1$.
    The graph $G$ is $2$-connected if and only if both $G_0$ and $G_1$ are $2$-connected.
\end{lemma}

\begin{lemma}
    \label{hajosconn1}
    Let $G$ be a Haj\'{o}s join of graphs $G_0$ and $G_1$, and let $k \ge 2$.
    The graph $G$ is $k$-edge-connected if and only if both $G_0$ and $G_1$ are $k$-edge-connected.
\end{lemma}
\begin{proof}
    Let $G = (G_0,v,u_0) \hajos (G_1,v,u_1)$ such that, for $i \in \{0,1\}$, we have $v,u_i \in V(G_i)$.
    Let $e_i=u_iv$, for $i \in \{0,1\}$, and $e=u_0u_1$. 

    Suppose that $G_i$ is $k$-edge-connected for $i \in \{0,1\}$.
    As $k \ge 2$, there exist two edge-disjoint $(u_i,v)$-paths in $G_i$, at most one of which contains $e_i$, so $G_i \ba e_i$ has at least one $(u_i,v)$-path for $i \in \{0,1\}$.
    Towards a contradiction, assume $G$ is not $k$-edge-connected.
    Then there is an edge cut $S$ of $G$ with $|S|<k$.
    Since $G\ba S$ contains at least two components, there is a vertex $x$ that is in a different component than $v$ in $G \ba S$.
    Therefore $S$ separates $x$ and $v$ in $G$, so there are strictly fewer than $k$ edge-disjoint $(x,v)$-paths in $G$.
    Without loss of generality, assume that $x \in V(G_0)$.
    Since $G_0$ is $k$-edge-connected, there are $k$ edge-disjoint $(x,v)$-paths in $G_0$, at most one of which contains $e_0$.
    If no such path contains $e_0$, then there are $k$ edge-disjoint $(x,v)$-paths in $G$, a contradiction.
    So there are $k-1$ $(x,v)$-paths in $G_0 \ba e_0$ and an $(x,v)$-path $P_0$ that contains $e_0$, where these $k$ paths are pairwise edge-disjoint.
    By concatenating the $(x,u_0)$-subpath of $P_0$, the $(u_0,u_1)$-path consisting of the single edge $e$, and a $(u_1,v)$-path in $G_1 \ba e_1$, we find an $(x,v)$-path in $G$ that is edge-disjoint from the other $k-1$ $(x,v)$-paths.
    So again there are $k$ edge-disjoint $(x,v)$-paths in $G$, a contradiction.
    We deduce that $G$ is $k$-edge-connected.

    Now assume that $G$ is $k$-edge-connected but, towards a contradiction, $G_0$ is not $k$-edge-connected.
    Then, for some $x,y \in V(G_0)$, there are fewer than $k$ edge-disjoint $(x,y)$-paths in $G_0$.
    Since $G$ is $k$-edge-connected, there are $k$ edge-disjoint $(x,y)$-paths in $G$, at most one of which contains $e$.
    Any $(x,y)$-path that does not contain $e$ is contained in $G_0$, so we may assume that one of the paths, $P_0$ say, contains $e$.
    Then $P_0$ contains both $e$ and $v_0$; without loss of generality, suppose $e$ appears before $v_0$.
    By concatenating the $(x,u_0)$-subpath of $P_0$, the $(u_0,v)$-path consisting of the edge $e_0$, and the $(v,y)$-subpath of $P_0$, we obtain an $(x,y)$-path in $G_0$ that is edge-disjoint from the other $k-1$ $(x,y)$-paths in $G_0$, a contradiction.
    We deduce that $G_0$ and, by symmetry, $G_1$ are $k$-edge-connected.
\end{proof}

\begin{lemma}
    \label{hajosconn2}
    Let $G$ be the Haj\'os join $(G_0,v,u_0) \hajos (G_1,v,u_1)$ of graphs $G_0$ and $G_1$ where, for $i \in \{0,1\}$, the graph $G_i$ has maximum local edge-connectivity~$k_i$, and $u_iv \in E(G_i)$ is not a bridge.
    Then the maximum local edge-connectivity of $G$ is $\max\{k_0,k_1\}$.
\end{lemma}
\begin{proof}
    Let $e=u_0u_1$, and let $e_i=u_iv$ for $i \in \{0,1\}$. 
    For brevity, let $\lambda(G)$ denote the maximum local edge-connectivity of $G$, and, for $x,y \in V(G)$, let $\lambda_G(x,y)$ denote the maximum number of edge-disjoint $(x,y)$-paths in $G$.

    First, consider $v_i \in V(G_i) \setminus \{v\}$ for $i \in \{0,1\}$.
    Clearly every $(v_0,v_1)$-path contains either $e$ or $v$.
    There are at most $k_i$ edge-disjoint $(v_i,v)$-paths in $G_i \ba e_i$, since $\lambda(G_i)=k_i$.
    Moreover, for any such $k_i$ edge-disjoint $(v_i,v)$-paths, no $(v_i,u_i)$-path is edge-disjoint with each of these paths, for otherwise $\lambda_{G_i}(v_i,v)=k_i+1$, contradicting $\lambda(G_i)=k_i$.
    It now follows that $\lambda_G(v_0,v_1) \le \max\{k_0,k_1\}$.

    Now consider distinct $x,y \in V(G_0)$.
    Observe that any $(x,y)$-path in $G$ that is not also in $G_0$ must contain $e$, so $\lambda_G(x,y) \le \lambda_{G_0}(x,y)+1$.
    Let $\lambda_{G_0}(x,y) = k$, and suppose one of the $k$ edge-disjoint $(x,y)$-paths in $G_0$ contains $e_0$.
    Then, without loss of generality, this path has an $(x,u_0)$-subpath and a $(v,y)$-subpath, in which case, by concatenating these with $u_0,e,u_1$ and a $(u_1,v)$-path in $G_1 \ba e_1$ (which exists since $e_1$ is not a bridge), we obtain one further $(x,y)$-path in $G$ that is edge-disjoint to the other paths.
    Thus $\lambda_G(x,y) = k$ in this case.
    On the other hand, suppose there are $k$ edge-disjoint $(x,y)$-paths in $G_0$ none of which contain $e_0$.
    Then there is no $(x,y)$-path in $G$ that is edge-disjoint to each of these paths, for such a path would contain $e$ and a $(u_1,v)$-subpath in $G_1$, which we could replace with $e_0$ in order to obtain an $(x,y)$-path in $G_0$ containing $e_0$ that is edge-disjoint to each of the $k$ paths.
    Thus $\lambda_G(x,y)=k = \lambda_{G_0}(x,y)$ in any case.
    By symmetry, we deduce that $\lambda_{G_i}(x,y) = \lambda_G(x,y)$, for any $i \in \{0,1\}$ and any $x,y \in V(G_i)$.

    It now follows that $\lambda(G) \le \max\{k_0,k_1\}$.
    Moreover, there exists $i \in \{0,1\}$ such that $G_i$ has distinct vertices $x,y$ with $\lambda_{G_i}(x,y) = \max\{k_0,k_1\}$, in which case $\lambda_{G}(x,y) = \max\{k_0,k_1\}$.
    So $\lambda(G) = \max\{k_0,k_1\}$, as required.
\end{proof}

\subsection{Haj\'{o}s joins and polar list colouring}
\label{joinablesimulation}

We start with a technical, but important, lemma for polar graphs whose underlying graph can be obtained by a Haj\'os join.  Essentially, it reduces a polar list colouring problem on such a polar graph to a polar list colouring problem on one of the parts.

\begin{lemma}
\label{joinabilitysimple}
    Let $Q$ be a polar graph where $G(Q)$ is the Haj\'{o}s join $(G_0,v,u_0) \hajos (G_1,v,u_1)$ of graphs $G_0$ and $G_1$ with $u_i,v \in V(G_i)$ for $i \in \{0,1\}$.
    Suppose $G_1$ is $k$-joinable. 
    Let $\Gamma=(L,R)$ be a $k$-polar assignment for $Q$,
    let $Q_i = Q[V(G_i)]$ for $i \in \{0,1\}$,
    and let $C = \{c \in L(v) : (v,c) \textrm{ forces $u_1$ in }(Q_1,\Gamma)\}$.
    \begin{enumerate}
        \item If $C=\emptyset$, then any $\Gamma$-colouring of $Q_0$ extends to $Q$.\label{jbl c}
        \item Otherwise, there exists a polar graph~$Q'_0$ with $G(Q'_0) = G_0$ and $Q'_0 \ba vu_0 = Q_0$, and a $k$-polar assignment~$\Gamma_0$ for $Q'_0$ such that $\Gamma_0|Q_0=\Gamma|Q$ and every $\Gamma_0$-colouring of $Q'_0$ extends to a $\Gamma$-colouring of $Q$.\label{jbl ab}
    \end{enumerate}
\end{lemma}

\begin{proof}
    Let $e=u_0u_1$, and let $\phi$ be a $\Gamma$-colouring of $Q_0$.
    Note that, in order to find an extension of $\phi$ to $Q$, it suffices to find a $\Gamma$-colouring $\phi'$ of $Q_1$ such that $\phi'(v)=\phi(v)$ and $\phi(u_0)$ and $\phi'(u_1)$ do not cause $e$ to violate \ref{Gamma 2} (see \cref{hajoscolourfig}).
    If $e$ is not polarised, we require that $\phi'(u_1) \neq \phi(u_0)$.
    If $e$ is polarised, with $R(e,u_0,u_1)=(\phi(u_0),c_1)$, then we require that $\phi'(u_1) = c_1$.
    Therefore, in either case, it suffices to find a $\Gamma$-colouring~$\phi'$ of $Q_1$ such that $\phi'(v)=\phi(v)$ and $\phi'(u_1) \neq c_1$ for some $c_1$.
    Since $G_1$ is $k$-joinable, $Q_1$ is $k$-unrestricted and, in particular, $k$-choosable.
    If $C = \emptyset$, then $(v,\phi(v))$ does not force $u_1$ in $Q_1$, so there exists the desired $\Gamma$-colouring $\phi'$.
    This proves \ref{jbl c}.

    Now suppose $C \neq \emptyset$.
    Let $Q_0^N$ and $Q_0^P$ be the graphs obtained from $Q_0$ by adding a non-polarised or polarised edge $u_0v$, respectively.
    Towards \ref{jbl ab}, we prove the following claim.
    \begin{claim}
        \label{jblbclaim}
        \ 
        \begin{enumerate}[label=\rm(\Roman*)]
            \item If $|C|=1$, or $|C| > 1$ and $u_0u_1$ is polarised, then there is a $k$-polar assignment~$\Gamma_P$ of $Q_0^P$ such that $\Gamma_P|Q_0=\Gamma|Q_0$ and any $\Gamma_P$-colouring of $Q_0^P$ extends to $Q$.\label{jbl b}
            \item If $|C| > 1$ and $u_0u_1$ is not polarised, then any $\Gamma$-colouring of $Q_0^N$ extends to $Q$.\label{jbl a}
        \end{enumerate}
    \end{claim}
    \begin{subproof}
        To prove \ref{jbl a}, assume that $|C| > 1$ and $e$ is not polarised.
        Since $G_1$ is $k$-joinable and $G(Q_1)=G_1\ba u_1v$, we have that $u_1$ is a $k$-mimic of $v$ in $Q_1$.
        Therefore, since $|C|>1$, for all $c \in C$ we have that $(v,c)$ forces $(u_1,c)$ in $(Q_1,\Gamma)$.
        Now 
        $\phi$ will not extend to $Q$ if and only if $\phi(v) \in C$ and $\phi(u_0)=\phi(v)$.
        It follows that
        any $\Gamma$-colouring of $Q_0^N$ extends to $Q$, as required.

        To prove \ref{jbl b}, first assume that $e$ is not polarised and $|C|=1$.
        Let $C=\{c\}$ and $d \in \mathbb{N}$ such that
        $(v,c)$ forces $(u_1,d)$ in $(Q_1,\Gamma)$ 
        (note that, even though $G_1$ is $k$-joinable, we could have $d \neq c$ since $|C|=1$).
        Since $e$ is not polarised, 
        $\phi$ does not extend to $Q$ if and only if $\phi(v)=c$ and $\phi(u_0)=d$.
        Let $\Gamma_P=(L|Q_0,R_P)$ be the $k$-polar assignment of $Q_0^P$ with $R_P|Q_0=R|Q_0$ and $R_P(e_P,u_0,v)=(d,c)$, where $e_P$ is the polarised edge in $Q_0^P$ between $u_0$ and $v$.
        Then any $\Gamma_P$-colouring of $Q_0^P$ extends to $Q$.

        Finally, assume that $|C| \geq 1$ and $e$ is polarised.
        Let $R(e,u_0,u_1)=(c_0,c_1)$ and note that 
        $\phi$ does not extend to $Q$ if and only if $\phi(v) \in C$ and $\phi(u_0)=c_0$ and $(v,\phi(v))$ forces $(u_1,c_1)$ in $(Q_1,\Gamma)$.
        If $|C| \ge 2$, then, as $G_1$ is $k$-joinable, for each $c \in C$ we have that $(v,c)$ forces $(u_1,c)$ in $(Q_1,\Gamma)$.
        Thus there is at most one $c\in C$ such that $(v,c)$ forces $(u_1,c)$ in $(Q_1,\Gamma)$, namely $c=c_1$ (if $c_1 \in C$).
        Let $c=c_1$ if $c_1 \in C$, otherwise choose $c \in \mathbb{N}$ arbitrarily.
        This time we let $R_P(e_P,u_0,v)=(c_0,c)$ and once again we have that any $\Gamma_P$-colouring of $Q_0^P$ extends to $Q$.
    \end{subproof}

    Now \ref{jbl ab} follows from \cref{jblbclaim} by setting $Q'_0=Q_0^P$ and $\Gamma_0 = \Gamma_P$ if \ref{jbl b} holds, and setting $Q'_0=Q_0^N$ and $\Gamma_0 = \Gamma$ otherwise.
\end{proof}

The next lemma is essentially a simpler form of \cref{joinabilitysimple}, which is often sufficient for our needs.

\begin{lemma}
    \label{hajosextend}
    Let $Q$ be a polar graph such that $G(Q)$ is the Haj\'{o}s join of graphs $G_0$ and $G_1$, where $G_1$ is $k$-joinable.
    Let $\Gamma$ be a $k$-polar assignment for $Q$. 
    There exists a polar graph~$Q'_0$ with $G(Q'_0) = G_0$ and a $k$-polar assignment~$\Gamma_0$ for $Q'_0$ such that $\Gamma_0|Q[V(G_0)]
    =\Gamma|Q$ and every $\Gamma_0$-colouring of $Q'_0$ extends to a $\Gamma$-colouring of $Q$.
\end{lemma}
\begin{proof}
    Let $G(Q) = (G_0,v_0,u_0) \hajos (G_1,v_1,u_1)$ with $u_i,v_i \in V(G_i)$ for $i \in \{0,1\}$.
    Let $C = \{c \in L(v_1) : (v_1,c) \textrm{ forces $u_1$ in }(Q[V(G_1)],\Gamma)\}$.
    If $C \neq \emptyset$, then the lemma holds by \cref{joinabilitysimple}\ref{jbl ab}.
    So suppose $C = \emptyset$.
    By \cref{joinabilitysimple}\ref{jbl c}, any $\Gamma$-colouring of $Q[V(G_0)]$ extends to $Q$.
    Let $\Gamma = (L,R)$.
    We add a polarised edge $e_0=u_0v_0$ to $Q[V(G_0)]$ to obtain a polar graph $Q_0'$ with $G(Q_0') = G_0$, and let $\Gamma_0=(L,R')$ where $R'$ is obtained from $R$ but with $R'(e_0,u_0,v_0) = (c,c')$ where $c$ and $c'$ are new for $\Gamma$.
    Then $Q$ is $\Gamma$-colourable
    if and only if $Q_0'$ is $\Gamma_0$-colourable, and the result follows.
\end{proof}

\begin{figure}
    \centering
    \includegraphics[scale=0.6]{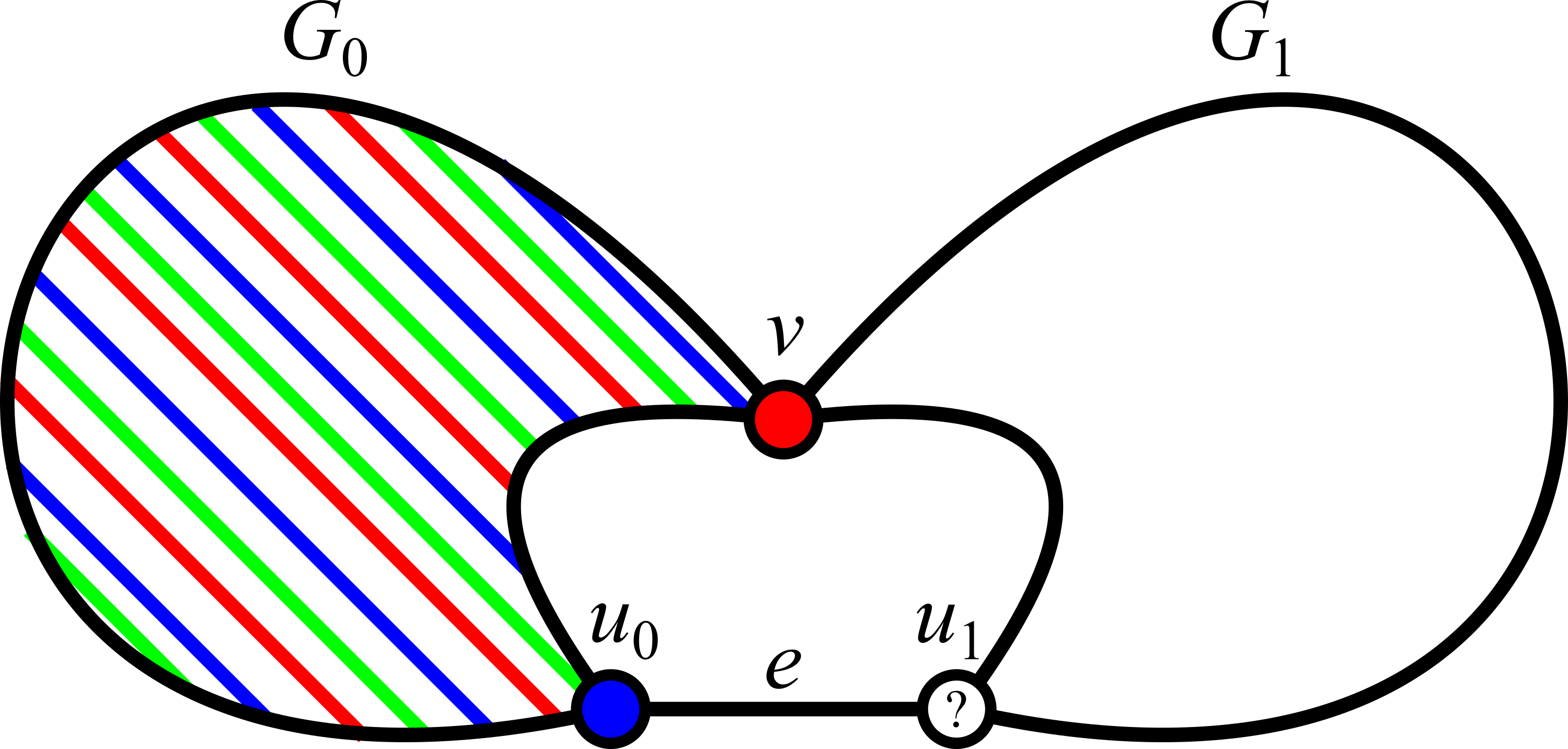}
    \caption{A Haj\'{o}s join of graphs $G_0$ and $G_1$. 
        Here $\phi$ is a colouring of $G_0$, and, to extend $\phi$ to a colouring of the entire graph, 
it suffices to find a colouring of $G_1$ where $v$ is coloured $\phi(v)$ and the colour chosen for $u_1$ does not cause $e$ to conflict with \ref{Gamma 2}.}
    \label{hajoscolourfig}
\end{figure}

We use \cref{joinabilitysimple,hajosextend} to prove the following:

\begin{theorem}
    \label{joinability plus hajos}
    If $G$ is the Haj\'{o}s join of two $k$-joinable graphs, 
    then $G$ is $k$-joinable.
\end{theorem}
\begin{proof}
    Let $G = (G_0,v,u_0) \hajos (G_1,v,u_1)$, and $e=u_0u_1$ is the new edge in $G$.
    Let $e_0=u_0v$ and $e_1=u_1v$ be the edges of $G_0$ and $G_1$ respectively that do not exist in $G$.
    Let $Q$ be a polar graph with $G(Q)=G$, and let $Q_i = Q[V(G_i)]$ for $i \in \{0,1\}$.
    
    \begin{claim}
        \label{jph1}
        For any $e' \in E(Q)$, we have that $Q\ba e'$ is $k$-unrestricted.
    \end{claim}
    \begin{subproof}
    Let $e' \in E(Q)$.
    We show that for any $k$-polar assignment $\Gamma=(L,R)$ for $Q \ba e'$, and any $x \in V(Q)$ and any $c_x \in L(x)$, there exists a $\Gamma$-colouring $\phi$ of $Q \ba e'$ with $\phi(x) = c_x$.
    If $e' = e$, then $Q\ba e'$ is a $1$-join of two $k$-unrestricted polar graphs, and therefore is $k$-unrestricted by \cref{1-joins unrestricted}. 
    So we may assume that $e' \in E(Q) \setminus \{e\}$, so $e' \in E(Q_i)$ for some $i \in \{0,1\}$.
    By symmetry, $i=1$, without loss of generality.
    Then $G\ba e' = (G_0,v,u_0) \hajos (G_1\ba e',v,u_1)$, where $e=u_0u_1$ is the added edge.
    We also fix an arbitrary $k$-polar assignment $\Gamma=(L,R)$ for $Q \ba e'$.

    Let $x \in V(G)$ and let $c_x \in L(x)$.
    First suppose that $x \in V(G_0)$.
    Since $G_0$ is $k$-joinable and $G(Q_0)=G_0\ba e_0$, we have that $Q_0$ is $k$-unrestricted, so there exists a $\Gamma$-colouring $\phi$ of $Q_0$ such that $\phi(x)=c_x$.
    By \cref{joinabilitysimple}(i), $\phi$ extends to $Q\ba e'$ unless $(v,\phi(v))$ forces $u_1$ in $(Q_1\ba e',\Gamma)$.
    Note that $G_1\ba e'$ is $k$-polar-unrestricted since $G_1$ is $k$-joinable, and $G(Q_1\ba e')=G_1\ba \{e',e_1\}$.
    Therefore, by \cref{unrestricted vacuous}, there is no $c \in L(v)$ such that $(v,c)$ forces $u_1$ in $(Q_1\ba e',\Gamma)$.
    In particular, $(v,\phi(v))$ does not force $u_1$ in $(Q_1\ba e',\Gamma)$, so $\phi$ extends to $Q\ba e'$.

    Next suppose that $x \in V(G_1)$.
    Since $G_1$ is $k$-joinable, we have that $Q_1 \ba \{e_1,e'\}$ is $k$-unrestricted, so there is a $\Gamma$-colouring $\phi$ for $Q_1 \ba \{e_1,e'\}$ with $\phi(x) =c_x$.
    If there is no $c \in L(v)$ such that $(v,c)$ forces $u_0$ in $(Q_0,\Gamma)$, then $\phi$ extends to $Q \ba e'$ by \cref{joinabilitysimple}(i).
    Otherwise, by \cref{joinabilitysimple}(ii), since $G_0$ is $k$-joinable, there exists a polar graph $Q_1'$ with $G(Q_1')=G_1\ba e'$, and a $k$-polar assignment $\Gamma'$ for $Q_1'$ such that $\Gamma'|Q_1'=\Gamma|Q_1$ and any $\Gamma'$-colouring of $Q_1'$ extends to a $\Gamma$-colouring of $Q \ba e'$.
    Moreover, since $G_1$ is $k$-joinable, $Q_1'$ is $\Gamma'$-unrestricted, so there exists a $\Gamma'$-colouring $\phi$ of $Q_1'$ that extends to a $\Gamma$-colouring of $Q \ba e'$ with $\phi(x)=c_x$.
\end{subproof}

    By \cref{jph1}, it remains to show, for each edge $e'=xy$ of $Q$, that $y$ is a $k$-mimic of $x$ in $Q \ba e'$.
    Let $e'=xy$ be an edge of $Q$ and, towards a contradiction, assume that $y$ is not a $k$-mimic of $x$.
    Let $\Gamma = (L,R)$ be a $k$-polar assignment for $Q \ba e'$ such that $x$ $2$-forces $y$ in $(Q\ba e',\Gamma)$, but there is some $c \in L(x)$ such that $(x,c)$ forces $(y,d)$ in $(Q\ba e', \Gamma)$, with $c \neq d$.
    Assume that $e' \in E(G_1)$, and therefore $G$ is the Haj\'{o}s join of $G_0$ and $G_1\ba e'$ with identified vertex $v$ and added edge $e=u_0u_1$.
    By \cref{hajosextend}, there exists a polar graph $Q'_1$ with $G(Q'_1)=G_1\ba e'$ and a $k$-polar assignment $\Gamma'=(L',R')$ for $Q'_1$ such that $\Gamma'|Q_1=\Gamma|Q_1$ and any $\Gamma'$-colouring of $Q'_1$ extends to $Q$.
    Since $G(Q'_1)=G_1\ba e'$ and $G_1$ is $k$-joinable, for every $c' \in L'(x)$ there exists a $\Gamma'$-colouring $\phi$ of $Q'_1$ such that $\phi(x)=c'$.
    If there is some $c' \in L(x)$ such that $(x,c')$ forces $(y,d')$ in $Q\ba e'$, then $c' \in L'(x)$ by definition, and so there exists a $\Gamma'$-colouring $\phi$ of $Q'_1$ such that $\phi(x)=c'$.
    Since $\phi$ extends to $Q$, we have that $\phi(y)=d'$.
    Therefore $(x,c')$ forces $(y,d')$ in $(Q'_1,\Gamma')$.
    This implies that $x$ $2$-forces $y$ in $(Q'_1,\Gamma')$, and $(x,c)$ forces $(y,d)$ in $(Q'_1,\Gamma')$, with $c \neq d$. So
    $y$ is not a $k$-mimic of $x$ in $Q'_1$, contradicting the fact that $G_1$ is $k$-joinable.

    By symmetry, it remains only to examine the case when $e'=e$ and $(x,y) = (u_0,u_1)$.
    Now $\Gamma = (L,R)$ is a $k$-polar assignment for $Q \ba e$ with $c \in L(u_0)$ such that $(u_0,c)$ forces $(u_1,d)$ in $(Q\ba e,\Gamma)$.
    We first claim that $(u_0,c)$ forces $v$ in $(Q\ba e, \Gamma)$.
    Suppose not.
    Let $\phi$ and $\phi'$ be $\Gamma$-colourings of $Q_0$ such that $\phi(u_0)=\phi'(u_0)=c$ but $\phi(v) \neq \phi'(v)$.
Since $Q_1$ is $k$-unrestricted, both $\phi$ and $\phi'$ extend to $Q\ba e$, and so let $\Phi_E$ be the set of all such extensions of $\phi$ and let $\Phi_E'$ be the set of all such extensions of $\phi'$. Since $(u_0,c)$ forces $(u_1,d)$ in $(Q\ba e,\Gamma)$, we have that $\phi_E(u_1)=d$ for all $\phi_E \in \Phi_E \cup \Phi_E'$. However, this implies that both $(v,\phi(v))$ and $(v,\phi'(v))$ force $(u_1,d)$ in $Q_1$, which contradicts the fact that $u_1$ is a $k$-mimic of $v$ in $Q_1$, since $\phi(v) \neq \phi'(v)$.
    Thus $(u_0,c)$ forces $(v,c_v)$, say, in $(Q\ba e, \Gamma)$.

    Next we claim that $(v,c_v)$ forces $(u_1,d)$ in $(Q\ba e, \Gamma)$.
    Suppose not.
    Let $\phi$ be a $\Gamma$-colouring of $Q_0$ where $\phi(u_0)=c$, so $\phi(v)=c_v$.
    Since $(v,c_v)$ does not force $(u_1,d)$ in $(Q_1,\Gamma)$, there exists some extension $\phi_E$ of $\phi$ to $Q$ such that $\phi_E(u_1) \neq d$.
    This contradicts the fact that $(u_0,c)$ forces $(u_1,d)$ in $(Q\ba e, \Gamma)$.

We have now shown that
$(u_0,c)$ forces $(v,c_v)$ in $(Q_0,\Gamma)$ and $(v,c_v)$ forces $(u_1,d)$ in $(Q_1,\Gamma)$.
Since $v$ is a $k$-mimic of $u_0$ in $Q_0$ and $u_1$ is a $k$-mimic of $v$ in $Q_1$, we have that $u_1$ is a $k$-mimic of $u_0$ in $Q\ba e$, as required.
\end{proof}

We also obtain the next two lemmas using \cref{joinabilitysimple,hajosextend}.

\begin{lemma}
    \label{hajos joining unrestricted}
    If $G$ is the Haj\'{o}s join of two $k$-polar-unrestricted graphs, then $G$ is $k$-polar-unrestricted.
\end{lemma}
\begin{proof}
    Let $G_0$ and $G_1$ be graphs with $u,v_i \in V(G_i)$ for $i \in \{0,1\}$, let $G = (G_0,v,u_0) \hajos (G_1,v,u_1)$, and let $Q$ be a polar graph with $G(Q)=G$.
    Let $\Gamma = (L,R)$ be a $k$-polar assignment, and let $r \in V(Q)$ and $c \in L(r)$.
    Without loss of generality, assume that $r \in V(G_0)$.
    Since $G_0$ is $k$-polar-unrestricted, $Q[V(G_0)]$ is $k$-unrestricted, so there exists a $\Gamma$-colouring $\phi$ of $Q[V(G_0)]$ such that $\phi(r)=c$.
    By \cref{unrestricted joinable corollary}, $G_1$ is $k$-joinable.
    Moreover, there is no $c \in L(v)$ such that $(v,c)$ forces $u_1$ in $(Q[V(G_1)],\Gamma)$, by \cref{unrestricted vacuous}.
    Thus, by \cref{joinabilitysimple}\ref{jbl c}, $\phi$ extends to $Q$. Hence $r$ is not $\Gamma$-restricted on $c$.
\end{proof}

\begin{lemma}
    \label{joinable plus unrestricted}
    Let $G$ be a Haj\'{o}s join of graphs $G_0$ and $G_1$. 
    Suppose that $G_0$ is $k$-joinable, and $G_1$ is $k$-polar-unrestricted.
    Then $G$ is $k$-polar-unrestricted.
\end{lemma}
\begin{proof}
    Let $G = (G_0,v,u_0) \hajos (G_1,v,u_1)$, and let $e_0=u_0v$ and $e_1=u_1v$ be the edges of $G_0$ and $G_1$, respectively, not present in $G$.
    Let $Q$ be a polar graph with underlying graph $G$,
    let $\Gamma=(L,R)$ be a $k$-polar assignment for $Q$, and let $r \in V(Q)$ with $c \in L(r)$.
    Let $Q_i = Q[V(G_i)]$ for $i \in \{0,1\}$.

    First assume that $r \in V(G_0)$.
    Since $G_0$ is $k$-joinable and $G(Q_0) = G_0 \ba e_0$, we have that $Q_0$ is $\Gamma$-unrestricted.
    Thus, there exists a $\Gamma$-colouring $\phi$ of $Q_0$ such that $\phi(r)=c$.
    Since $G_1$ is $k$-joinable by \cref{unrestricted joinable corollary}, and there is no $c \in L(v)$ such that $(v,c)$ forces $u_1$ in $(Q_1,\Gamma)$ by \cref{unrestricted vacuous}, we have that $\phi$ extends to $Q$ by \cref{joinabilitysimple}\ref{jbl c}.
    Hence $r$ is not $\Gamma$-restricted on $c$ in $Q$.

Next assume that $r \in V(G_1)$. 
By \cref{hajosextend}, there exists a polar graph $Q'_1$ with $G(Q'_1)=G_1$, and a $k$-polar assignment $\Gamma'$ for $Q'_1$ such that any $\Gamma'$-colouring of $Q'_1$ extends to a $\Gamma$-colouring of $Q$.
Since $G_1$ is $k$-polar-unrestricted, there exists a $\Gamma'$-colouring $\phi'$ of $Q'_1$ where $\phi'(r)=c$.
Thus $\phi'$ extends to a $\Gamma$-colouring of $Q$, and therefore $r$ is not $\Gamma$-restricted on $c$ in $Q$.
It now follows that $Q$ is $k$-polar-unrestricted, as required.
\end{proof}

\subsection{The class $\mathcal{T}_k$}

We now prove that if $G$ is the almost $k$-regular extension of a Gallai tree (even one with $K_2$-blocks), then $G$ is in $\mathcal{T}_k$.

\begin{lemma}
\label{Hajos joins give k2}
Let $G$ be a graph with a vertex $h \in V(G)$ such that $G-h$ is a Gallai tree, and $d_G(v)=k$ for all $v \in V(G) \setminus \{h\}$. Then $G \in \mathcal{T}_k$.
\end{lemma}
\begin{proof}
    Observe that $G$ is the almost $k$-regular extension of $G-h$.
    Thus, if $G-h$ has no $K_2$-blocks, then $G$ is clearly in $\mathcal{T}_k^-$, and so is also in $\mathcal{T}_k$. Therefore, the lemma holds when $G-h$ has no $K_2$-blocks.

    Towards a contradiction, suppose that there exists such a graph $G$ with $G \notin \mathcal{T}_k$.
    We may assume $G$ has the minimum number of vertices amongst all such graphs.
    By the foregoing, $G-h$ has a $K_2$-block $B$.
    Let $V(B) = \{u_0,u_1\}$ and $E(B) = \{e\}$.
    Since $B$ is a block of $G-h$, the edge $e$ is a bridge of $G-h$.
    Let $H_0$ and $H_1$ be the two components of $(G-h)\ba e$, where $u_i \in H_i$ for each $i \in \{0,1\}$.
    Note that $H_0$ and $H_1$ are both Gallai trees.
    For $i \in \{0,1\}$, let $G_i$ for be the graph obtained from $G[V(H_i) \cup \{h\}]$ by adding an edge $e_i=hu_i$.
    We have $G_i-h=H_i$, where $H_i$ is a Gallai tree.
    Moreover, for any $v \in V(H_i)$, we have $d_{G_i}(v)=d_G(v)=k$ (this is clear if $v \neq u_i$, whereas if $v=u_i$, then, in $G_i$, we lose the edge $e$, but add the edge $e_i$).
    Therefore, by the minimality of $G$, both $G_0$ and $G_1$ are in $\mathcal{T}_k$.
    However, $G$ is the Haj\'{o}s join of $G_0$ and $G_1$, so $G \in \mathcal{T}_k$, a contradiction.
\end{proof}

In \cref{the hajos decomp lemma}, we will show that there is a natural way to decompose a graph $G$ into two components using an edge cut of size at most $k$ such that when the components are in $\mathcal{T}_k$, then $G$ is in $\mathcal{T}_k$.
We now describe the graphs of this decomposition.

Suppose that a graph $G$ has a edge cut $S$ that separates $X$ from $Y$, where $(X,Y)$ is a partition of $V(G)$. We use the following notation to describe two particular graphs obtained from $G$ and $(X,Y)$.
Let $X_S$ (or $Y_S$) be the subset of $X$ (or $Y$, respectively) consisting of vertices incident to an edge in $S$.
For $e \in S$, let $x_e$ (or $y_e$) be the vertex in $X_S$ (or $Y_S$, respectively) incident to $e$ in $G$.
Let $G'_X$ be the graph obtained from $G[X]$ by adding a new vertex $w$ and, for each $e \in S$, adding a single edge $e$ between $w$ and $x_e$.
Note that there may be parallel edges incident with $w$ in $G'_X$.
Let $G_Y$ be the graph obtained from $G[Y]$ by adding a set $U=\{u_e : e \in S\}$ of $|S|$ new vertices such that $U$ is a clique in $G_Y$, and adding an edge $e=u_ey_e$ for each edge $e \in S$.
Note that in $Q'_X$, the set of edges incident with $w$ is $S$; whereas $S$ is an edge cut in $G_Y$.
For an illustration, see \cref{K Edge Cut Deconstruction image}.

\begin{figure}
    \centering
    \includegraphics[scale=0.5]{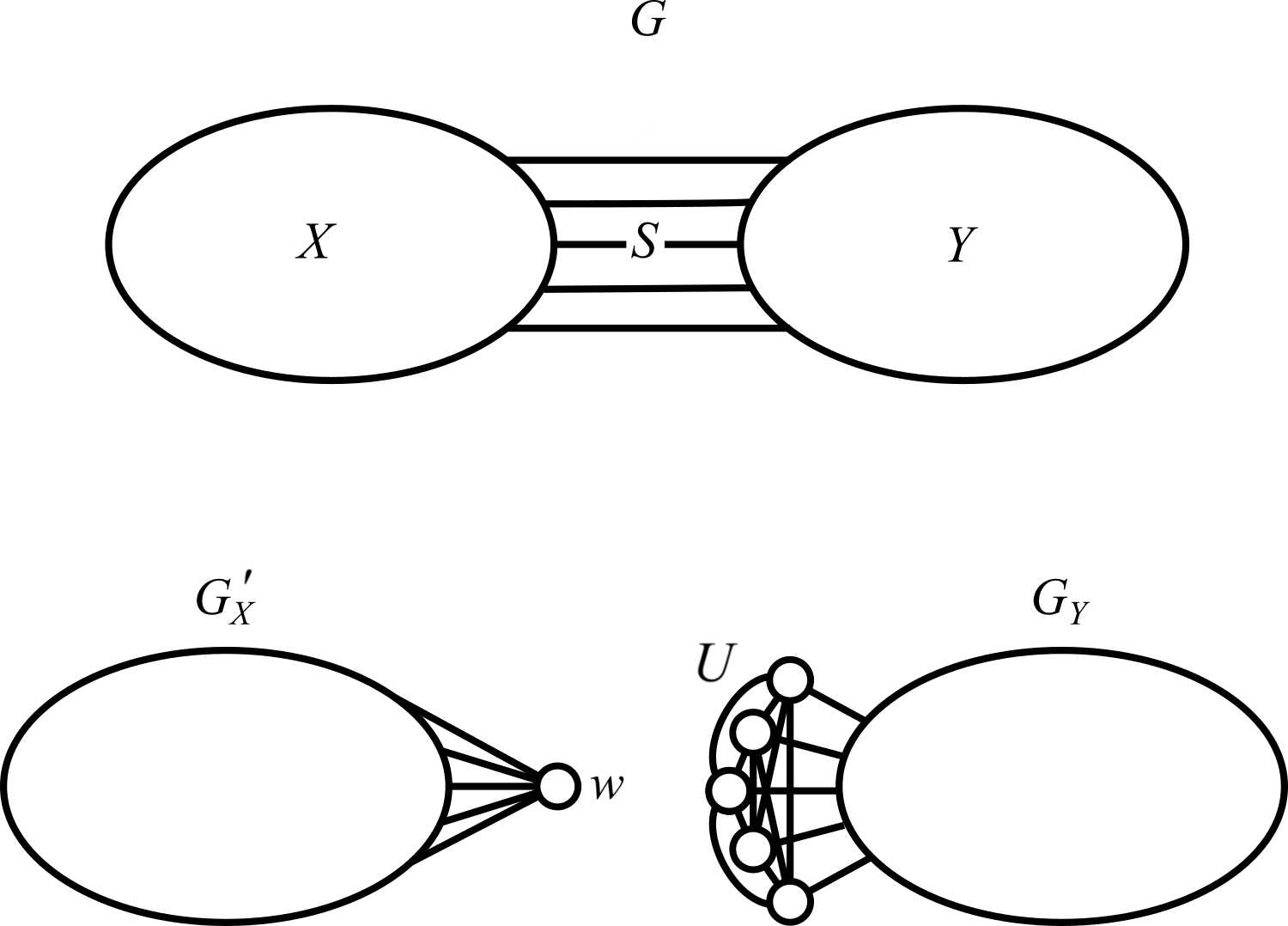}
    \caption{The graphs $G'_X$ and $G_Y$}
    \label{K Edge Cut Deconstruction image}
\end{figure}

\begin{lemma}
    \label{doi-new}
    Let $k \ge 3$, and let $G$ be the Haj\'os join of graphs $H_0$ and $H_1$, with $H_0,H_1 \in \mathcal{T}_k$ and $V(H_0) \cap V(H_1)=\{z\}$.
    Let $H \in \{H_0,H_1\}$ and $v \in V(H)$.  Then
    \begin{enumerate}
        \item $d_G(v) \ge d_H(v) \ge k$, and
        \item $d_G(v) > d_H(v)$ if and only if $v=z$.
    \end{enumerate}
\end{lemma}
\begin{proof}
    First, assume that $H_0,H_1 \in \mathcal{T}_k^-$.
    Since $H \in \mathcal{T}_k^-$, we have $d_H(v) \ge k \ge 3$ for all $v \in V(H)$.
    Moreover, if $u \in V(H) \setminus \{z\}$, then clearly $d_{H}(u) = d_{G}(u)$.
    On the other hand, $d_G(z) = d_{H_0}(z) + d_{H_1}(z) - 2 \ge d_{H_0}(z) + 1$, so, by symmetry, $d_G(z) > d_H(z)$.
    The lemma follows by a routine induction argument.
\end{proof}

\begin{lemma}
    \label{the hajos decomp lemma}
    Let $k \ge 3$, and let $G$ be a graph with a minimal edge cut $S$ that separates $X$ from $Y$, where $(X,Y)$ is a partition of $V(G)$, and $|S| \le k$.
    If $G'_X$ and $G_Y$ are both in $\mathcal{T}_k$, then $G \in \mathcal{T}_k$.
\end{lemma}
\begin{proof}
    Assume that $G'_X$ and $G_Y$ are in $\mathcal{T}_k$ but, towards a contradiction, $G \notin \mathcal{T}_k$.
    We also assume that $G$ has the minimum number of vertices amongst all such counterexamples.

    \begin{claim}
        \label{the s=k claim}
        $|S|=k$.
    \end{claim}
    \begin{subproof}
        Note that $d_{G'_X}(w)=|S|$.
        If $G'_X \in \mathcal{T}_k^-$, then $G'_X$ has minimum degree~$k$.
        If $G'_X \notin \mathcal{T}_k^-$, then, by \cref{doi-new}(i), $G'_X$ again has minimum degree~$k$.
        So, in either case, $|S| = d_{G'_X}(w) \geq k$.
        Since $|S| \le k$, we have $|S|=k$, as required.
    \end{subproof}

    Suppose that $G_Y \in \mathcal{T}_k^-$.
    By \cref{the s=k claim}, $G_Y[U] \cong K_k$, and so, by \cref{tkcliques}, $G_Y \cong K_{k+1}$.
    Then $Y=V(G_Y) \setminus U$ consists of a single vertex, and it follows that $G'_X \cong G$.
    Therefore, since $G'_X \in \mathcal{T}_k$, we have that $G \in \mathcal{T}_k$, as required.

    Therefore, we may assume that $G_Y \notin \mathcal{T}_k^-$, so there are graphs $G_0, G_1 \in \mathcal{T}_k$ such that $G_Y = (G_0,z,v_0) \hajos (G_1,z,v_1)$, with $z,v_i \in V(G_i)$ for $i \in \{0,1\}$.
    Let $e_0=v_0z$ and $e_1=v_1z$ be the edges of $G_0$ and $G_1$, respectively, not present in $G_Y$.
    We first prove the following claim.

    \begin{claim}
        \label{the Gi claim}
        For some $i \in \{0,1\}$, we have $U \subseteq V(G_i) \setminus \{z\}$, and $U$ is a clique in $G_i$.
    \end{claim}
    \begin{subproof}
        Towards a contradiction, assume that there exist $u_0 \in U \cap V(G_0)$ and $u_1 \in U \cap V(G_1)$.
        If $u_0=z$, then $d_{G_Y}(u_0) > d_{G_0}(u_0) \ge k$ by \cref{doi-new}, which contradicts that every vertex in $U$ has degree $k$ in $G$, by the construction of $G_Y$ and \cref{the s=k claim}.
        So $u_0 \in V(G_0)\setminus \{z\}$ and, similarly, $u_1 \in V(G_1)\setminus \{z\}$.
        Since $u_0$ and $u_1$ are adjacent in $G_Y$, we have $u_0=v_0$ and $u_1=v_1$.
        However, since $|U| = k$, by \cref{the s=k claim}, and $k \geq 3$, there is some $u' \in U \setminus \{u_0,u_1\}$ adjacent to both $v_0$ and $v_1$ in $G_Y$. This vertex can only be $z$, but then $d_{G_Y}(z) >k$ and so $z \notin U$, a contradiction.
        From this contradiction, we deduce that $U \subseteq V(G_i) \setminus \{z\}$, for some $i \in \{0,1\}$.
        Now $G_Y[U] = G_i[U]$.  Therefore, as $U$ is a clique in $G_Y$, it is also a clique in $G_i$.
    \end{subproof}

    Without loss of generality assume that $i=0$ in \cref{the Gi claim}, so $U \subseteq V(G_0) \setminus \{z\}$ and $G_0[U]$ is a complete graph, while $U \cap V(G_1)=\emptyset$.
    Note that $|U|=k$ and, for each $u \in U$, there is a unique edge in $G_Y$ between $u$ and a vertex not in $U$.

    \begin{claim}
    \label{G0-UY is connected}
        $G_0-U$ is connected.
    \end{claim}
    \begin{subproof}
        Since the edge cut $S$ is minimal, $G[Y]=G_Y-U$ is connected.
        Towards a contradiction, assume that $G_0-U$ is not connected.
        Then there exist distinct vertices $v,v' \in V(G_0)-U$ such that there is no $(v,v')$-path in $G_0-U$.
        Observe that $v,v' \in V(G_Y-U)$.
        Since $G_Y-U=G[Y]$ is connected, there exists a $(v,v')$-path in $G_Y-U$.
        Since this path does not exist in $G_0-U$, it contains a vertex in $V(G_1) \setminus \{z\}$, and therefore it contains both $z$ and $v_0$.
        However, we can replace a $(z,v_0)$-subpath with the edge $e_0=zv_0$ in $G_0$ to obtain a $(v,v')$-path in $G_0-U$, a contradiction.
    \end{subproof}

    We now construct a graph $G'$ by taking the union of $G_0-U$ and $G[X]$, and then adding an edge for each $e \in S$ as follows.
    For each $e \in S$, we have $u_e \in U \subseteq V(G_0) \setminus \{z\}$ so, as $u_ey_e$ is an edge of $G_Y$, either $y_e \in V(G_0)$ or $y_e = v_1$; if $y_e \in V(G_0)$ then we add the edge $e=x_ey_e$ in $G'$, otherwise, when $y_e=v_1$, we add the edge $e=x_ez$ in $G'$.
    Now $S$ is a set of edges in $G'$ (as well as a set of edges in $G$).
    Let $Y' = V(G_0-U)$.
    Notice that each edge of $G'$ in $S$ is incident to a vertex in $X$, and a vertex in $Y'$, so $G'\ba S$ consists of two connected components, $G[X]$ and $G_0-U$, where the latter is connected by \cref{G0-UY is connected}.
    Now $(X,Y')$ is a partition of $V(G')$ where $S$ is a minimal edge cut of $G'$ that separates $X$ from $Y'$ with $|S| \le k$.
    Both $G'_X$ and $G_{Y'} = G_0$ are in $\mathcal{T}_k$, and $|V(G')| < |V(G)|$, so, by the minimality of $G$, we have that $G'$ is in $\mathcal{T}_k$.

    We now define a graph $G''$ obtained from a Haj\'os join on $G'$ and $G_1$.
    Suppose $v_0 \notin U$. Then $e_0=zv_0 \in E(G_0-U)$, so $e_0 \in E(G')$.
    In this case, we let $G'' = (G',z,v_0) \hajos (G_1,z,v_1)$.
    On the other hand, suppose $v_0 \in U$.
    Then $v_0 = u_e$ for some $e \in S$, in which case $y_e = v_1$, so $x_ez$ is an edge of $G'$ by construction.
    In this case, we let $G'' = (G',z,x_e) \hajos (G_1,z,v_1)$.
    In either case, we let $e'$ be the edge added in $G''$ that is not present in $G'$ or $G_1$.
    A rough diagram of these graphs is given in \cref{schemfig}, focussing on the case that $v_0 \notin U$.
    Since $G''$ is the Haj\'os join of two graphs in $\mathcal{T}_k$, we have $G'' \in \mathcal{T}_k$.

    \begin{figure}
        \centering
        \includegraphics[scale=0.5]{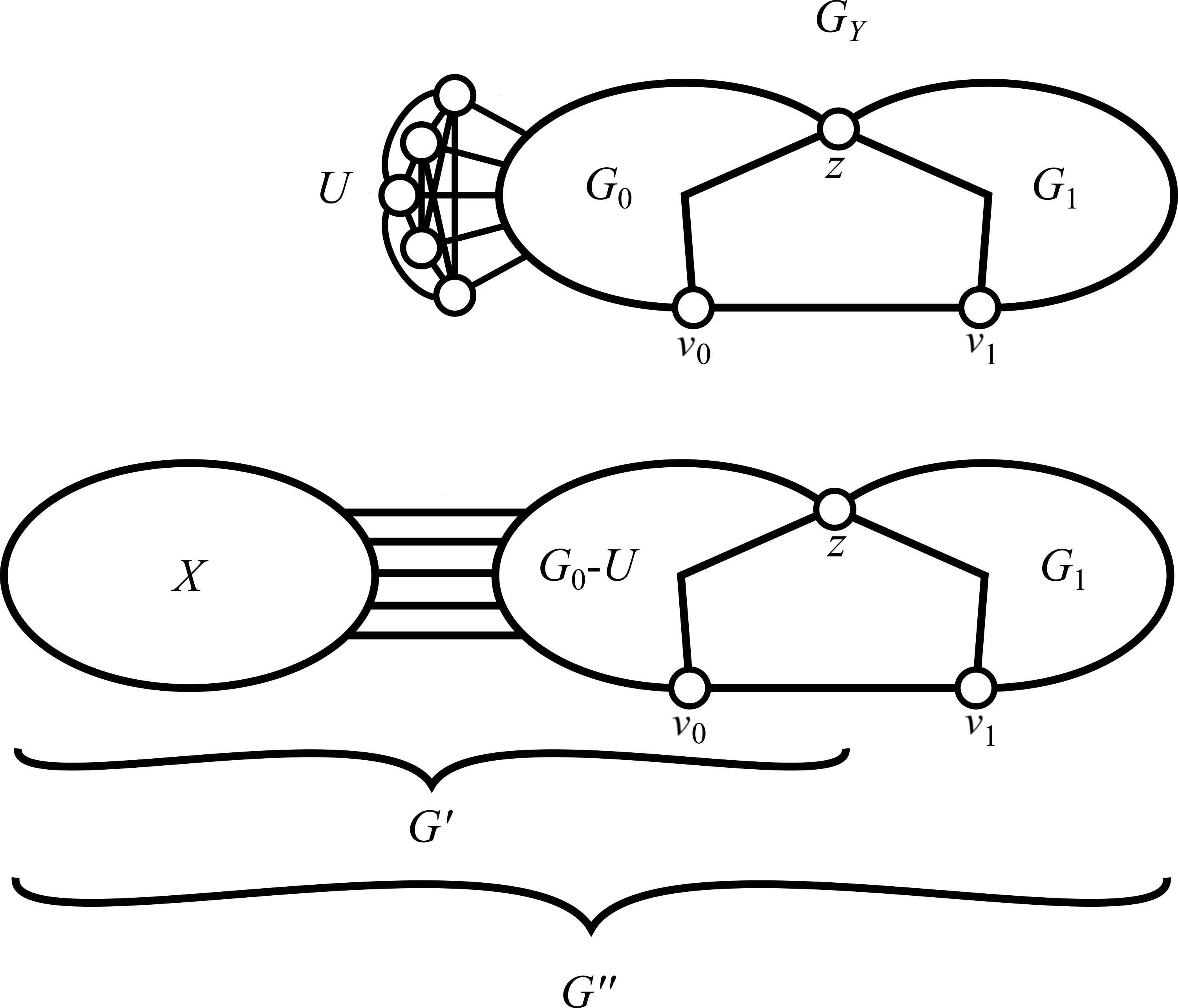}
        \caption{A schematic diagram of $G_Y$, $U$, $G_0$, $G_1$, $G'$ and $G''$.}
        \label{schemfig}
    \end{figure}

    We claim that $G''\cong G$.
    It is not too hard to see that, by construction, $V(G'')=V(G)$ and $E(G'')=E(G)$.
    Moreover, incidences between a vertices $v,v'$ are certainly the same in $G$ and $G''$ when $\{v,v'\} \subseteq X$, or $\{v,v'\} \subseteq V(G_0-U)$, or $\{v,v'\} \subseteq V(G_1)$.
    Consider the case where $v_0 \notin U$.
    By construction, the edges in $S$ have the same ends in $G$ as in $G''$.
    Moreover, the edge $v_0v_1$ in $G_Y$ that is added as part of the Haj\'{o}s join between $G_0$ and $G_1$ corresponds to the edge $e'=v_0v_1$ in $G''$ added as part of the Haj\'{o}s join between $G'$ and $G_1$.
    Now it is clear that $G=G''$ in the case that $v_0 \notin U$.

    It remains to consider the case where $v_0 \in U$.
    Let $e\in S$ such that $u_e = v_0$, in which case $y_e=v_1$.
    It is clear that the edges in $S \setminus \{e\}$ have the same ends in $G$ as in $G''$.
    For $e$, observe that $x_ez$ is the edge added in $G_Y$ as part of the Haj\'{o}s join between $G_0$ and $G_1$.
    In the construction of $G''$ by the Haj\'{o}s join between $G'$ and $G_1$, the edge $x_ez$ is removed and the edge $x_ev_1=x_ey_e$ is added.
    So incidences are preserved for edges in $S$.
    It follows that $G \cong G''$.  Thus $G \in \mathcal{T}_k$, a contradiction.
\end{proof}

We now consider consequences of results in the previous two subsections for the class $\mathcal{T}_k$.


\begin{lemma}
    \label{connectivity properties corollary}
    Let $G$ be a graph in $\mathcal{T}_k$, for $k \ge 3$.
    Then $G$ is $k$-edge-connected and has maximum local edge-connectivity~$k$. 
\end{lemma}
\begin{proof}
    Let $H \in \mathcal{T}_k^-$.  Then $H$ is the almost $k$-regular extension of a Gallai tree~$T$, with hub~$h$.
    Since $H$ has at most one vertex with degree more than $k$, namely $h$, we have that $H$ has maximum local edge-connectivity at most $k$.
    By \cref{k edge con lemma}, $H$ is $k$-edge-connected; it follows that the maximum local edge-connectivity of $H$ is precisely $k$.
    Now, by \cref{hajosconn1,hajosconn2}, it follows that $G \in \mathcal{T}_k$ is $k$-edge-connected and has maximum local edge-connectivity~$k$.
\end{proof}

\begin{corollary}
\label{can't have both corollary}
Let $G$ be a graph with a proper subgraph that is in $\mathcal{T}_k$, for $k \ge 3$.
Then, either $G$ is not $2$-connected, or the maximum local edge-connectivity of $G$ is more than $k$.
\end{corollary}
\begin{proof}
    Assume that $G$ is $2$-connected.
    We will show that the maximum local edge-connectivity of $G$ is more than $k$.
    Let $G'$ be a proper subgraph of $G$ that is in $\mathcal{T}_k$.
    Then $G'$ is $k$-edge-connected, by \cref{connectivity properties corollary}.
    First, suppose there exists $u \in V(G) \setminus V(G')$, and let $v_0$ and $v_1$ be distinct vertices in $V(G')$.
    Since $G$ is $2$-connected, \cref{2fanlemma} implies that there is a $(u,v_0)$-path $P_0$ and a $(u,v_1)$-path $P_1$ that are internally disjoint.
    Let $w_0$ be the first vertex of $P_0$ in $V(G')$ and let $w_1$ be the first vertex of $P_1$ in $V(G_1)$.
    Since $G'$ is $k$-edge-connected, there are $k$ edge-disjoint $(w_0,w_1)$-paths in $G'$.
    By adjoining the $(u,w_0)$-subpath of $P_0$ and the $(u,w_1)$-subpath of $P_1$, we see that there are at least $k+1$ edge-disjoint $(w_0,w_1)$-paths in $G$.
    Now we may assume that $V(G')=V(G)$, so there exists $e \in E(G) \setminus E(G')$.
    Let $e=uv$.
    Since $G'$ is $k$-edge-connected, there are $k$ edge-disjoint $(u,v)$-paths in $G'$.
    Together with $e$, we see that there are at least $k+1$ edge-disjoint $(u,v)$-paths in $G$.
\end{proof}

\begin{lemma}
    \label{the joinability corollary}
    If $G \in \mathcal{T}_k$, for $k\ge 3$, then $G$ is $k$-joinable.
\end{lemma}
\begin{proof}
    Towards a contradiction, suppose there exists a graph in $\mathcal{T}_k$ that is not $k$-joinable.
    Amongst all counterexamples, let $G$ be such a graph with the smallest number of vertices.
    By \cref{tkminustwo}, $G$ is not in $\mathcal{T}_k^-$, and therefore $G$ can be obtained as the Haj\'{o}s join of two graphs, $G_1$ and $G_2$, in $\mathcal{T}_k$.
    Both $G_1$ and $G_2$ have fewer vertices than $G$, so they are $k$-joinable.  But then $G$ is $k$-joinable by \cref{joinability plus hajos}, a contradiction.
\end{proof}

\begin{theorem}
    \label{tkes}
    Let $Q$ be a polar graph where $G(Q) \in \mathcal{T}_k$, for $k \ge 3$.
    Let $\Gamma$ be a $k$-polar assignment for $Q$.
    Then $Q$ is not $\Gamma$-colourable if and only if
    $G(Q) \in \mathcal{H}_k$,
    $F(Q)=\emptyset$,
    and $\Gamma$ is uniform.
\end{theorem}

\begin{proof}
    If $G(Q) \in \mathcal{T}_k^-$, then the result holds by \cref{tk c}.
    So $G(Q)$ is a Haj\'{o}s join of graphs $G_0$ and $G_1$, with $G_0,G_1 \in \mathcal{T}_k$.
    Let $G(Q) = (G_0,v,u_0) \hajos (G_1,v,u_1)$ with $v,u_i \in V(G_i)$ for $i \in \{0,1\}$, and let $e$ be the edge $u_0u_1$ in $G$. 
    Let $Q_i = Q[V(G_i)]$ for $i \in \{0,1\}$.
    By \cref{the joinability corollary}, $G_i$ is $k$-joinable, so $Q_i$ is $k$-unrestricted.

    Let $\Gamma = (L,R)$.
    For $i \in \{0,1\}$, let $Q_i^P$ (and $Q_i^N$) be the polar graph obtained by adding a single polarised (or non-polarised, respectively) edge $e_i=vu_i$ to $Q_i$, so that $G(Q_i^P) = G_i$.
    For $c,d \in \mathbb{N}$, let $\Gamma_i^{(c,d)} = (L_i,R_i)$ be the $k$-polar assignment for $Q_i^P$ with $\Gamma_i^{(c,d)}|Q_i=\Gamma|Q_i$ and $R_i(e_i,u_i,v)=(c,d)$.

    We first prove the following claim, for the case where $e$ is polarised.

    \begin{claim}
        Suppose that $e$ is polarised, and $R(e,u_0,u_1)=(c_0,c_1)$.
        Let $x \in V(G_0)$.
        The vertex $x$ is $\Gamma$-restricted on $a \in L(x)$ in $Q$ if and only if there exists some $c\in\mathbb{N}$ such that $x$ is $\Gamma_0^{(c_0,c)}$-restricted on $a$ in $Q_0^P$, and $v$ is $\Gamma_1^{(c_1,c)}$-restricted on $c$ in $Q_1^P$.\label{+ a}
    \end{claim}

    \begin{subproof}
        Suppose that $x$ is $\Gamma$-restricted on $a \in L(x)$ in $Q$.
        Let $\phi_0$ be an arbitrary $\Gamma$-colouring of $Q_0$ with $\phi_0(x)=a$.
        Then $\phi_0$ does not extend to $Q$.
        Since $Q_1$ is $k$-unrestricted, there exists a $\Gamma$-colouring $\phi_1$ of $Q_1$ such that $\phi_1(v)=\phi_0(v)$.
        Thus, if $\phi_0(u_0) \neq c_0$ or $\phi_1(u_1) \neq c_1$, then $\phi_0$ extends to $Q$, a contradiction.
        So $\phi_0(u_0) = c_0$ and $(v,\phi_0(v))$ forces $(u_1,c_1)$ in $(Q_1,\Gamma)$ in any such $\Gamma$-colouring $\phi_1$.
        Since $G(Q_1)=G_1\ba u_1v$ and $G_1$ is $k$-joinable, $u_1$ is a $k$-mimic of $v$ in $Q_1$.
        Thus, there is a unique $c \in \mathbb{N}$ such that $(v,c)$ forces $(u_1,c_1)$ in $(Q_1,\Gamma)$ (since if $v$ $2$-forces $u_1$ in $(Q_1,\Gamma)$ and $(v,c)$ forces $(u_1,c_1)$ for some colour $c$, then $c=c_1$).
        Now $\phi_0(v)=c$ and $\phi_0(u_0)=c_0$ for any such $\phi_0$.
        Hence $x$ is $\Gamma_0^{(c_0,c)}$-restricted on $a$ in $Q_0^P$.
        Similarly, for any $\Gamma$-colouring $\phi_1$ of $Q_1$ such that $\phi_1(v)=c$, we have that $\phi_1(u_1) \neq c_1$.
        Thus $v$ is $\Gamma_1^{(c_1,c)}$-restricted on $c$.
        This proves one direction.

        For the other direction, note that $x$ is $\Gamma_0^{(c_0,c)}$-restricted on $a$ in $Q_0^P$ but not in $Q_0=Q_0^P\ba e_0$.
        Therefore, for any $\Gamma_0^{(c_0,c)}$-colouring $\phi_0$ of $Q_0$ such that $\phi_0(x)=a$, we have that $\phi_0(v)=c$ and $\phi_0(u_0)=c_0$.
        On the other hand, $v$ is $\Gamma_1^{(c_1,c)}$-restricted on $c$ in $Q_1^P$ but not in in $Q_1^P\ba e_1$, so, for any $\Gamma_1^{(c_1,c)}$-colouring $\phi_1$ of $Q_1$, where $\phi_1(v)=c$, we have that $\phi_1(u_1)=c_1$.
        Then $\phi_0$ does not extend to $Q$, and therefore $x$ is $\Gamma$-restricted on $a$ in $Q$, as required.
    \end{subproof}

    The next claim is for the case where $e$ is not polarised.
    \begin{claim}
        Suppose that $e$ is not polarised, and $x \in V(G_0)$.
        If $x$ is $\Gamma$-restricted on $a \in L(x)$ in $Q$, then either
        \begin{enumerate}
            \item $x$ is $\Gamma$-restricted on $a$ in $Q_0^N$ and $v$ is $\Gamma$-restricted on more than one colour in $Q_1^N$, or
            \item there exist $c,d \in \mathbb{N}$ such that
                $x$ is $\Gamma_0^{(d,c)}$-restricted on $a$ in $Q_0^P$, and
                $v$ is $\Gamma_1^{(d,c)}$-restricted on $c$ in $Q_1^P$.
        \end{enumerate}\label{+ b}
    \end{claim}
    \begin{subproof}
        Let $x$ be $\Gamma$-restricted on $a$ in $Q$.
        A $\Gamma$-colouring $\phi_0$ of $Q_0$ extends to $Q$ if and only if there is a $\Gamma$-colouring $\phi_1$ of $Q_1$ such that $\phi_1(v)=\phi_0(v)$ and $\phi_1(u_1) \neq \phi_0(u_0)$.
        For every $\Gamma$-colouring $\phi_0$ of $Q_0$ with $\phi_0(x)=a$, we have that $\phi_0$ does not extend to $Q$, so $(v,\phi_0(v))$ forces $(u_1,\phi_0(u_0))$ in $(Q_1,\Gamma)$.
        Suppose that $(x,a)$ forces $(v,c)$ in $(Q_0,\Gamma)$, for some $c \in \mathbb{N}$.
        Then there exists $d \in \mathbb{N}$ such that for any $\Gamma$-colouring $\phi_0$ of $Q_0$ with $\phi_0(x)=a$, we have $\phi_0(v)=c$, and, since $\phi_0$ does not extend to $Q$, we also have that $\phi_0(u_0)=d$, and $(v,c)$ forces $(u_1,d)$ in $(Q_1,\Gamma)$.
        It follows that $x$ is $\Gamma_0^{(d,c)}$-restricted on $a$ in $Q_0^P$ and $v$ is $\Gamma_1^{(d,c)}$-restricted on $c$ in $Q_1^P$, so (ii) holds.

        Now suppose that $(x,a)$ does not force $v$ in $(Q_0,\Gamma)$.
        Then there are $\Gamma$-colourings $\phi_0$ and $\phi_0'$ for $Q_0$ with $\phi_0(x)=\phi_0'(x)=a$ but $\phi_0(v) \neq \phi_0'(v)$, where neither $\phi_0$ nor $\phi_0'$ extends to $Q$.
        Therefore, $v$ $2$-forces $u_1$ in $(Q_1,\Gamma)$.
        Since $G_1$ is $k$-joinable and $G(Q_1)=G_1\ba e_1$, we have that $u_1$ is a $k$-mimic of $v$ in $Q_1$, so $(v,\phi_0(v))$ forces $(u_1,\phi_0(v))$ for all such $\phi_0$.
        Thus, for any such $\phi_0$, we have that $\phi_0(u_0)=\phi_0(v)$, so $x$ is $\Gamma$-restricted on $a$ in $Q_0^N$.
        Similarly, for any such $\phi_0$, we have that $v$ is $\Gamma$-restricted on $\phi_0(v)$ in $Q_1^N$.
        Since there is no $c \in \mathbb{N}$ such that $\phi_0(v)=c$ for every such $\phi_0$, we have that $v$ is $\Gamma$-restricted on more than one colour in $Q_1^N$, so (ii) holds.
    \end{subproof}

    We require one more claim.

    \begin{claim}
        If there exists $x \in V(Q)$ that is $\Gamma$-restricted on more than one colour in $Q$, then $F(Q)=\emptyset$ and $G \in \mathcal{H}_k$.\label{+ d}
    \end{claim}

    \begin{subproof}
        Suppose that $Q$ is a minimum-sized counterexample to the claim, so  $x \in V(Q)$ is $\Gamma$-restricted on $a$ and $a'$ in $Q$, for distinct $a,a' \in \mathbb{N}$, and either $F(Q) \neq \emptyset$ or $G(Q) \notin \mathcal{H}_k$.
        Note that $G(Q_0^P) = G_0$. 
        Thus, if we can show that there is a vertex that is $\Gamma_0$-restricted on more than one colour in $Q_0^P$, for some $k$-polar assignment $\Gamma_0$ of $Q_0^P$, then, as $Q_0^P$ contains at least one polarised edge, this would contradict that $Q$ is minimum sized.  We frequently use this strategy to obtain a contradiction in what follows.

        First, assume that $e$ is polarised, and suppose $R(e,u_0,u_1)=(c_0,c_1)$.
        By applying \cref{+ a} using $a$, there exists $c \in \mathbb{N}$ such that $x$ is $\Gamma_0^{(c_0,c)}$-restricted on $a$ in $Q_0^P$ and $v$ is $\Gamma_1^{(c_1,c)}$-restricted on $c$ in $Q_1^P$.
        By a second application of \cref{+ a} using $a'$, there exists $c' \in \mathbb{N}$ such that $x$ is $\Gamma_0^{(c_0,c')}$-restricted on $a'$ in $Q_0^P$ and $v$ is $\Gamma_1^{(c_1,c')}$-restricted on $c'$ in $Q_1^P$.
        Since $Q_1$ is $k$-unrestricted, we have that $(v,c)$ and $(v,c')$ force $(u_1,c_1)$ in $(Q_1^P\ba e_1,\Gamma)$.
        If $c \neq c'$, then $v$ $2$-forces $u_1$ in $(Q_1,\Gamma)$, and it follows that $u_1$ is not a $k$-mimic of $v$ in $Q_1$, so $G_1$ is not $k$-joinable, a contradiction.
        Hence $c=c'$, so $\Gamma_0^{(c_0,c)}=\Gamma_0^{(c_0,c')}$.
        Now, for any $\Gamma$-colouring $\phi$ of $Q_0^P\ba e_0$, if $\phi(x) \in \{a,a'\}$, then $\phi(v)=c$ and $\phi(u_0)=c_0$.
        Therefore $x$ is $\Gamma_0^{(c_0,c)}$-restricted on more than one colour in $Q_0^P$.
        Since $Q_0^P$ contains at least one polarised edge, we contradict that $Q$ is minimum sized.
        From this contradiction, we deduce that $e$ is not polarised.

        Now $e$ is not polarised.
        Assume that $x$ is not $\Gamma$-restricted on $a$ in $Q_0^N$.
        Then, by \ref{+ b}, there exist $c,d \in \mathbb{N}$ such that $x$ is $\Gamma_0^{(d,c)}$-restricted on $a$ in $Q_0^P$ and $v$ is $\Gamma_1^{(d,c)}$-restricted on $c$ in $Q_1^P$.
        Therefore, for every $\Gamma$-colouring $\phi$ of $Q_0$ with $\phi(x)=a$, we have that $\phi(v)=c$ and $\phi(u_0)=d$.
        If $x$ is also $\Gamma_0^{(d,c)}$-restricted on $a'$ in $Q_0^P$, then, as $Q_0^P$ contains a polarised edge, this contradicts that $Q$ is minimum sized.
        So $x$ is not $\Gamma_0^{(d,c)}$-restricted on $a'$ in $Q_0^P$.
        Therefore, there is a $\Gamma$-colouring $\phi'$ of $Q_0$ such that $\phi'(x)=a'$ but either $\phi(v) \neq c$ or $\phi(u_0) \neq d$.
        Suppose $\phi'(v)=c$. Then, since $\phi'$ does not extend to $Q$ and $e$ is not polarised, $(v,c)$ forces $(u_1,\phi'(u_0))$ in $(Q_1,\Gamma)$.
        However, since $v$ is $\Gamma_1^{(d,c)}$-restricted on $c$ in $Q_1^P$, we have that $(v,c)$ forces $(u_1,d)$ in $(Q_1,\Gamma)$.
        So $\phi'(u_0)=d$, a contradiction.
        Therefore $\phi'(v) \neq c$.
        We have that $(v,c)$ forces $(u_1,d)$ in $(Q_1,\Gamma)$, and, since $x$ is $\Gamma$-restricted on $a'$, we have that $(v,\phi'(v))$ forces $(u_1,\phi'(u_0))$ in $(Q_1,\Gamma)$, with $\phi'(v) \neq c$.
        So $v$ $2$-forces $u_1$ in $Q_1$.
        Since $G_1$ is $k$-joinable, we deduce that $c=d$.
        But then $x$ is $\Gamma_0^{(c,c)}$-restricted on $a$ in $Q_0^P$, so is also $\Gamma$-restricted on $a$ in $Q_0^N$, a contradiction.

        We deduce that $x$ is $\Gamma$-restricted on both $a$ and $a'$ in $Q_0^N$.
        Since $Q$ is a minimum-sized counterexample, $G_0 \in \mathcal{H}_k$ and $F(Q_0)=\emptyset$.
        If there exist $c,d \in \mathbb{N}$ such that, for every $\Gamma$-colouring $\phi$ of $Q_0$ with $\phi(x) \in \{a,a'\}$, we have $\phi(v)=c$ and $\phi(u_0)=d$, then $x$ is $\Gamma_0^{(d,c)}$-restricted on both $a$ and $a'$ in $Q_0^P$.
        Since $Q_0^P$ contains a polarised edge, this contradicts that $Q$ is minimum sized.
        Therefore, there exist $\Gamma$-colourings $\phi$ and $\phi'$ of $Q_0$ with $\{\phi(x),\phi'(x)\}\subseteq \{a,a'\}$ such that either $\phi(v) \neq \phi'(v)$ or $\phi(u_0) \neq \phi'(u_0)$.
        Since neither $\phi$ nor $\phi'$ is a $\Gamma$-colouring of $Q_0^N$, we have that $\phi(v)=\phi(u_0)$ and $\phi'(v)=\phi'(u_0)$, so $\phi(v) \neq \phi'(v)$.
        Since neither $\phi$ nor $\phi'$ extends to $Q$, we have that $(v,\phi(v))$ forces $(u_1,\phi(v))$ and $(v,\phi'(v))$ forces $(u_1,\phi'(v))$ in $(Q_1,\Gamma)$.
        Therefore $v$ is $\Gamma$-restricted on distinct colours $\phi(v)$ and $\phi'(v)$ in $Q_1^N$, and so, since $Q$ is minimum sized, we have that $G_1 \in \mathcal{H}_k$ and $F(Q_1)=\emptyset$.
        Now $F(Q)=\emptyset$ and, since $G$ is the Haj\'{o}s join of $G_0$ and $G_1$, with $G_0,G_1 \in \mathcal{H}_k$, we have that $G \in \mathcal{H}_k$, a contradiction.
        This completes the proof of the claim.
    \end{subproof}

    We now prove the theorem in earnest.
    For one direction, suppose that $G(Q) \in \mathcal{H}_k$, $F(Q)=\emptyset$, and $\Gamma$ is uniform.
    Then, by \cref{stthm}, $Q$ is not $\Gamma$-colourable.
    For the other direction, suppose $Q$ is not $\Gamma$-colourable.
    Then, for all $x \in V(Q)$, we have that $x$ is $\Gamma$-restricted on $c$ for each $c \in L(x)$.
    By \cref{+ d}, we have $F(Q)=\emptyset$ and $G(Q) \in \mathcal{H}_k$.
    It remains only to show that if $Q$ is not $\Gamma$-colourable, then $\Gamma$ is uniform.

    Towards a contradiction, suppose that $Q$ is not $\Gamma$-colourable, but $\Gamma$ is not uniform, where $Q$ is minimum sized.
    Since $Q_0$ is $k$-unrestricted, for each $c \in L(v)$ there exists a $\Gamma$-colouring $\phi_c$ of $Q_0$ such that $\phi_c(v)=c$, but $\phi_c$ does not extend to $Q$.
    Since $G_1$ is $k$-joinable, \cref{joinabilitysimple}(ii) holds; that is, for $Q_0' \in \{Q_0^N,Q_0^P\}$, there is a $k$-polar assignment $\Gamma_0$ for $Q_0'$ such that $\Gamma_0|Q_0=\Gamma|Q_0$ and every $\Gamma_0$-colouring of $Q_0'$ extends to a $\Gamma$-colouring of $Q$.
    Since $Q$ is not $\Gamma$-colourable, $Q_0'$ is not $\Gamma_0$-colourable.
    As $Q_0^P$ contains a polarised edge, \cref{+ d} implies that no vertex of $Q_0^P$ is $\Gamma_0$-restricted on more than one colour, and therefore $Q_0^P$ is $\Gamma_0$-colourable.
    Thus $Q'_0=Q_0^N$, and $Q_0^N$ is not $\Gamma$-colourable.
    Since $Q$ is a minimum-sized counterexample, $\Gamma|V(Q_0)$ is uniform.
    Moreover, for each $c \in L(v)$, the $\Gamma$-colouring $\phi_c$ of $Q_0$ does not extend to $Q_0^N$, so $\phi_c(u_0)=\phi_c(v)$.
    Therefore $(v,c)$ forces $(u_1,c)$ in $(Q_1,\Gamma)$ for every $c \in L(v)$.
    It follows that $Q_1^N$ is not $\Gamma$-colourable, and therefore, since $Q$ is a minimum-sized counterexample, $\Gamma|V(G_1)$ is also uniform.
    As $\Gamma|V(G_0)$ and $\Gamma|V(G_1)$ are both uniform and $v \in V(G_0) \cap V(G_1)$, we have that $\Gamma$ is uniform, a contradiction.  We deduce that if $Q$ is not $\Gamma$-colourable, then $\Gamma$ is uniform.
\end{proof}

\section{Proofs of \texorpdfstring{\cref{mainthm,restrictedthm}}{Theorems 1.3 and 1.5}} 

We are now ready to prove the Brooks-type theorem for graphs with maximum local edge-connectivity~$k$ (\cref{mainthm}).  We first prove \cref{restrictedthm}.

\begin{lemma}
\label{zero or one higher degree}
Let $Q$ be a $2$-connected polar graph with at most one vertex of degree more than $k$, for $k \geq 3$. If $Q$ is $k$-restricted, then $G(Q) \in \mathcal{T}_k$.
\end{lemma}
\begin{proof}
    Assume that $Q$ is $k$-restricted. 
    Let $r$ be a vertex of $Q$ and let $\Gamma=(L,R)$ be a $k$-polar assignment for $Q$ such that $r$ is $\Gamma$-restricted on $c$ in $Q$ for some $c \in L(r)$.
    First assume that $Q$ has no vertices with degree more than $k$.
    Then $k=|L(u)| \geq d_Q(u)$ for all $u \in V(Q)$; so $\Gamma$ is a degree-polar assignment.
    Let $(Q-r,\Gamma_c)$ be the $(r,c)$-colour deletion from $(Q,\Gamma)$.
    Then $\Gamma_c$ is also a degree-polar assignment and so, by \cref{dcpolarthm}, $Q-r$ is a polar Gallai tree and $\Gamma_c$ is bad.
    Since $\Gamma_c$ is bad, $|L_c(u)|=d_{Q-r}(u)$ for all $u \in V(Q-r)$.
    However, $|L_c(u)| \geq k-e_Q(u,r) \geq d_Q(u)-e_Q(u,r) = d_{Q-r}(u)$, so equality holds throughout.
    Thus $d_Q(u)=k$ for all $u \in V(Q-r)$.
    Therefore, by \cref{Hajos joins give k2}, $G(Q)$ is in $\mathcal{T}_k$.

    Now we may assume that $Q$ contains some vertex $h$ with $d_Q(h) > k$.
    Then, by \cref{krtcorollary}, $G(Q-h)$ is a Gallai tree, and $d_Q(u)=k$ for all $u \in V(Q-h)$.
    By \cref{Hajos joins give k2}, $G(Q) \in \mathcal{T}_k$, as required.
\end{proof}

We now prove the first main result of this section, which implies \cref{restrictedthm}.

\begin{theorem}
\label{final classify theorem}
Let $Q$ be a $2$-connected polar graph with maximum local edge-connectivity at most $k$, for $k \geq 3$. If $Q$ is $k$-restricted, then $G(Q) \in \mathcal{T}_k$.
\end{theorem}
\begin{proof}
    Towards a contradiction, suppose that $Q$ is $k$-restricted but $G(Q) \notin \mathcal{T}_k$.  We further assume that, amongst all such graphs, $Q$ has the minimum number of vertices.
    By \cref{zero or one higher degree}, $Q$ contains distinct vertices $h_x$ and $h_y$ with degree more than $k$.
    Let $S$ be an edge cut of minimum size such that $S$ separates $X$ from $Y$, where $(X,Y)$ is a partition of $V(G)$ with $h_x \in X$ and $h_y \in Y$.
    For $e \in S$, we let $x_e$ (and $y_e$) denote the vertex incident to $e$ in $X$ (and $Y$, respectively).
    Since $Q$ has maximum local edge-connectivity~$k$, Menger's Theorem implies that $|S| \leq k$.
    Note also that this implies that $|Y| > 1$, since $h_y \in Y$ and $d(h_y) > k$.

    \begin{claim}
        \label{claimone}
        $Q[Y]$ is $k$-unrestricted.
    \end{claim}
    \begin{subproof}
        Towards a contradiction, suppose that $Q[Y]$ is $k$-restricted.
        Then there exists $r \in Y$ and a $k$-polar assignment $\Gamma=(L,R)$ of $Q[Y]$ such that $r$ is $\Gamma$-restricted on $c$ for some $c \in L(r)$.
        By \cref{1joinscu}, there is a $2$-connected $k$-restricted polar subgraph $Q_Y$ of $Q[Y]$.
        Note that $Q_Y$ has maximum local edge-connectivity at most $k$, since $G(Q_Y)$ is a subgraph of $G(Q)$.
        Therefore, by the minimality of $Q$, we have $G(Q_Y) \in \mathcal{T}_k$, but this contradicts \cref{can't have both corollary}.
        We deduce that $Q[Y]$ is $k$-unrestricted, as required.
    \end{subproof}

    Since $Q$ is $k$-restricted there exists $r \in V(Q)$ and a $k$-polar assignment $\Gamma=(L,R)$ for $Q$ such that $r$ is $\Gamma$-restricted on $c$ for some $c \in L(r)$.
    Without loss of generality assume that $r \in Y$.

    Let $Q'_X$ be the polar graph obtained from $Q[X]$ by adding a vertex $w$, and, for each edge $e \in S$, 
    we add a polarised edge $e$ in $Q'_X$ between $x_e$ and $w$.  In this way, $S$ is an edge cut in $Q$, and also the set of edges incident with $w$ in $Q'_X$.
    Note that, letting $G=G(Q)$, we have that $G(Q'_X)=G'_X$ as described just prior to \cref{the hajos decomp lemma}.

    \begin{claim}
        \label{QXcolour}
        The vertex $w$ is $k$-restricted in $Q'_X$.
        Moreover, for any $\Gamma$-colouring $\phi$ of $Q[Y]$ such that $\phi(r)=c$, there exists a $k$-polar assignment $\Gamma'_X$ for $Q'_X$ such that the $(w,d)$-colour deletion from $(Q'_X,\Gamma'_X)$ equals the $(Y,\phi)$-colour deletion from $(Q,\Gamma)$.
    \end{claim}
    \begin{subproof}
        Let $\phi$ be a $\Gamma$-colouring of $Q[Y]$ such that $\phi(r)=c$; such a colouring exists by \cref{claimone}.
        Let $(Q[X],\Gamma_X)$ be the $(Y,\phi)$-colour deletion from $(Q,\Gamma)$.
        Since $r$ is $\Gamma$-restricted on $c$ in $Q$, we have that $Q_X$ is not $\Gamma_X$-colourable.
        From this, we will construct a $k$-polar assignment $\Gamma'_X$ for $Q'_X$ such that $w$ is $\Gamma'_X$-restricted in $Q'_X$.
        For each $e \in S$, 
        if $e$ is not polarised in $Q$, then let $c_e=\phi(y_e)$; if $e$ is polarised in $Q$ with $R(e,x_e,y_e)=(c_x,\phi(y_e))$ for some $c_x$, then let $c_e=c_x$; otherwise let $c_e$ be any colour not in $L(x_e)$.
        Now let $\Gamma'_X=(L'_X,R'_X)$ be a $k$-polar assignment for $Q'_X$ such that $\Gamma'_X|X=\Gamma|X$ and $L'_X(w)$ consists of $k$ colours one of which is $d$, and, for each polarised edge $e=x_ew$, let $R'_X(e,x_e,w)=(c_e,d)$.
        Notice that, by construction, the $(w,d)$-colour deletion from $(Q'_X,\Gamma'_X)$ is $(Q[X],\Gamma_X)$.
        Since $Q_X$ is not $\Gamma_X$-colourable, $w$ is $\Gamma'_X$-restricted on $d$ in $Q'_X$, as required.
    \end{subproof}

    \begin{claim}
        \label{QX plays nice claim}
        The polar graph $Q'_X$ is $2$-connected and has maximum local edge-connectivity at most $k$.
    \end{claim}
    \begin{subproof}
        Since $S$ is a minimum-sized edge cut in $Q$, the graph $Q[X]$ is connected and, therefore, $Q'_X$ is connected.
        Since $|S| \leq k$ and $d_{Q'_X}(h_x) > k$, we have $|V(Q'_X)| \ge 3$. 
        To show that $Q'_X$ is $2$-connected, it now suffices to show that $Q'_X$ does not contain a cut vertex.
        Certainly $w$ is not a cut vertex, since $Q'_X-w=Q[X]$.
        Moreover, if $u \in V(Q'_X) \setminus \{w\}$ is a cut vertex, say, then $\{u\}$ separates $w$ from some vertex in $x \in V(Q'_X) \setminus \{u,w\}$, in which case it follows that $\{u\}$ separates $x$ from $h_y$ in $Q$, contradicting that $Q$ is $2$-connected.
        So $Q'_X$ is $2$-connected.

        To see that $Q'_X$ has maximum local edge-connectivity at most~$k$, first observe that there are at most $k$ edge-disjoint $(w,v)$-paths in $Q'_X$ for any $v \in V(Q'_X) \setminus \{w\}$, since $d(w)=|S| \le k$.
        Then, observe that there are $|S|$ edge-disjoint $(X_S,h_y)$-paths in $Q$, and these are in bijection with the $|S|$ edges incident with $w$ in $Q'_X$.
        For any $u,v \in X$, a $(u,v)$-path in $Q'_X$ has a corresponding path in $Q$ that we obtain by replacing any edge of the form $u_{e}w$ in $Q'_X$ with the corresponding $(x_e,h_y)$-path in $Q$.
        Since $Q$ has maximum local edge-connectivity at most $k$, it follows that there are at most $k$ $(u,v)$-paths in $Q'_X$, so indeed $Q'_X$ has maximum local edge-connectivity at most~$k$.
    \end{subproof}

    Now $|V(Q'_X)| < |V(Q)|$, since $|Y| \ge 2$, and so, by \cref{QXcolour,QX plays nice claim} and the minimality of $Q$, we have that $G(Q'_X) \in \mathcal{T}_k$.
    From \cref{doi-new}(i), we have $d_{Q'_X}(u) \geq k$ for all $u \in V(Q'_X)$.
    In particular, $|S| = d_{Q'_X}(w) \geq k$, but since $|S| \leq k$, we have $|S|=d_{Q'_X}(w)=k$.

    \begin{claim}
        \label{intkminus}
        $G(Q'_X) \in \mathcal{T}_k^-$.
    \end{claim}
    \begin{subproof}
        Towards a contradiction, assume that $G(Q'_X) \notin \mathcal{T}_k^-$.
        Then $G(Q'_X)$ is a Haj\'{o}s join of two graphs $G_0,G_1 \in \mathcal{T}_k$.
        Let $G(Q'_X) = (G_0,z,u_0) \hajos (G_1,z,u_1)$ such that, for $i \in \{0,1\}$, we have $z,u_i \in V(G_i)$, and let $e=u_0u_1 \in E(Q'_X)$.
        If $z=w$, then as $d_{Q'_X}(w) = k$, this contradicts \cref{doi-new}(ii).
        So $z \neq w$.
        Without loss of generality, $w \in V(G_0) \setminus \{z\}$.

        Let $Z_0 = V(Q) \setminus V(G_1)$ and $Z_1 = V(G_1)\setminus \{z\}$, and observe that $(Z_0,\{z\},Z_1)$ is a partition of $V(Q)$ with $Y \subseteq Z_0$ and $u_1 \in Z_1$.
        Let $v_0 \in Z_0$ and $v_1 \in Z_1$.
        We claim that every $(v_0,v_1)$-path in $Q$ contains either $z$ or $e$.
        Let $P$ be a $(v_1,v_0)$-path in $Q$.
        First suppose $P$ contains no edges from $S$.
        Then, since $v_1 \in V(G_1) \subseteq X$, we have that $P$ is also a path in $Q'_X$.
        Since $v_0 \notin V(G_1)$ and $v_1 \in V(G_1) \setminus \{z\}$, this implies that $P$ contains either $e$ or $z$, as required.
        Now suppose $P$ contains at least one edge in $S$.
        Let $y$ be the first vertex in $P$ that is not in $X$, and let $P'$ be the $(v_1,y)$-subpath of $P$.
        Then every vertex of $P'$ except the final vertex $y$ is in $X$, and so by replacing $y$ with $w$, we see that $P'$ corresponds to a $(v_1,w)$-path in $Q'_X$.  Now $w \in V(G_0) \setminus \{z\}$ and $v_1 \in V(G_1) \setminus \{z\}$, so $P'$, and hence $P$, contains either $e$ or $z$, as required.

        Now, by the previous paragraph and since $Q$ is $2$-connected, it follows that $G(Q)$ can be obtained as the Haj\'os join of two graphs.
        We have $e=u_0u_1$ with $u_0 \in Z_0$.
        Let $G_0'$ be the graph obtained from $G(Q[Z_0 \cup \{z\}])$ by adding an edge between $u_0$ and $z$.
        Then $G(Q)$ is the Haj\'os join $(G_0',z,u_0) \hajos (G_1,z,u_1)$.
        As $Q$ is $2$-connected and has maximum local edge-connectivity at most $k$, we have, by \cref{hajosconn0,hajosconn2}, that both $G_1$ and $G_0'$ are $2$-connected and have maximum local edge-connectivity at most $k$.
        As $G(Q)$ is not $k$-polar-unrestricted, at least one of $G_0'$ and $G_1$ is not $k$-polar-unrestricted.
        As $|V(G_0')|,|V(G_1)| < |V(Q)|$, there exists a $2$-connected polar graph $Q'$ with maximum local edge-connectivity at most $k$ such that $|V(Q')| < |V(Q)|$, and therefore, by the minimality of $Q$, we have $G(Q') \in \mathcal{T}_k$.
        By \cref{the joinability corollary}, $G(Q')$ is $k$-joinable.
        Now, by \cref{joinable plus unrestricted}, we deduce both $G_0'$ and $G_1$ are not $k$-polar-unrestricted.
        It follows, by the minimality of $Q$, that both $G_0'$ and $G_1$ are in $\mathcal{T}_k$.
        Therefore $G(Q) \in \mathcal{T}_k$, a contradiction.
        We deduce that $G(Q'_X) \in \mathcal{T}_k^-$, as required.
    \end{subproof}

    Observe that, by the construction of $Q'_X$, we have $d_{Q'_X}(u)=d_Q(u)$ for all $u \in X$.
    In particular, $d_{Q'_X}(h_x)=d_Q(h_x) >k$.
    Now $Q'_X$ is a polar graph such that $G(Q'_X) \in \mathcal{T}_k^-$, by \cref{intkminus}, and $Q'_X$ contains a vertex $h_x$ of degree greater than $k$, and a vertex $w$ that is $k$-restricted, where each of the $k$ edges incident with $w$ is polarised. 
    By \cref{tkminusone}\ref{tk a}, $w$ is central.
    If $Q'_X$ is not a polar odd wheel, then, by \cref{no parallel edges lemma}(ii), $w$ is incident to a pair of parallel edges, both of which are polarised, implying that $w$ is not central.
    So $Q'_X$ is a polar odd wheel and $k=3$.
    Since $d_{Q'_X}(h_x) > 3$, we have $|V(Q'_X)| > 4$ and $h_x$ is the hub of $Q'_X$.
    Hence $|S|=3$ and $|X| \ge 4$.
    Let $X_S$ (or $Y_S$) be the set of vertices of $X$ (or $Y$, respectively) that are incident with edges of $S$ in $Q$.
    Then $|X_S|=3$ and, as $h_x$ is adjacent to $w$ in $Q'_X$, we have $h_x \in X_S$.

    Now we know the structure of $Q'_X$, we consider colourings of $Q[Y]$ that do not extend to $Q$.
    Let $\Phi$ be the set of $\Gamma$-colourings $\phi$ of $Q[Y]$ such that $\phi(r)=c$.
    By \cref{claimone}, $\Phi$ is non-empty.
    Since $r$ is $\Gamma$-restricted on $c$ in $Q$, we have that $\phi$ does not extend to $Q$ for each $\phi \in \Phi$.

    \begin{claim}
        \label{claimextra}
        Let $\phi \in \Phi$ and $e \in S$.
        \begin{itemize}
            \item If $e$ is polarised in $Q$, then $R(e,x_e,y_e)=(c,\phi(y_e))$ for some $c \in L(x_e)$.
            \item If $e$ is not polarised in $Q$, then $\phi(y_e) \in L(x_e)$.
        \end{itemize}
    \end{claim}
    \begin{subproof}
        Let $(Q[X],\Gamma_X)$ be the $(Y,\phi)$-colour deletion from $(Q,\Gamma)$, with $\Gamma_X=(L_X,R_X)$.
        Since $r$ is $\Gamma$-restricted on $c$ in $Q$, we have that $Q[X]$ is not $\Gamma_X$-colourable.
        By \cref{QXcolour}, there exists a $3$-polar assignment $\Gamma'_X=(L_X,R_X)$ for $Q'_X$ such that the $(w,d)$-colour deletion from $(Q'_X,\Gamma'_X)$ is $(Q[X],\Gamma_X)$.
        Then, by \cref{odd wheels in G2}, $|L_X(x_e)| = 2$ for each $e \in S$.
        Since $|L(x_e)|=3$ for each $e \in S$, the claim follows.
    \end{subproof}

    \begin{claim}
        \label{Y claim}
        Either
        \begin{enumerate}
            \item for any $\phi_0,\phi_1 \in \Phi$, we have $\phi_0|_{Y_S}=\phi_1|_{Y_S}$; or
            \item for any $\phi \in \Phi$, there is some $d_\phi \in \mathbb{N}$ such that $\phi(y)=d_\phi$ for all $y \in Y_S$.
        \end{enumerate}
%
%
    \end{claim}
    \begin{subproof}
        Suppose that there exist $\phi_0,\phi_1 \in \Phi$ such that $\phi_0|_{Y_S} \neq \phi_1|_{Y_S}$.
        Let $(Q[X],\Gamma_i)$ be the $(Y,\phi_i)$-colour deletion from $(Q,\Gamma)$, for $i \in \{0,1\}$.
        Then $Q[X]$ is not $\Gamma_0$-colourable, nor $\Gamma_1$-colourable.
        By \cref{QXcolour}, for each $i \in \{0,1\}$ there exists a $3$-polar assignment $\Gamma'_i$ for $Q'_X$ such that the $(w,d)$-colour deletion from $(Q'_X,\Gamma'_i)$ is $(Q[X],\Gamma_i)$.
        Let $\Gamma_i = (L_i,R_i)$ for $i \in \{0,1\}$.
        For $x \in X_S$, we have $|L(x)| = 3$ since $\Gamma$ is a $3$-polar assignment.
        For $e \in S$, we have $|L_0(x_e)| = |L_1(x_e)| = 2$, by \cref{claimextra}. 
        Thus, for $e \in S$, we have $L_0(x_e) = L_1(x_e)$ if and only if $\phi_0(y_e) = \phi_1(y_e)$.
        Since $\phi_0|_{Y_S} \neq \phi_1|_{Y_S}$, there exists some $e' \in S$ such that $\phi_0(y_{e'}) \neq \phi_1(y_{e'})$ and $L_0(x_{e'}) \neq L_1(x_{e'})$.
        Now $e'$ is not polarised in $Q$, as otherwise $R(e',x_{e'},y_{e'})=(c_x,c_y)$ where $c_y \neq \phi_i(y_{e'})$ for some $i \in \{0,1\}$, violating \cref{claimextra}.
        By \cref{odd wheels in G2}, $L_i|_{X_S}$ is uniform for $i \in \{0,1\}$.
        In particular, for each $e \in S$, we have $\phi_0(y_{e}) \neq \phi_1(y_{e})$ and $L_0(x_{e}) \neq L_1(x_{e})$.
        Now, similarly to before, each $e \in S$ is not polarised in $Q$.
        It then follows that $\phi_i(y) = \phi_i(y_{e'})$ for all $y \in Y_S$ and $i \in \{0,1\}$.
        This shows that if there exist $\phi_0,\phi_1 \in \Phi$ such that $\phi_0|_{Y_S} \neq \phi_1|_{Y_S}$, then there exist $d_0,d_1\in \mathbb{N}$ such that $\phi_i(y)=d_i$ for all $y \in Y_S$ and $i \in \{0,1\}$.

        Assume now that (i) does not hold.  Then there exist $\phi_0,\phi_1 \in \Phi$ such that $\phi_0|_{Y_S} \neq \phi_1|_{Y_S}$, and, by the previous paragraph, there exist $d_0,d_1\in \mathbb{N}$ such that $\phi_i(y)=d_i$ for all $y \in Y_S$.
        If $|\Phi|=2$, then (ii) holds.
        Otherwise, for any $\phi_2 \in \Phi \setminus \{\phi_0,\phi_1\}$, we have either $\phi_2|_{Y_S}=\phi_0|_{Y_S}$, in which case $\phi_2(y)=d_0$ for all $y \in Y_S$, or $\phi_2|_{Y_S} \neq \phi_0|_{Y_S}$, in which case there exists $d_2\in \mathbb{N}$ such that $\phi_i(y)=d_2$ for all $y \in Y_S$, by the previous paragraph.
    \end{subproof}

    Let $Q_Y$ be the polar graph obtained from $Q[Y]$ as follows:
    for each $e \in S$, we add a new vertex $u_e$, and a new non-polarised edge between $u_e$ and $y_e$, and then add edges so that $\{u_e : e \in S\}$ is a clique where all edges are not polarised.
    Since $Q$ is $2$-connected with maximum local edge-connectivity at most $k$,
    it is easily seen that $Q_Y$ is $2$-connected and has maximum local edge-connectivity at most $k$.
    Note that, letting $G=G(Q)$, we have that $G(Q_Y)=G_Y$ as described just prior to \cref{the hajos decomp lemma}.
    Let $U = \{u_e : e \in S\}$.
    Recall that $|S|=k$, so $|U|=k$ (where $k=3$).

    First assume that \cref{Y claim}(i) holds. 
    Let $\phi \in \Phi$ and let $C \subseteq \mathbb{N}$ such that $|C|=2$ and $C$ is disjoint from $\{\phi(y_e) : e \in S\}$.
    Let $\Gamma_Y=(L_Y,R_Y)$ be the $3$-polar assignment for $Q_Y$ such that $\Gamma_Y|Y=\Gamma|Y$ and $L_Y(u_e)=C \cup \{\phi(y_e)\}$ where $e \in S$.
    Now let $(Q_Y[U],\Gamma_U)$ be the $(Y,\phi)$-colour deletion from $(Q_Y,\Gamma_Y)$, and let $\Gamma_U=(L_U,R_U)$.
    Then $G(Q_Y[U]) \cong K_3$ and $Q_Y[U]$ has no polarised edges, where $L_U(u)=C$ for all $u \in U$ by \cref{claimextra}, so $Q_Y[U]$ is not $\Gamma_U$-colourable.
    Therefore, each colouring in $\Phi$ does not extend to $Q_Y$, so $r$ is $\Gamma_Y$-restricted on $c$ in $Q_Y$.
    Recall that $|X| > |S|$, so, by the minimality of $Q$, we have that $G(Q_Y) \in \mathcal{T}_k$.

    Now we may assume that \cref{Y claim}(ii) holds, so for any $\phi \in \Phi$, there is some $d_\phi \in \mathbb{N}$ such that $\phi(y)=d_\phi$ for all $y \in Y_S$.
    Let $D = \{d_\phi : \phi \in \Phi\}$, and observe that $|D| \leq 3$, since $|L(y)|=3$ for any $y \in Y_S$.
    Let $C$ be a set of $3$ colours that contains $D$.
    Let $\Gamma_Y=(L_Y,R_Y)$ be the $3$-polar assignment for $Q_Y$ such that $\Gamma_Y|Y=\Gamma|Y$ and $L_Y(u_e)=C$ for each $e \in S$.
    Now let $(Q_Y[U],\Gamma_U)$ be the $(Y,\phi)$-colour deletion from $(Q_Y,\Gamma_Y)$ for some $\phi \in \Phi$, with $\Gamma_U=(L_U,R_U)$.
    Then $G(Q_Y[U]) \cong K_3$ and $Q_Y[U]$ has no polarised edges, and $L_U(u)=C \setminus \{d_\phi\}$ for all $u \in U$, by \cref{claimextra}.
    That is, $\Gamma_U$ is a uniform $2$-polar assignment, therefore $Q_Y[U]$ is not $\Gamma_U$-colourable.
    Since this holds for any $\phi$ in $\Phi$, we once again see that $r$ is $\Gamma_Y$-restricted on $c$ in $Q_Y$, and, again by the minimality of $Q$, we have that $G(Q_Y)$ is in $\mathcal{T}_k$.

    Therefore both $G(Q'_X)$ and $G(Q_Y)$ are in $\mathcal{T}_k$ and so, by \cref{the hajos decomp lemma}, $G(Q)$ is also in $\mathcal{T}_k$, a contradiction. This completes the proof.
\end{proof}

Building on \cref{final classify theorem}, one can obtain a polynomial-time algorithm for determining the $k$-restriction index of a vertex in a $2$-connected graph with maximum local edge-connectivity $k$, see \cite{Bastida2025}.

Finally, we prove the following, that clearly implies \cref{mainthm}.

\begin{theorem}
    \label{mainthmstrong}
    Let $G$ be a $2$-connected graph with maximum local edge-connectivity~$k$, for $k\ge 3$, and let $L$ be a $k$-list assignment for $G$. Then $G$ is not $L$-colourable if and only if $G \in \mathcal{H}_k$ and $L$ is uniform.
\end{theorem}
\begin{proof}
    Suppose that $G$ is not $L$-colourable.  Then $G$ is not $k$-choosable, so it is $k$-restricted, and therefore, by \cref{final classify theorem}, $G \in \mathcal{T}_k$.
    Then, by \cref{tkes}, $G \in \mathcal{H}_k$ and $L$ is uniform, as required.
    This proves one direction.
    The other direction follows from \cref{stthm}.
\end{proof}

\section{Proof of \texorpdfstring{\cref{hardnessthm}}{Theorem 1.4}}

In this section we prove \cref{hardnessthm}.
%
Let $k$ be a fixed positive integer.
We define the following problem that generalises \textsc{$k$-choosability}.

\noindent
\fbox{\begin{minipage}{.98\linewidth}
\textsc{$f$-choosability bounded by $k$}\\\nopagebreak
\textsc{Instance}: A graph $G$ and a function $f:V(G) \rightarrow \mathbb{N}$ where $0\leq f(v)\leq k$ for all $v \in V(G)$.\\
\textsc{Question}: Is $G$ $f$-choosable?
\end{minipage}} \\

Erd\H os, Rubin, and Taylor~\cite{ERT1979} showed that \textsc{$f$-choosability bounded by $3$} is $\Pi_2$-complete when restricted to bipartite graphs.
They achieve this by describing a reduction from the problem \textsc{$\forall \exists$3-Sat} to the problem \textsc{$f$-choosability bounded by $3$}.
On close inspection, it is easily seen that the graphs obtained in this reduction are not only bipartite, but also have maximum local edge-connectivity at most~$3$.
Thus, we have the following:

\begin{proposition}
    \label{ertcorollary}
    For each integer $k \geq 3$, \textsc{$f$-choosability bounded by $k$}, restricted to graphs with maximum local edge-connectivity at most $3$, is $\Pi_2$-complete.
\end{proposition}

We will prove \cref{hardnessthm} by reducing from \textsc{$f$-choosability bounded by $k$} to \textsc{$k$-choosability}.
The reduction makes use of top-ups (see \cref{topupdef}).  We first describe a graph, $T_k$, with a $k$-restricted vertex, that we will use in the top-up.


For each integer $k \ge 3$, we define the graph $T_k$ as follows.  (Note that we have already seen $T_3$, as illustrated in \cref{figassignunion}; its construction is illustrated in \cref{constructT3}.)
Starting with a copy of $I_k$ (the graph consisting of two vertices with $k$ edges between them) on vertex set $\{u,v\}$, 
we perform $k-1$ Haj\'{o}s joins, in turn, with a copy of $K_{k+1}$, on $(u,v)$ and an arbitrary pair of vertices in the copy of $K_{k+1}$.
Note that $T_k \in \mathcal{T}_k$.
It is not too difficult to show that $m_{T_k,k}(u)=1$. 

\begin{figure}[h]
    \centering
    \includegraphics[scale=0.5]{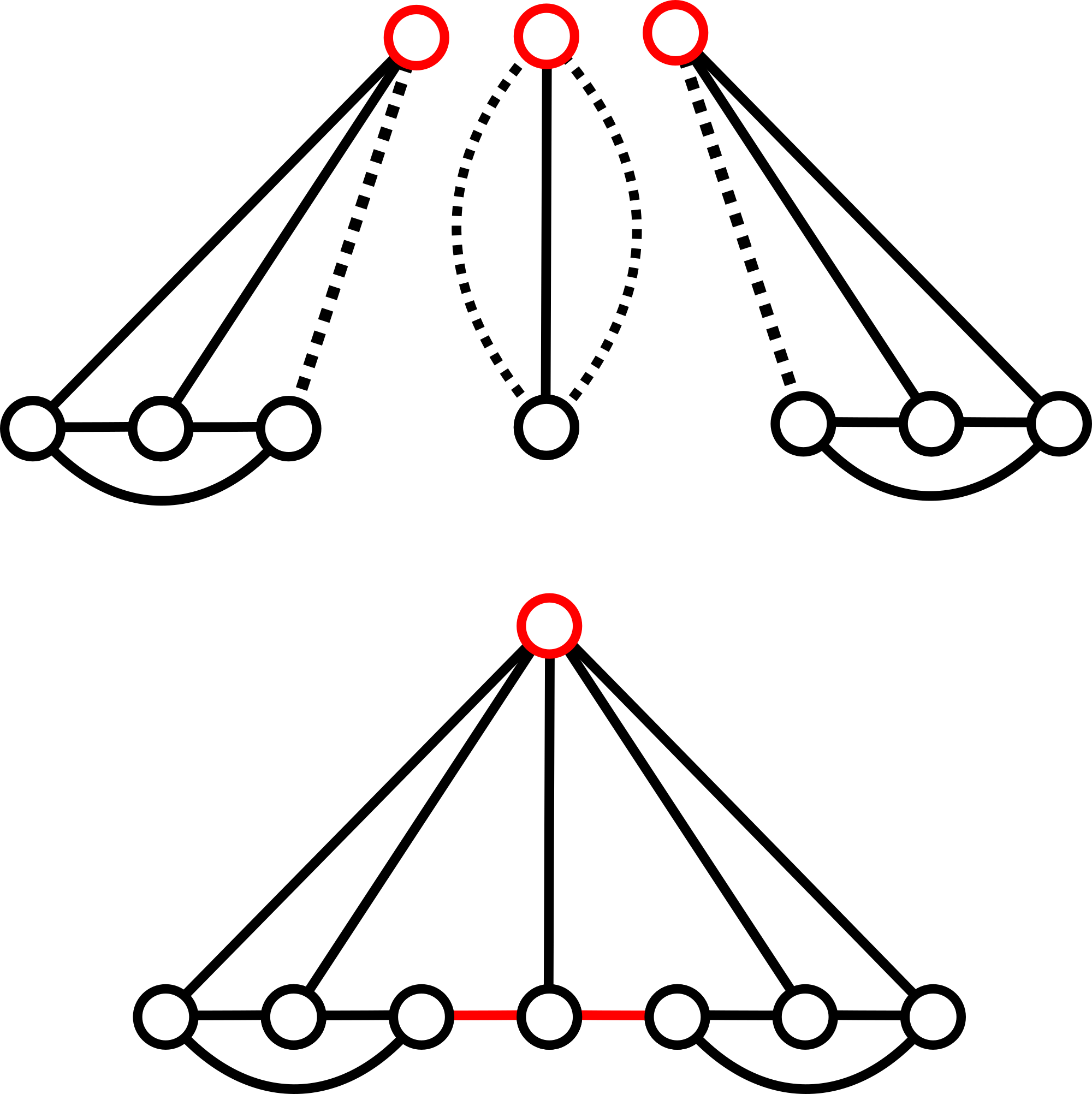}
    \caption{The construction of $T_3$ from three graphs in $\mathcal{T}_3^-$, using two Haj\'{o}s joins.} 
    \label{constructT3}
\end{figure}

We can now prove \cref{hardnessthm}, which we restate below.

\begin{reptheorem}{hardnessthm}
    For each integer $k \ge 3$, the \textsc{$k$-choosability} problem, when restricted to graphs with maximum local edge-connectivity~$k$, is $\Pi_2$-complete.
\end{reptheorem}
\begin{proof}
    We reduce an instance $(G,f)$ of the problem \textsc{$f$-choosability bounded by $k$}, for $G \in \mathcal{G}$, to an instance $(G_+)$ of the problem \textsc{$k$-choosability}, where
    $G_+$ is the $(G,f,T_k,u,k)$-top-up.
    Note that we can compute $G_+$ in polynomial time.
    Since $G$ and $T_k$ have maximum local edge-connectivity at most $k$, the graph $G_+$ also has maximum local edge-connectivity at most $k$. In fact, by \cref{connectivity properties corollary}, $T_k$ is $k$-edge-connected, and there is at least one vertex $v$ of $G$ with $f(v) < k$, so the graph $G_+$ has maximum local edge-connectivity $k$.
    Now, by \cref{topuplemma}, $G$ is $f$-choosable if and only if $G_+$ is $k$-choosable, as required.
    The result then follows, using \cref{ertcorollary}.
\end{proof}

\bibliographystyle{abbrv}
\bibliography{lib}

\appendix

\section{NP-completeness of $k$-colourability for graphs with maximum local connectivity~$k$}
\label{appendix}

For $k=3$, it was shown in \cite{ABHMT2017} that deciding if a graph with maximum local connectivity at most~$3$ is $3$-colourable is NP-complete.
Here we generalise this as follows:
\begin{proposition}
  \label{mlcnpc3}
  Let $k \ge 3$.
  The problem of deciding if a graph with maximum local connectivity at most~$k$ is $k$-colourable is NP-complete.
\end{proposition}


We prove this by generalising the approach used in \cite[Proposition~4.2]{ABHMT2017}.
We reduce from the unrestricted version of \textsc{$k$-colourability}:
given an instance of this problem,
we replace each vertex of degree at least $k+1$ with a gadget that ensures that the resulting graph has maximum local connectivity at most~$k$.
Shortly, we describe this gadget; first, we require some definitions.

For an integer $n \ge 4$, let $K_n^-$ be the graph obtained from $K_n$ be removing a single edge.  The graph $K_n^-$ has $n-2$ vertices of degree~$n-1$ and two vertices with degree~$n-2$; we call the two degree-$(n-2)$ vertices of $K_n^-$ the \emph{ends}.
A tree is \emph{cubic} if all vertices have either degree one or degree three.  A degree-$1$ vertex is a \emph{leaf}; and an edge that is incident to a leaf is a \emph{pendant} edge, whereas an edge that is incident to two degree-$3$ vertices is an \emph{internal} edge.

For $\ell \geq 4$, let $T$ be a cubic tree with $\ell$ leaves. 
We first subdivide each internal edge of $T$ once, obtaining the tree~$T'$.
Now, for each edge $xy$ of $T'$, we remove $xy$, take a copy of $K_{k+1}^-$ and identify, firstly, the vertex $x$ with one end of $K_{k+1}^-$, and, secondly, $y$ with the other end of $K_{k+1}^-$.
A degree-$(k-1)$ vertex in the resulting graph $T''$ corresponds to a leaf of $T$; we call such a vertex an \emph{outlet}.
We also call $T''$ a \emph{hub gadget} with $\ell$ outlets. 
Observe that for any integer $\ell \ge k+1$, there exists a hub gadget with exactly $\ell$ outlets. 
An example of a hub gadget with five outlets, when $k=4$, is shown in \cref{outlets-fig}.

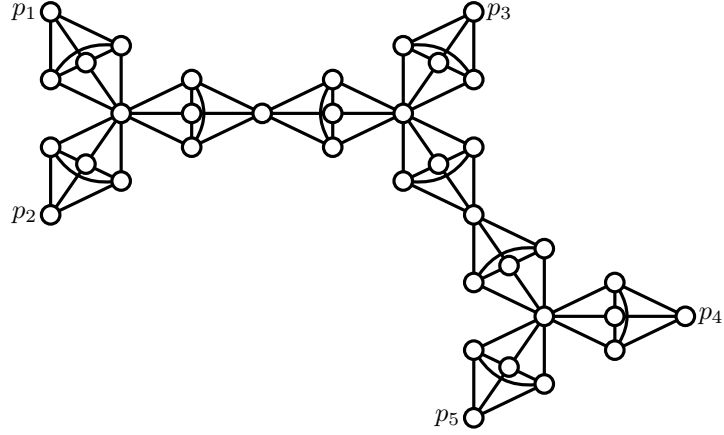
\begin{figure}
  \centering
  \begin{tikzpicture}[scale=1.5]
    \tikzset{VertexStyle/.append style = {draw,minimum height=7,minimum width=7}}
    \tikzset{EdgeStyle/.append style = {very thick}}
    \Vertex[x=-1.875,y=0.9,LabelOut=true,L=$p_1$,Lpos=180]{a1}
    \Vertex[x=-1.875,y=-0.9,LabelOut=true,L=$p_2$,Lpos=180]{a2}
    \Vertex[x=1.875,y=0.9,LabelOut=true,L=$p_3$]{b3}
    \Vertex[x=3.75,y=-1.8,LabelOut=true,L=$p_4$]{b10}
    \Vertex[x=1.875,y=-2.7,LabelOut=true,L=$p_5$,Lpos=180]{b11}
    \SetVertexNoLabel
    \Vertex[x=1.875,y=-0.9,LabelOut=true,L=$p_4$]{b4}
    \Vertex[x=-1.25,y=0]{a}
    \Vertex[x=-1.25,y=0.6]{aa1}
    \Vertex[x=-1.5625,y=0.45]{aa15}
    \Vertex[x=-1.875,y=0.3]{aa2}

    \Vertex[x=-1.25,y=-0.6]{aa5}
    \Vertex[x=-1.5625,y=-0.45]{aa55}
    \Vertex[x=-1.875,y=-0.3]{aa6}

    \Vertex[x=1.25,y=0.6]{bb1}
    \Vertex[x=1.5625,y=0.45]{bb15}
    \Vertex[x=1.875,y=0.3]{bb2}

    \Vertex[x=1.25,y=-0.6]{bb5}
    \Vertex[x=1.5625,y=-0.45]{bb55}
    \Vertex[x=1.875,y=-0.3]{bb6}

    \Vertex[x=-0.625,y=0.3]{dab1}
    \Vertex[x=-0.625,y=0]{dab15}
    \Vertex[x=-0.625,y=-0.3]{dab2}
    \Vertex[x=0.625,y=0.3]{dab3}
    \Vertex[x=0.625,y=0]{dab35}
    \Vertex[x=0.625,y=-0.3]{dab4}
    \Vertex[x=0,y=0]{ab}
    \Vertex[x=1.25,y=0]{b}

    \Vertex[x=1.875,y=-1.5]{bb7}
    \Vertex[x=2.1875,y=-1.35]{bb75}
    \Vertex[x=2.5,y=-1.2]{bb8}
    \Vertex[x=2.5,y=-1.8]{b9}

    \Vertex[x=3.125,y=-1.5]{dab5}
    \Vertex[x=3.125,y=-1.8]{dab55}
    \Vertex[x=3.125,y=-2.1]{dab6}

    \Vertex[x=2.5,y=-2.4]{aa7}
    \Vertex[x=2.1875,y=-2.25]{aa75}
    \Vertex[x=1.875,y=-2.1]{aa8}

    \Edge(a)(aa1)
    \Edge(a)(aa15)
    \Edge(a)(aa2)
    \Edge(aa1)(aa15)
    \Edge(aa15)(aa2)
    \Edge(a1)(aa2)
    \Edge(a1)(aa15)
    \Edge(a1)(aa1)

    \Edge(a)(aa5)
    \Edge(a)(aa55)
    \Edge(a)(aa6)
    \Edge(aa5)(aa55)
    \Edge(aa55)(aa6)
    \Edge(a2)(aa6)
    \Edge(a2)(aa55)
    \Edge(a2)(aa5)

    \Edge(a)(dab1)
    \Edge(a)(dab15)
    \Edge(a)(dab2)
    \Edge(dab1)(dab15)
    \Edge(dab15)(dab2)
    \Edge(ab)(dab1)
    \Edge(ab)(dab15)
    \Edge(ab)(dab2)
    \Edge(ab)(dab3)
    \Edge(ab)(dab35)
    \Edge(ab)(dab4)
    \Edge(dab3)(dab35)
    \Edge(dab35)(dab4)
    \Edge(b)(dab3)
    \Edge(b)(dab35)
    \Edge(b)(dab4)

    \Edge(b)(bb1)
    \Edge(b)(bb15)
    \Edge(b)(bb2)
    \Edge(bb1)(bb15)
    \Edge(bb15)(bb2)
    \Edge(b3)(bb2)
    \Edge(b3)(bb15)
    \Edge(b3)(bb1)

    \Edge(b)(bb5)
    \Edge(b)(bb55)
    \Edge(b)(bb6)
    \Edge(bb5)(bb55)
    \Edge(bb55)(bb6)
    \Edge(b4)(bb6)
    \Edge(b4)(bb55)
    \Edge(b4)(bb5)

    \Edge(b4)(bb7)
    \Edge(b4)(bb75)
    \Edge(b4)(bb8)
    \Edge(bb7)(bb75)
    \Edge(bb75)(bb8)
    \Edge(b9)(bb7)
    \Edge(b9)(bb75)
    \Edge(b9)(bb8)

    \Edge(dab5)(dab55)
    \Edge(dab55)(dab6)
    \Edge(b9)(dab5)
    \Edge(b9)(dab55)
    \Edge(b9)(dab6)
    \Edge(b10)(dab5)
    \Edge(b10)(dab55)
    \Edge(b10)(dab6)

    \Edge(b9)(aa7)
    \Edge(b9)(aa75)
    \Edge(b9)(aa8)
    \Edge(aa7)(aa75)
    \Edge(aa75)(aa8)
    \Edge(b11)(aa7)
    \Edge(b11)(aa75)
    \Edge(b11)(aa8)

    \tikzset{EdgeStyle/.append style = {bend left}}
    \Edge(aa5)(aa6)
    \Edge(bb1)(bb2)
    \Edge(dab1)(dab2)
    \Edge(bb7)(bb8)
    \Edge(dab5)(dab6)
    \Edge(aa7)(aa8)
    \tikzset{EdgeStyle/.append style = {bend right}}
    \Edge(aa1)(aa2)
    \Edge(dab3)(dab4)
    \Edge(bb5)(bb6)
  \end{tikzpicture}
  \caption{A hub gadget for $k=4$ and $\ell=5$.}
  \label{outlets-fig}
\end{figure}

\begin{proof}[Proof of \cref{mlcnpc3}]
  Let $G$ be an instance of \textsc{$k$-colourability}.
  We may assume that $G$ is $2$-connected.
  For each $v \in V(G)$ such that $d(v) \geq k+1$, we delete $v$, introduce a (disjoint) hub gadget with outlets $p_1,p_2,\dotsc,p_{d(v)}$, and then, for each neighbour $n_i$ of $v$ in $G$, for $i \in \{1,2,\dotsc,d(v)\}$, add a single edge between $p_i$ and $n_i$.
  Thus each outlet has degree~$k$ in the resulting graph $G'$.

  It is clear that $G'$ is $2$-connected.
  Now we show that $G'$ has maximum local connectivity at most~$k$.
  Recall that $\kappa(u,v)$ denotes the maximum number of internally disjoint $(u,v)$-paths.
  Clearly $\kappa(x,y) \le k$ if $d(x) \leq k$ or $d(y) \leq k$.
  Suppose $d(x),d(y) \geq k+1$.
  Then $x$ and $y$ belong to a hub gadget and are not outlets.
  So $x$ belongs to either two or three $K_{k+1}^-$s, each with an end distinct from $x$. Let $P$ be the set of these end vertices. When $y \notin P$, an $xy$-path must pass through some $p \in P$, so $\kappa(x,y) \le 3 \le k$, as required.  Otherwise, $x$ and $y$ are ends of a copy of $K_{k+1}^-$, which we label $J_1$, and neither end of $J_1$ is an outlet.  There are $k-1$ internally vertex disjoint $xy$-paths in $G'[V(J_1)]$.  But one of the ends of $J_1$ is in precisely one other copy of $K_{k+1}^-$, $J_2$ say, and all other $xy$-paths in $G'$ must pass through the end of $J_2$ distinct from $x$ and $y$.
  So $\kappa(x,y) \le k$, as required.

  Suppose $G$ is $k$-colourable and let $\phi$ be a $k$-colouring of $G$.  
  We show that $G'$ is $k$-colourable.  
  Start by colouring each vertex $v$ in $V(G)\cap V(G')$ the colour $\phi(v)$. 
  For each hub gadget $H$ of 
  $h$ in $G$,
  colour every end of a $K_{k+1}^-$ in $H$ the colour $\phi(h)$.  Clearly, each outlet is given a different colour to its neighbours in $V(G)$ since $\phi$ is a $k$-colouring of $G$.
  The remaining $k-1$ vertices of each $K_{k+1}^-$ contained in $H$ have two neighbours the same colour $\phi(h)$, so can be coloured using the remaining $k-1$ colours.  Thus $G'$ is $k$-colourable.

  Now suppose that $G'$ is $k$-colourable. Each end of a $K_{k+1}^-$ must have the same colour in a $k$-colouring of $G'$, so all outlets of a hub gadget have the same colour.  Let $H$ be the hub gadget of $h$ in $G$, where $h \in V(G)$.  We colour $h$ with the colour of all the outlets of $H$ in the $k$-colouring of $G'$.  For each vertex $v \in V(G) \cap V(G')$, we colour $v$ with the same colour as in the $k$-colouring of $G'$, thus obtaining a $k$-colouring of $G$.
\end{proof}

\end{document}